\documentclass{dis}

\usepackage{amsmath, amssymb, amsthm}
\usepackage{dsfont}  
\usepackage{stmaryrd}  
\usepackage{graphicx}
\usepackage{enumitem} 
\usepackage{mathtools}
\usepackage{extarrows}

\usepackage{pstricks,pst-math,pst-xkey}
\usepackage{pstricks-add}
\usepackage{pst-plot}
\usepackage{makeidx}

\usepackage{color}

\newtheorem{theorem}{Theorem}[chapter]
\newtheorem{corollary}[theorem]{Corollary}
\newtheorem{lemma}[theorem]{Lemma}

\newtheorem{remark}[theorem]{Remark}

\newtheoremstyle{citing}{}{}{\itshape}{}{\bfseries}{.}{ }{\thmnote{#3}}
\theoremstyle{citing}
\newtheorem{cit}{}

\newcommand{\Z}{\ensuremath{\mathbb{Z}}}
\newcommand{\N}{\ensuremath{\mathbb{N}}}
\newcommand{\R}{\ensuremath{\mathbb{R}}}
\renewcommand{\a}{\alpha}
\newcommand{\g}{\gamma}
\newcommand{\G}{\Gamma}
\newcommand{\eps}{\varepsilon}
\newcommand{\D}{\Delta}
\renewcommand{\O}{\Omega}

\renewcommand{\P}{\mathbb{P}}
\newcommand{\Var}{{\rm Var}}
\newcommand{\EE}{\mathcal{E}}
\newcommand{\FF}{\mathcal{F}}
\renewcommand{\l}{\langle}
\renewcommand{\r}{\rangle}
\newcommand{\ind}{\scalebox{1.2}{\raisebox{-0.2mm}{$\mathds{1}$}}}

\newcommand{\Op}{\O_{\rm P}}
\newcommand{\Orc}{\O_{\rm RC}}
\newcommand{\Oj}{\O_{\rm J}}
\def\jointname{FKES}
\def\jointnamelong{Fortuin-Kasteleyn-Edwards-Sokal}
\newcommand{\Zt}{\raisebox{2pt}{ $\tilde{}$ }\mkern-12mu \Z}
\newcommand{\hb}{P_{\rm HB}}
\newcommand{\sw}{P_{\rm SW}}
\newcommand{\Pt}{\raisebox{2pt}{ $\tilde{}$ }\mkern-14mu P}
\newcommand{\tsw}{\Pt_{\rm SW}}
\newcommand{\thb}{\Pt_{\rm HB}}
\renewcommand{\sb}{\Pt_{\rm SB}}
\newcommand{\C}{\mathcal{C}}
\newcommand{\T}{\mathcal{T}}
\newcommand{\csw}{c_{\text{\rm\tiny SW}}}
\newcommand{\tcsw}{\widetilde{c}_{\text{\rm\tiny SW}}}

\newcommand{\Gt}{\,\raisebox{2pt}{ $\tilde{}$ }\mkern-15mu G}
\newcommand{\Vt}{\,\raisebox{2pt}{ $\tilde{}$ }\mkern-13mu V}
\newcommand{\Et}{\,\raisebox{2pt}{ $\tilde{}$ }\mkern-15mu E}
\newcommand{\Ztt}{\,\raisebox{2pt}{ $\tilde{}$ }\mkern-15mu Z}

\renewcommand{\S}{\ensuremath{\mathbb{S}}}
\newcommand{\genus}{\ensuremath{\mathbf{g}}}

\newcommand{\conn}[1]{\mathop{\stackrel{#1}{\leftrightarrow}}}
\newcommand{\xconn}[1]{\xleftrightarrow{#1}}
\newcommand{\xnconn}[1]{\longarrownot\xleftrightarrow{#1}}
\newcommand{\lconn}[1]{\mathop{\stackrel{#1}{\longleftrightarrow}}}
\newcommand{\nconn}[1]{\mathop{\hspace{1.3pt}\stackrel{#1}{\longarrownot\longleftrightarrow}}\hspace{1.3pt}}

\newcommand{\tvd}[1]{\left\Vert #1 \right\Vert_{\rm TV}}	
\newcommand{\abs}[1]{\left\vert #1 \right\vert}	
\newcommand{\norm}[1]{\left\Vert #1 \right\Vert}	
\renewcommand{\tilde}[1]{\widetilde{#1}}

\makeindex

\begin{document}


\mathclass{Primary 60K35; Secondary 60J10, 82C20.}
\thanks{
	First and foremost, I would like to thank the advisor of my PhD thesis
	(the present paper is based on it), Erich Novak, 
	for his constant support and the many discussions we had, 
	not only about mathematics. 
	This has certainly led to the development of my present interest in and 
	fun with mathematics.

	I am also grateful to Aicke Hinrichs and Daniel Rudolf for 
	helpful comments and suggestions during the work on my 
	thesis, as well as for the valuable discussions we had about various 
	topics.

	Finally, I want to express my thanks to my colleagues and friends 
	Erich, Aicke, Daniel, 
	Henning Kempka, Lev Markhasin, Philipp Rudolph, 
	Winfried Sickel, Markus Weimar and Heidi Weyhausen 
	who were (more or less regular) members of our daily coffee break 
	that served with its friendly and relaxed atmosphere as a source of 
	extra energy every day.
}

\abbrevauthors{Mario Ullrich}
\abbrevtitle{Rapid mixing of Swendsen-Wang in two dimensions}

\title{Rapid mixing of Swendsen-Wang dynamics in two dimensions}

\author{Mario Ullrich}
\address{Mathematical Institute\\ Friedrich Schiller University Jena\\
Ernst-Abbe-Platz 2\\ 07743 Jena, Germany\\
E-mail: ullrich.mario@gmail.com}

\maketitledis

\tableofcontents

\begin{abstract}
We prove comparison results for the Swendsen--Wang (SW) dynamics, 
the heat-bath (HB) dynamics for the Potts model and the single-bond 
(SB) dynamics for the random-cluster model on arbitrary graphs.
In particular, we prove that rapid (i.e.~polynomial) mixing of HB 
implies rapid mixing 
of SW on graphs with bounded maximum degree 
and that rapid mixing of SW and rapid mixing of SB are equivalent.
Additionally, the spectral gap of SW and SB on planar 
graphs is bounded from above and from below by the spectral gap 
of these dynamics on the corresponding dual graph with suitably 
changed temperature. 

As a consequence we obtain rapid mixing of the 
Swendsen--Wang dynamics for the Potts model on the two-dimensional 
square lattice at all non-critical temperatures as well as 
rapid mixing for the two-dimensional Ising model at all temperatures. 
%
Furthermore, we obtain new results for general graphs 
at high or low enough temperatures.
\end{abstract}

\makeabstract


\chapter{Introduction and results}

We study the mixing properties of Markov chains 
for the \emph{$q$-state Potts model}. 
The Potts model with $q\in\N$ states at 
\emph{inverse temperature} $\beta\ge0$ consists 
of the set $\Op=\{1,\dots,q\}^V$ of (not necessarily proper) 
colorings of a graph 
$G=(V,E)$ together with the probability measure 
\[
\pi(\sigma) \;=\; \frac{1}{Z}\, e^{\beta \abs{E(\sigma)}},
	\qquad \sigma\in\O_{\rm P},
\] 
where $E(\sigma)$ denotes the 
set of edges of the graph $G$ with equally colored endvertices 
and $Z$ is the normalization constant.
This model, especially in the case $q=2$ where it is called the 
\emph{Ising model}, attracted much interest over the last decades 
not only in statistical physics but also in several branches 
of mathematics and computer science.
The goal is to sample from $\pi$, at least approximately, and 
since exact sampling is in general not feasible, 
Markov chains are typically used. 
There are a couple of (more or less efficient) Markov chains 
to attack this problem and, usually, 
it is an easy task to show that their distributions converge 
in the long-time limit towards the right distribution. 
However, it is much more delicate to prove bounds on the 
\emph{mixing time}, i.e.~the number of steps
a Markov chain has to run in order that its distribution is 
``close'' to its limit distribution. 
In the following, a chain is said to be \emph{rapidly mixing} 
for a family of graphs if, for each graph of this family, 
the Markov chain can be defined analogously on it 
and its mixing time is bounded by a polynomial in the 
number of vertices of the graph.

We will consider two Markov chains for the Potts model, 
namely the \emph{heat-bath $($HB\/$)$ dynamics} and the 
\emph{Swendsen--Wang $($SW\/$)$ dynamics}. 
The heat-bath dynamics is the most common chain for 
this purpose. It is a local Markov chain that can be 
described as follows. Suppose that the current configuration 
is $\sigma=\bigl(\sigma(v)\bigr)_{v\in V}\in\Op$. 
In each step, a vertex $v\in V$ of the underlying 
graph is chosen uniformly at random and a new color 
is assigned to $v$ with respect to the conditional 
probability given that the color of all other vertices is fixed, 
so that the new configuration $\tau$ satisfies 
$\tau(u)=\sigma(u)$ for all $u\neq v$. 
The heat-bath dynamics is proven to be rapidly mixing in 
several instances (see Section~\ref{sec:2_known-results}), 
but is typically slowly (i.e.~not rapidly) mixing 
if the inverse temperature $\beta$ is large. 
For example, if the underlying graph is the two-dimensional 
square lattice $\Z_L^2$ of side length $L$, it is known that 
there exists a \emph{critical inverse temperature} $\beta_c(q)$ 
such that the heat-bath dynamics is rapidly mixing if 
$\beta<\beta_c(q)$ and slowly mixing if $\beta>\beta_c(q)$ 
(the latter seems to be proven only for $q=2$ but it is at least 
expected to be true for all $q\ge2$). 
See Section~\ref{sec:2_known-results} for more known 
results and the specific bounds.

The second Markov chain under consideration is the 
Swendsen--Wang dynamics that changes the color of a large portion 
of the vertices in each step. 
One step of this chain, given the current state $\sigma\in\Op$, 
can be described by the following two-step procedure. 
First, generate a subset $A\subset E(\sigma)$ of the edges of 
the graph with equally colored endvertices, 
such that every edge of $E(\sigma)$ 
is included in $A$ with probability $1-e^{-\beta}$. 
In the second step, assign independently and uniformly at random 
a color from $\{1,\dots,q\}$ to each connected component 
of the subgraph $(V,A)$. This gives a new Potts configuration. 
It is widely believed that this Markov chain is rapidly mixing 
in much more cases than the heat-bath dynamics, 
e.g.~at low temperatures (large $\beta$), 
and is therefore the preferred algorithm in practice 
since its invention around 1987.
Nevertheless, results in this direction are rare. 
Besides some rapid mixing results concerning special classes of 
graphs like trees, cycles and the complete graph, 
or results for sufficiently high or low temperatures, 
there is no result that shows that it is generally advisable 
to use Swendsen--Wang instead of heat-bath dynamics.

Our first main result (see Chapter~\ref{chap:spin}) 
shows that rapid mixing of heat-bath dynamics 
implies rapid mixing of Swendsen--Wang, 
if the underlying graph has bounded maximum degree, 
which partially confirms the above-mentioned intuition.

In order to give precise statements  
we need some notation.
We define the spectral gap of a Markov chain with 
transition matrix $P$ as 1 minus the second largest eigenvalue 
of $P$ in absolute value and we denote it by $\lambda(P)$. 
It is well-known that polynomial (in~$\abs{V}$) bounds on the 
mixing time are equivalent to polynomial bounds on the inverse 
spectral gap (see Lemma~\ref{lemma:mixing-gap}), 
and since the spectral gap seems to be more convenient for 
comparison results, this is the quantity we are interested in. 
(Note that we classify Markov chains by their transition 
matrices, since the quantities of interest do not depend on 
anything else.)

For the following comparison results we consider the two 
above-mentioned Markov chains for the Potts model, 
as defined in Section~\ref{sec:2_dynamics}, 
and we use this notation 
for their transition matrices:
\begin{itemize}
	\item $\hb$ for the heat-bath dynamics for the Potts model 
					(see~\eqref{eq:P-HB}), \vspace{2mm}
	\item $\sw$ for the Swendsen--Wang dynamics for the Potts model 
					(see~\eqref{eq:SW-P}).
\end{itemize} 
The first result that we want to present here is a comparison 
between heat-bath and Swendsen--Wang dynamics for the Potts model. 

\vspace{2mm}

\begin{cit}[Theorem~\ref{th:main-spin}] 
Suppose that $\sw$ $($resp.~$\hb$$)$ is the transition matrix of the 
Swendsen--Wang $($resp.~heat-bath$)$ dynamics for the $q$-state Potts 
model at inverse temperature $\beta$ on a graph $G$ with 
maximum degree $\D$. 
Then
\[
\lambda(\sw) \;\ge\; \csw\,\lambda(\hb),
\]
where
\begin{equation}
\csw \;:=\; \csw(\D,\beta,q) 
\;=\; q^{-1}\left(q\,e^{2\beta}\right)^{-2\D}.
\end{equation}
\vspace{-3mm}
\end{cit}

This result implies rapid mixing of the Swendsen-Wang dynamics in 
some new instances. But, since the heat-bath dynamics is typically 
slow at low temperatures, it is not helpful for large values of 
$\beta$. 
For this reason we switch to dynamics for a closely related 
model on the edges of the underlying graph, i.e.~the 
\emph{random-cluster $($RC\/$)$ model}, 
where it was possible to deduce some lower bounds on the spectral gap 
at low temperatures 
from lower bounds at high 
temperatures. 
The random-cluster model consists of the set 
of subsets of the edges $\Orc=\{A:A\subset E\}$ and the 
probability measure 
\[
\mu(A) \;=\; \frac1{Z}\,
\left(\frac{p}{1-p}\right)^{\abs{A}}\,q^{c(A)},\qquad A\subset E, 
\]
where $c(A)$ denotes the number of connected components of the 
subgraph $(V,A)$ and $p:=1-e^{-\beta}$.
In fact, the Swendsen--Wang dynamics is based on the tight 
connection of Potts and random-cluster models. 
That is, there exists a clever coupling of $\pi$ and $\mu$, 
say $\nu$, such that the conditional probabilities, 
given either a Potts or a random-cluster configuration, 
are equal to the 
probability distributions that are involved in the first 
and the second step of the Swendsen--Wang dynamics, 
respectively (see Section~\ref{sec:2_models}). 

From this construction it is obvious that the Swendsen--Wang 
dynamics (with its two steps in reverse order) also defines a 
Markov chain for the random-cluster model. 
Additionally, we define the \emph{single-bond $($SB$)$ dynamics} 
which is a local Markov chain for the random-cluster model that 
chooses an edge $e\in E$ of the graph uniformly at random and 
includes or deletes $e$ from the current random-cluster 
configuration with a certain probability (see~\eqref{eq:SB}). 
The transition matrices of these two dynamics are denoted as follows:
\begin{itemize}
	\item $\tsw$ for the Swendsen--Wang dynamics for the 
					RC model (see~\eqref{eq:SW-RC}), \vspace{2mm}
	\item $\sb$ for the single-bond dynamics for the 
					RC model (see~\eqref{eq:SB}). 
\end{itemize}

As an easy corollary of the construction of the 
Swendsen--Wang dynamics we obtain $\lambda(\sw)=\lambda(\tsw)$ 
(see Lemma~\ref{lemma:SW_P-RC}). 
Thus, every result on $\tsw$ immediately yields a result on $\sw$.

It turns out that Swendsen--Wang and single-bond dynamics can be 
represented on the joint Potts/random-cluster model 
(i.e.~the model corresponding to the coupling $\nu$) using the same 
``building blocks'', which leads to the second main result of 
this paper. We prove that rapid mixing of 
Swendsen--Wang dynamics is equivalent to rapid mixing of single-bond dynamics.

\begin{cit}[Theorem~\ref{th:main-SB}]
Let $\tsw$ $($resp.~$\sb$$)$ be the transition matrix of the 
Swendsen--Wang $($resp.~single-bond$)$ dynamics for the 
random-cluster model on a graph with $m\!\ge\!3$ edges. Then 
\[
\lambda(\sb) \;\le\; \lambda(\tsw) \;\le\; 8 m \log m\; \lambda(\sb).
\]
\vspace{-1mm}
\end{cit}

Finally, in Chapter~\ref{chap:2d} we restrict ourself to a special 
class of graphs, namely \emph{planar graphs}, and, 
using the notion of \emph{dual graphs}, it will be possible 
to relate the mixing properties of the Swendsen--Wang dynamics 
on the original graph to the mixing properties on its dual with a 
suitable change of the temperature parameter. 
This is done by proving such a result for the single-bond dynamics 
and, using the comparison result above, translating it to the 
Swendsen--Wang dynamics. 

For this, assume that the underlying graph is planar, 
i.e.~can be drawn in the plane without intersecting edges.
Furthermore, given a random-cluster model on a planar graph $G$ 
with parameters $p$ and $q$, we call the random-cluster model 
on the dual graph $G^\dag$ with parameters 
$p^*=\frac{q(1-p)}{p+q(1-p)}$ and $q$ the \emph{dual model}.

\begin{cit}[Theorem~\ref{th:SW-SB_dual}]
Let $\tsw$ $($resp.~$\sb$$)$ be the transition matrix of the 
Swendsen--Wang $($resp.~single-bond$)$ dynamics for the 
random-cluster model on a planar graph $G$ with $m$~edges 
and let $\tsw^\dag$ $($resp.~$\sb^\dag$$)$ be the SW $($resp.~SB$)$ 
dynamics for the dual model.
Then 
\begin{align*}
\lambda(\sb) \;&\le\; q \; \lambda(\sb^\dag) \\[-5mm]
\intertext{\vspace*{-3mm} and}
\lambda(\tsw) \;&\le\; 8q\, m \log m\; \lambda(\tsw^\dag).
\end{align*}
\vspace{-2mm}
\end{cit}

Since the dual model of the dual model is the primal 
random-cluster model (if $G$ is connected), 
we get the bound in the other direction by applying 
Theorem~\ref{th:SW-SB_dual} twice. 
This theorem allows us to obtain some new results on 
rapid mixing at low temperatures directly 
from known results at high temperatures. 

Now we turn to applications of the theorems above. 
But first, note that we do not give a direct analysis of the 
Markov chains under consideration. 
Thus, all results on rapid mixing in specific 
settings rely ultimately on already known mixing results.  
In fact, the three results given below are based on 
lower bounds on the spectral gap of the heat-bath dynamics 
for the Potts model. 
We refer to Section~\ref{sec:2_known-results} for a 
collection of all previously known results that are 
necessary for the analysis in this paper.

The first application, which was the main reason 
for our study, deals with the two-dimensional 
square lattice $\Z_L^2$ of side length $L$. 
This is the graph $\Z^2_L=(V_{L,2},E_{L,2})$ 
with vertex set $V_{L,2}=\{1,\dots,L\}^2\subset\Z^2$ and edge set 
$E_{L,2}=\bigl\{\{u,v\}\subset V_{L,2}:\,\abs{u-v}=1\bigr\}$, 
where $\abs{\,\cdot\,}$ denotes the Euclidean norm.

\begin{cit}[Theorem~\ref{th:SW_square2}]
Let $\sw$ be the transition matrix of the 
Swendsen--Wang dynamics for the $q$-state Potts model 
on $\Z^2_L$ at inverse temperature $\beta$. Let $n=L^2$. 
Then there exist constants $c_\beta=c_\beta(q),c'>0$ and $C<\infty$ 
such that 
\begin{itemize}
\item\quad $\displaystyle \lambda(\sw) \;\ge\; \frac{c_\beta}{n}$ 
			\qquad\qquad\quad for $\beta < \beta_c(q)$, \vspace{2mm}
\item\quad $\displaystyle \lambda(\sw) \;\ge\; \frac{c_\beta}{n^2 \log n}$ 
			\qquad\; for $\beta > \beta_c(q)$, \vspace{2mm}
\item\quad $\displaystyle \lambda(\sw) \;\ge\; c' n^{-C}$ 
			\qquad\quad\, for $q=2$ and $\beta = \beta_c(2)$,
\end{itemize}
where $\beta_c(q)\,=\,\log(1+\sqrt{q})$.
\vspace{3mm}
\end{cit}

This result shows 
rapid mixing of the Swendsen--Wang dynamics for the Potts model 
on the 
two-dimensional square lattice $\Z_L^2$ at \emph{all} non-critical 
temperatures, i.e.~at all $\beta\neq\beta_c(q)$, as well as 
rapid mixing at all temperatures in the case $q=2$. 
Note that it was not even known that the SW dynamics mixes rapidly 
for $\beta<\beta_c$, as it is known for the heat-bath dynamics. 

As a byproduct we obtain the following result for the 
single-bond dynamics for the random-cluster model.

\begin{cit}[Theorem~\ref{th:SB_2d}]
Let $\sb$ be the transition matrix of the 
single-bond dynamics for the RC model 
on $\Z^2_L$ with parameters $p$ and $q$. Let $m=2L(L-1)=\abs{E_{L,2}}$. 
Then there exist constants $c_p=c_p(q),c'>0$ and $C<\infty$ such that 
\begin{itemize}
\item\quad $\displaystyle \lambda(\sb) \;\ge\; \frac{c_p}{m^2 \log m}$ 
			\qquad for $p \neq p_c(q)$,\vspace{2mm}
\item\quad $\displaystyle \lambda(\sb) \;\ge\; c' m^{-C}$ 
			\qquad\quad for $q=2$ and $p = p_c(2)$,
\end{itemize}
where $p_c(q)\,=\,\frac{\sqrt{q}}{1+\sqrt{q}}$.
\vspace{3mm}
\end{cit}

For the third application we consider 
the more general family of 
planar graphs of bounded maximum degree.

\begin{cit}[Corollary~\ref{coro:SW-planar}]
The Swendsen--Wang dynamics for the random-cluster model
with parameters $p$ and $q$
on a planar, simple and connected graph $G$ with $m$ edges 
and maximum degree $\D\ge6$ satisfies
\[
\lambda(\tsw)\;\ge\;\frac{c(1-\eps)}{m}, \qquad \text{ if }\; 
p \,\le\, \frac{\eps}{3\sqrt{\D-3}},
\vspace{-3mm}
\] 
and 
\[
\lambda(\tsw)\;\ge\;\frac{c(1-\eps)}{m^2 \log m}, \qquad 
\text{ if }\; 
p \,\ge\, 1- \frac{\eps}{q \D^\dag}, \;\;
\] 
for some $c=c(\D,p,q)>0$ and $\eps>0$, 
where $\D^\dag$ is the maximum degree of a dual graph of~$G$.
\vspace{1mm}
\end{cit}

Additionally, we present a similar result for arbitrary 
graphs of bounded maximum degree. 
In this case the Swendsen--Wang dynamics is proven to be rapidly 
mixing if $p\le\eps/\D$ (see Corollary~\ref{coro:SW_degree} 
and Lemma~\ref{lemma:SW_P-RC}).
These results enlarge the previously known set of temperatures 
where rapid mixing is known (see \cite{Hu} or 
Theorem~\ref{known-th:SW_degree}), but lead to a worse bound on 
the spectral gap.


\chapter{Detailed introduction}

In this chapter we provide the necessary definitions and notations. 
The experienced reader could 
skip this chapter and visit it when necessary.

First we give an elementary introduction to 
Markov chains on finite state spaces. 
See Levin, Peres and Wilmer \cite{LPW} for more details.
Then we define the measure of efficiency of 
Markov chains which we are concerned with, the spectral gap, 
and explain the relation of this quantity to the 
\emph{rate of convergence} of a Markov chain to its stationary 
distribution, as well as the relation to another quantity, 
the mixing time.
In the subsequent sections we focus on 
Markov chains for Potts and random-cluster 
models, and give precise definitions of the models and the  
Markov chains used.
Finally, in Section~\ref{sec:2_known-results}, 
we state several known results on their 
mixing properties.


\section{Markov chains} \label{sec:2_markov}

Let $\O$ be a finite set, $P$ be an $\O\times\O$-matrix with 
$P(x,y)\ge0$ and $\sum_{z\in\O} P(x,z)=1$ for all $x,y\in\O$, 
and $X=(X_t)_{t\in\N}$, 
$X_t\in\O$, be a sequence of $\O$-valued random variables
defined on a probability space $(\mathfrak{X},\mathcal{F},\P)$. 
We call $X$ a 
\emph{Markov chain on $\O$ with transition matrix $P$}
\index{Markov chain} 
if, for all $t\ge1$ and all $x_0,\dots,x_{t}\in\O$ such 
that $\P(X_0=x_0,\dots,X_{t-1}=x_{t-1})>0$, we have 
\[\begin{split}
\P\left(X_t=x_t \bigm| X_0=x_0,\dots, X_{t-1}=x_{t-1}\right) 
\,&=\, \P\left(X_t=x_t \bigm| X_{t-1}=x_{t-1}\right) \\
&=\, P(x_{t-1},x_t).
\end{split}\]
This is called the {\it Markov property}. 
Note that this implies that, for all $x,y\in\O$ and 
$s\in\N$ with $\P(X_s=x)>0$,
\[
P^t(x,y) \,=\, \sum_{z\in\O} P^{t-1}(x,z) P(z,y) 
	\,=\, \P\left(X_{s+t}=y \bigm| X_s=x\right) 
	\qquad\text{ for all } t\ge0,
\]
which defines the \emph{$t$-step transition probabilities}.
It is clear that the Markov property implies that the knowledge 
of $P$ and of the distribution of 
$X=(X_t)_{t\in\N}$ at any time, say~0, determines the distribution 
of the Markov chain at all future times:
Suppose that the distribution of the Markov chain at time $0$ is 
given by some probability mass function $\eta_0$, 
i.e. $\eta_0(x)=\P(X_0=x)$ for all $x\in\O$; then the 
distribution of the Markov chain at time $t\ge1$ satisfies
\begin{equation} \label{eq:P-X_t}
\P(X_t=x) \,=\, \sum_{z\in\O} \eta_0(z)\, P^t(z,x).
\end{equation}
Therefore, we can simulate one realization of the Markov chain $X$ 
on $\O$ with transition matrix $P$ and \emph{initial distribution}
\index{distribution!initial} 
$\eta_0$ 
by first generating a state $x_0$ with respect to the distribution 
$\eta_0$, and then generating successively the state $x_t$ with 
respect to $P(x_{t-1},\cdot)$, $t\ge1$.
In particular, if the initial distribution of the Markov 
chain is concentrated at a single state, i.e. 
$\eta_0(x)=\ind(x=x_0)$ for some $x_0\in\O$, the distributions 
of the Markov chain take the simple form 
$\P(X_t=x) = P^t(x_0,x)$.
Here, the 
\emph{indicator function}\index{indicator function ($\ind$)} 
$\ind$ is defined 
by $\ind(x=x_0)=1$ if $x=x_0$ and $\ind(x=x_0)=0$ otherwise. 
In general, $\ind$ equals 1 if the statement in the parentheses is true 
and equals 0 otherwise.
From now on we will identify probability mass functions with 
row-vectors, so that equality \eqref{eq:P-X_t} simplifies to 
$\P(X_t=x) = \eta_0 P^t(x)$, and 
for a probability mass function $\pi$ we define, for $A\subset\O$, 
the probability measure $\pi$ by
\[
\pi(A) \,:=\, \sum_{x\in A} \pi(x).
\]
(By convention, we will use $\pi$ interchangeably as a measure and a 
mass function.) \\
Markov chains are typically used to sample 
(approximately) from distributions for which direct simulation 
is not feasible. For this it is necessary that, at least in the 
limit, the distribution of the Markov chain reaches this 
distribution, say $\pi$, i.e.
\vspace{1mm}
\begin{equation} \label{eq:P-limit}
\pi(x)\,=\,\lim_{t\to\infty} \P(X_t=x)\qquad \text{ for all } x\in\O.
\end{equation}

\noindent It is in general not guaranteed that such limits exist, but we 
can present a sufficient condition for 
existence. 
We call a Markov chain 
\emph{irreducible} 
if for 
all $x,y\in\O$, there exists $t$ such that $P^t(x,y)>0$, and 
\emph{aperiodic} 
if 
for all $x\in\O$, ${\rm gcd}\left\{t\ge1:\,P^t(x,x)>0\right\}=1$, 
where ${\rm gcd}\{J\}$ denotes the greatest common divisor of all 
elements in $J\subset\N$. 
Clearly, for an irreducible Markov chain, $P(x,x)>0$ for any 
$x\in\O$ is a sufficient condition for aperiodicity.
All Markov chains in this paper will be irreducible and aperiodic, 
and we use \emph{ergodic}\index{ergodic} 
as an abbreviation for irreducible and aperiodic. 
It is well-known that the above limit \eqref{eq:P-limit} exists for 
aperiodic Markov chains $X$ and, 
if the Markov chain is also irreducible, 
the limit distribution $\pi$ is independent of the initial 
distribution $\eta_0$ and satisfies $\pi(x)>0$, $\forall x\in\O$.
Thus, $\pi P=\pi$, which follows obviously from \eqref{eq:P-limit}, 
is equivalent to \eqref{eq:P-limit} whenever $X$ is an ergodic 
Markov chain with transition matrix $P$.
We call a distribution $\pi$ with $\pi P=\pi$ a 
\emph{stationary distribution}\index{distribution!stationary} 
of the Markov chain with transition matrix~$P$. 

Another important concept is the reversibility of Markov chains. 
A Markov chain with transition matrix $P$ is 
\emph{reversible}\index{reversible} 
(or satisfies {\it detailed balance}) 
\emph{with respect to $\pi$} if for all $x,y\in\O$,
\[
\pi(x) P(x,y) \,=\, \pi(y) P(y,x).
\]
It is not hard to prove that, under this condition, $\pi$ is 
a stationary distribution of~$P$.
Finally, we call a Markov chain \emph{lazy}\index{lazy}, if its transition 
matrix $P$ satisfies $P(x,x)\ge\frac12$, $\forall x\in\O$. 
Obviously, lazy Markov chains are aperiodic. 
If $X$ is a Markov chain with transition matrix $P$, 
and $\tilde{X}$ is a Markov chain with transition matrix $Q$, 
such that $Q(x,y)=\frac12 P(x,y)$ for all $x\neq y$, 
$x,y\in\O$, and $\P(\tilde{X}_0=x)=\P(X_0=x)$, $\forall x\in\O$, 
then we say that $\tilde{X}$ is the \emph{lazy version of} $X$.
Necessarily, $Q(x,x)=\frac12+\frac12 P(x,x)$ for all $x\in\O$.

Throughout this paper we will refer to properties of a Markov chain 
with transition matrix $P$ as properties of $P$, unless they depend 
on more than $P$. For example we say for an ergodic Markov 
chain with transition matrix $P$ that is reversible with respect 
to~$\pi$, that $P$ is ergodic and reversible with 
respect to~$\pi$.


\section{Spectral gap and mixing time} \label{sec:2_gap}

In the following we want to estimate the efficiency of  
Markov chains for approximate sampling from their stationary 
distribution. 
For more details and aspects of the convergence 
of Markov chains to their stationary distribution see Levin, Peres 
and Wilmer~\cite{LPW}.

To quantify this efficiency we first define 
the \emph{total variation distance}\index{total variation distance} 
of two distributions $\nu$ and $\pi$ on $\O$ by
\[
\tvd{\nu-\pi} \,:=\, \frac12\sum_{x\in\O}\abs{\nu(x)-\pi(x)} 
	\,=\, \max_{A\subset\O} \abs{\nu(A)-\pi(A)}.
\]
See \cite[Prop.~4.2]{LPW} for the second equality.
Using this as a metric on the set of all probability measures 
on $\O$ it is natural to ask how fast the convergence 
in \eqref{eq:P-limit} takes place, i.e. how fast the distribution 
of a Markov chain for increasing time $t$ converges to its 
stationary distribution.
The next statement, which is called the \emph{Convergence Theorem}, 
provides a more quantitative version of \eqref{eq:P-limit} 
(see e.g. \cite[Theorem~4.9]{LPW}).

\begin{theorem}\label{th:convergence}
Let $P$ be ergodic with stationary distribution $\pi$. 
Then there exist constants $\a\in[0,1)$ and $C>0$ such that
\begin{equation} \label{eq:convergence}
\max_{x\in\O} \tvd{P^t(x,\cdot)-\pi} \,\le\, C \a^t.
\end{equation}
\end{theorem}

This theorem shows that, for every $x\in\O$, the distribution of 
a Markov chain with transition matrix $P$ and initial distribution 
concentrated at $x$ converges exponentially fast (in total variation) 
to its stationary distribution.
In fact, 
the worst initial distribution, 
i.e. the distribution $\eta_0$ that maximizes 
$\tvd{\eta_0 P^t(\cdot)-\pi}$, is concentrated at a single state. 
To see this, note that for a distribution $\eta_0$ on $\O$,
\[\begin{split}
\tvd{\eta_0 P^t(\cdot)-\pi} 
\,&=\, \tvd{\sum_{x\in\O}\eta_0(x) P^t(x,\cdot)-\pi}
\,=\, \tvd{\sum_{x\in\O}\eta_0(x) \left(P^t(x,\cdot)-\pi\right)}\\
&\le\,\sum_{x\in\O}\eta_0(x)\tvd{P^t(x,\cdot)-\pi}
\,\le\,\max_{x\in\O}\tvd{P^t(x,\cdot)-\pi}.
\end{split}\]
Thus, it is enough to consider the maximum as in 
Theorem~\ref{th:convergence} to get a statement for arbitrary 
initial distributions.

For given $\eps>0$, we are interested in the minimal time $t$ such 
that the total variation distance between the distribution of 
our Markov chain at time $t$ and its stationary distribution 
is at most $\eps$, independent of the initial distribution. 
This $t$ is called the $\eps$-\emph{mixing time} and is defined by 
\[
t_{\rm mix}(P,\eps) \,:=\, \min\left\{t\ge0:\, 
	\max_{x\in\O} \tvd{P^t(x,\cdot)-\pi}\le\eps\right\}.
\]
Using the fact that $t_{\rm mix}(P,\eps)\le
\lceil\log(\eps^{-1})\rceil\, t_{\rm mix}(P,1/2e)$ 
(see e.g. \cite[p. 55]{LPW}), it is enough to consider the 
\emph{mixing time}\index{mixing time} 
$t_{\rm mix}(P):=t_{\rm mix}(P,1/2e)$.
Unless otherwise stated, $\log$ denotes the natural logarithm.
Obviously, bounds on $t_{\rm mix}$ imply (and follow from) 
bounds on the optimal, i.e. smallest, constants $C$ and $\a$ 
in Theorem~\ref{th:convergence}. 
To be precise, if we know some constants $C$ and $\a$ for which 
\eqref{eq:convergence} holds, we obtain
$t_{\rm mix}(P) \le \log(2e\, C) \cdot\log(1/\a)^{-1}$.
For the reverse direction note that 
\[
\max_{x\in\O} \tvd{P^t(x,\cdot)-\pi} 
\,\le\, e^{-\lfloor t/t_{\rm mix}\rfloor}
\,\le\, e^{1-t/t_{\rm mix}} 
\]
with $t_{\rm mix}=t_{\rm mix}(P)$ 
(see \cite[eq. (4.34)]{LPW}).

In this paper we are concerned with bounding the optimal 
constant $\a$ in \eqref{eq:convergence} and, in particular, 
comparing the optimal constants for different Markov chains. 
We will see (Lemma~\ref{lemma:mixing-gap}) that this also 
implies bounds on the mixing time.

For this we introduce the spectral gap of a Markov chain. 
Let $P$ be the transition matrix of a Markov chain on 
$\Omega$ that is ergodic and reversible with 
respect to $\pi$. 
We regard $P$ as an operator that maps functions 
$f:\O\to\R$ to functions by
\begin{equation} \label{eq:map}
Pf(x) \;:=\; \sum_{y\in\O}\,P(x,y)\,f(y).
\end{equation}
Such an operator is called the 
\emph{Markov operator}\index{Markov operator} 
that corresponds to~$P$ and we will use the same notation 
for the Markov operator and its corresponding transition matrix. 
Note that $Pf(x)$ is the expectation of $f$ with respect to the 
distribution $P(x,\cdot)$, i.e. the distribution of the Markov 
chain with transition matrix $P$ and initial distribution 
concentrated at $x\in\O$ after one step.
If we consider a function $f$ on $\O$ as an element of $\R^\O$ 
(i.e. a column vector), 
then $Pf$ is simply matrix multiplication. 
Additionally, we endow the function space  
$\R^\O$ with the inner product
\begin{equation} \label{eq:inner_product}
\l f, g\r_\pi \,:=\, \sum_{x\in\O} f(x)\, g(x)\, \pi(x), 
	\qquad f,g\in\R^\O,
\end{equation}
and denote by $L_2(\pi)$ the inner product space (or Hilbert space) 
that consists of $\R^\O$ with the inner product 
$\l\cdot,\cdot\r_\pi$. In particular, the norm in $L_2(\pi)$ 
is given by 
\[
\norm{f}_\pi^2 \,:=\, \l f, f\r_\pi
\,=\, \sum_{x\in\O}f(x)^2\,\pi(x), \qquad f\in\R^\O.
\]
Using the inner product \eqref{eq:inner_product} we define 
the \emph{adjoint operator}\index{adjoint operator} 
$P^*$ of $P$ as the (unique) 
operator that satisfies $\l f, Pg\r_{\pi} = \l P^*f, g\r_{\pi}$ 
for all $f,g\in L_2(\pi)$. Then $P^*$ is also a Markov 
operator and the corresponding transition matrix is 
\begin{equation} \label{eq:adjoint}
P^*(x,y) = \frac{\pi(y)}{\pi(x)} P(y,x), \qquad x,y\in\O.
\end{equation}
Since $P$ is reversible with respect to $\pi$, we obtain  
$P=P^*$. Hence, $P$ defines a self-adjoint operator.
This 
implies that 
$P$ has only real eigenvalues, i.e. real numbers $\xi$ with 
$Pf=\xi f$ for some $0\neq f\in\R^\O$  
(see e.g. \cite[Thm.~9.1-1]{Krey}). 
By the ergodicity of $P$, 
these eigenvalues $\{\xi_i\}$ satisfy $-1<\xi_i\le1$; 
additionally, $Pf=f$ if and only if $f$ is constant (see 
\cite[Lemma~12.1]{LPW}).

We define the (absolute) 
\emph{spectral gap}\index{spectral gap} 
of $P$ by 
\[
\lambda(P) \;:=\; 1 - \max\Bigl\{\abs{\xi}:\, \xi 
		\text{ is an eigenvalue of } P,\; \xi\neq1\Bigr\}.
\]
It is well-known that the spectral gap can be written in terms 
of norms of the Markov operator $P$. For this we define the 
spectral norm (or simply \emph{operator norm}\index{operator norm}) 
of an operator $P$ by
\begin{equation} \label{eq:norm}
\Vert P\Vert_\pi \;:=\; \Vert P\Vert_{L_2(\pi)\to L_2(\pi)} 
\;=\; \max_{\Vert f\Vert_\pi=1} \Vert Pf\Vert_\pi.
\end{equation}
We use $\Vert\cdot\Vert_\pi$ interchangeably for functions and 
operators, because it will be clear from the context which 
norm is used. 
For a self-adjoint $P$ the operator norm $\Vert P\Vert_\pi$ equals 
the largest eigenvalue of $P$ in absolute value.
To give a representation of the spectral gap we also need the 
operator $S_\pi$ that is defined by $S_\pi f=\l f, 1\r_\pi$. 
This (Markov) operator corresponds to the (transition) matrix 
$S_\pi(x,y)=\pi(y)$ for $x,y\in\O$. Obviously, $S_\pi$ has only the 
eigenvalues 1 and 0, and the eigenspace to eigenvalue 0 is 
$\{f\in L_2(\pi):\, \l f, 1\r_\pi=0\}$, which is also 
the union of all eigenspaces of $P$ for eigenvalues different 
from~1. Thus,
\vspace{1mm}
\begin{equation} \label{eq:gap-norm}
\lambda(P) \,=\, 1 - \norm{P-S_\pi}_\pi.
\end{equation}

As stated above, the spectral gap and the speed of convergence 
in Theorem~\ref{th:convergence} are closely related.
The next lemma (see \cite[Coro.~12.6]{LPW}) demonstrates this relation.

\begin{lemma}\label{lemma:TV-gap}
Let $P$ be ergodic and reversible with respect to 
$\pi$. Then 
\[
\lim_{t\to\infty}\left(\max_{x\in\O} 
	\tvd{P^t(x,\cdot)-\pi}\right)^{1/t} \,=\, 1-\lambda(P).
\]
\end{lemma}

This shows that $1-\lambda(P)$ is the optimal constant 
$\a$ in \eqref{eq:convergence}.
One may hope that this asymptotic equality holds, at least 
approximately, also non-asymptotically, i.e.~that there exist 
constants $c,C>0$ such that
\[
c \,\bigl(1-\lambda(P)\bigr)^t 
\;\le\; \max_{x\in\O}\tvd{P^t(x,\cdot)-\pi}
\;\le\; C \bigl(1-\lambda(P)\bigr)^t 
\]
for all $t\in\N$, and in fact, this inequality holds with 
constants $c=\frac12$ \cite[(12.13)]{LPW} and 
$C^{-1}=\min_{x\in\O}\pi(x)$ \cite[(12.11)]{LPW}.
Plugging this into the definition of the mixing time we 
deduce that 
mixing time and spectral gap of a Markov chain (on finite state spaces) 
satisfy the following inequality (see e.g. 
\cite[Theorem~12.3 \& 12.4]{LPW}).

\begin{lemma} \label{lemma:mixing-gap}
Let $P$ be the transition matrix of a reversible, ergodic Markov chain with 
state space $\O$ and stationary distribution $\pi$. Then
\[
\lambda(P)^{-1}-1 \;\le\; t_{\rm mix}(P) 
\;\le\; \log\left(\frac{2e}{\pi_{\rm min}}\right)\,\lambda(P)^{-1}, 
\]
where $\pi_{\rm min}:=\min_{x\in\O}\pi(x)$.
\end{lemma}

We are interested in the spectral gap of specific Markov chains and, 
in particular, in the dependence on the size of the state space 
if the Markov chain can be defined analogously on an unbounded 
family of state spaces.
Thus, if we consider an \emph{indexed family} of state spaces
$\{\O_n\}_{n\in\N}$ with a corresponding family of transition 
matrices $\{P_n\}_{n\in\N}$, we say that the Markov chain is 
\emph{rapidly mixing}\index{rapid mixing}
for the given family if 
$\lambda(P_n)^{-1} \le c\log(|\O_n|)^C$
for all $n\in\N$ and some $c,C < \infty$.

\vspace{3mm}
\begin{remark}
The main reason for sampling from 
a given probability distribution $\pi$ on $\O$ is 
probably the approximation of expectations 
$S_\pi f = \l f,1\r_\pi$ for certain functions $f\in\R^\O$. 
This is frequently done e.g.~in statistical shysics to 
deepen the understanding of the underlying model.
If exact sampling from $\pi$ is not feasible, then this often can be 
done by \emph{Markov chain Monte Carlo methods} that can be 
described as follows. Choose an initial state $x_0\in\O$ 
(deterministically or by some distribution), then simulate 
$k$ steps of the Markov chain with transition matrix $P$ 
to obtain a sequence $x_1,\dots,x_{k}\in\O$ and, finally, 
take the average $A_k(f):=\frac{1}{k+1}\sum_{i=0}^{k} f(x_i)$. 
Usually it is better to omit several states at the beginning of the 
sequence to improve the performance 
(see e.g.~Rudolf~\cite{Rud-diss}). The number of these omitted 
states is called \emph{burn-in} in the literature.
We know from the \emph{Ergodic Theorem} \cite[Thm.~4.16]{LPW} 
that $\lim_{k\to\infty}A_k(f)=S_\pi f$ almost surely, whenever 
$P$ is irreducible and has stationary distribution $\pi$. 
In this context the spectral gap plays an important role in 
bounding the number $k$ of steps of the Markov chain that are necessary 
to achieve a prescribed error $\eps>0$. 
For a bound on $k$ for the probabilistic error criterion that 
depends on the spectral gap and the $\eps$-mixing time see 
Levin et~al.~\cite[Sec.~12.6]{LPW}. 
Another bound that deals with the mean square error and 
depends only on the spectral gap is given in \cite{Rud09,Rud10};  
see also Novak and Wo{\'z}niakowski~\cite{NW2} for some 
context and results in a more general setting. 
Overall, one can say that the existence of a rapidly mixing 
Markov chain for a family of state spaces $\{\O_n\}_{n\in\N}$ 
with corresponding measures $\{\pi_n\}_{n\in\N}$ 
leads to an efficient algorithm for the approximation of 
expectations of functions (with respect to $\pi_n$) defined on $\O_n$, 
i.e.~an algorithm that needs time proportional to $\eps^{-2}$ 
times a polynomial in $\log(|\O_n|)$ (times the variance of the 
considered function).
\end{remark}
\vspace{3mm}

To finish this section we present a simple technique to compare 
the spectral gaps of two Markov chains on the same state space, 
but with possibly different stationary distributions. 
This result is well-known (see e.g.~\cite{DSC2} or \cite{LPW}), 
but since it is used several times in this paper
we present its proof here. 
See also Dyer, Goldberg, Jerrum and Martin~\cite{DGJM} 
for a survey on more general techniques for comparison of Markov chains.

\vspace{1mm}
\begin{lemma}\label{lemma:prelim_comparison}
Suppose $P_1$ $($resp.~$P_2$$)$ is an ergodic and 
reversible transition matrix 
with stationary distribution $\pi_1$ $($resp.~$\pi_2$$)$ on $\O$. 
If there exist constants $a,A>0$ such that
\[
\frac{\pi_1(x) P_1(x,y)}{\pi_2(x) P_2(x,y)} \le A 
\qquad\text{ and }\qquad
\frac{\pi_1(x)}{\pi_2(x)} \ge a
\]
for all $x,y\in\O$, then
\[
\lambda(P_1) \,\le\, \frac{A}{a}\, \lambda(P_2).
\]
If $P_2$ has only non-negative eigenvalues 
it is enough to verify the conditions for $x\neq y$.
\end{lemma}
\vspace{1mm}

\begin{proof}
For $f\in\R^\O$ define
\[
\EE_1(f) \,:=\, \l (I-P_1)f, f\r_{\pi_1} \qquad\text{ and }\qquad
\FF_1(f) \,:=\, \l (I+P_1)f, f\r_{\pi_1}
\]
with the identity $If:=f$. 
Equivalently we define $\EE_2$ and $\FF_2$ for $P_2$.
By reversibility and ergodicity, $P_1$ has only real eigenvalues 
\[
1=\xi_1(P_1)>\xi_2(P_1)\ge\dots\ge\xi_{\abs{\O}}(P_1)>-1, 
\] 
i.e. $\lambda(P_1)=\min\{1-\xi_2(P_1), 1+\xi_{\abs{\O}}(P_1)\}$.
Thus, 
using the ``min-max characterization'' of the eigenvalues 
\cite[Thms. 4.2.2 \& 4.2.11]{HJ-matrix} 
and the fact that $\xi_1(P_1)$ corresponds to the constant eigenfunction, 
we obtain
\[
1-\xi_2(P_1) \,=\, 1-\max_{\substack{f\neq0:\\ \l f,1\r_{\pi_1}=0}}
	\frac{\l Pf, f\r_{\pi_1}}{\l f, f\r_{\pi_1}} 
\,=\, \min_{\substack{f\neq0:\\ \l f,1\r_{\pi_1}=0}}
	\frac{\EE_1(f)}{\l f, f\r_{\pi_1}} 
\,=\, \min_{\substack{f\neq0:\\ \l f,1\r_{\pi_1}=0}} 
	\frac{\EE_1(f)}{\Var_{\pi_1}(f)}
\]
with $\Var_{\pi_1}(g)=\l g,g\r_{\pi_1}-\l g,1\r_{\pi_1}^2$ 
for $g\in\R^\O$. 
Noting that $\EE_1(f)=\EE_1(f+c)$ and 
$\Var_{\pi_1}(f)=\Var_{\pi_1}(f+c)$ for every $c\in\R$, 
we get
\[
1-\xi_2(P_1)
\,=\, \min_{\substack{f\in\R^\O:\\ \Var_{\pi_1}(f)\neq0}} 
	\frac{\EE_1(f)}{\Var_{\pi_1}(f)}, 
\]
where $\Var_{\pi_1}(f)\neq0$ iff $f$ is not constant.
Equivalently,
\[
1+\xi_{\abs{\O}}(P_1) \,=\, 1+\min_{f\neq0}\,
	\frac{\l Pf, f\r_{\pi_1}}{\l f, f\r_{\pi_1}} 
\,=\, \min_{f\neq0}\, 
	\frac{\FF_1(f)}{\l f, f\r_{\pi_1}}.
\]
It is easy to check (using reversibility) that 
\[
\EE_1(f) \,=\, \frac12\sum_{x,y\in\O}\bigl(f(x)-f(y)\bigr)^2 
	\pi_1(x) P_1(x,y)
\]
and
\[
\FF_1(f) \,=\, \frac12\sum_{x,y\in\O}\bigl(f(x)+f(y)\bigr)^2 
	\pi_1(x) P_1(x,y).
\]
It follows from the first assumption of this lemma 
that $\EE_1(f)\le A\,\EE_2(f)$ 
and $\FF_1(f)\le A\,\FF_2(f)$ for every $f\in\R^\O$. 
The second assumption implies 
$\l f,f\r_{\pi_1}\ge a \,\l f,f\r_{\pi_2}$ and hence that 
$1+\xi_{\abs{\O}}(P_1)\le\frac{A}{a}(1+\xi_{\abs{\O}}(P_2))$.
Additionally, for $f\in\R^\O$ with $\l f,1\r_{\pi_1}=0$, 
we obtain
\[
\Var_{\pi_2}(f) \,=\, \l f,f\r_{\pi_2} - \l f,1\r_{\pi_2}^2 
\,\le\, \l f,f\r_{\pi_2} 
\,\le\, \frac1a\, \l f,f\r_{\pi_1} 
\,=\, \frac1a\,\Var_{\pi_1}(f),
\]
which implies $\Var_{\pi_1}(f)\ge a\,\Var_{\pi_2}(f)$ for every 
$f\in\R^\O$ and therefore 
$1-\xi_2(P_1)\le\frac{A}{a}(1-\xi_2(P_2))$. 
Note that if $P_2$ has only non-negative eigenvalues 
it is enough to compare $\EE_1$ and $\EE_2$, 
since $\lambda(P_2)=1-\xi_2(P_2)$. Both do not 
depend on the diagonal elements of $P_1$ and $P_2$.
This proves the claim.
\end{proof}


\section{The models} \label{sec:2_models}

In this section we introduce the models that 
we study in this paper. 
Although we are mainly interested in sampling from the 
Potts model, we additionally need the closely related 
random-cluster and \jointnamelong\ (\jointname) models. 
Since all these models are defined on an underlying graph, 
we begin with providing some general graph terminology; 
see e.g.~Diestel~\cite{Diestel} or Mohar and Thomassen~\cite{Mohar} 
for a more comprehensive introduction to graph theory.

A 
\emph{graph}
\index{graph} 
$G$ is a pair $(V,E)$, where $V$ is the 
finite set of \emph{vertices} and $E$ is the set of $\emph{edges}$, 
together with a function $\varphi$ that assigns to each edge 
$e\in E$ a set of at most two vertices, i.e. $\varphi(e)=\{u,v\}$ 
for some $u,v\in V$, which are called its 
\emph{endvertices}\index{endvertices}. 
We denote the endvertices of an edge $e\in E$ by 
$e^{(1)}$ and $e^{(2)}$, i.e. $\varphi(e)=\{e^{(1)}, e^{(2)}\}$. 
Let $\varphi(E)=\{\varphi(e):\, e\in E \}$. 
We say that $u$ and $v$ are \emph{neighbors}\index{neighbors} in $G$, 
if $\{u,v\}\in \varphi(E)$.
Furthermore, $u$ and $v$ are called 
\emph{connected}\index{connected}, if there exist vertices 
$v_0,\dots,v_n\in V$ 
such that $v_0=u$, $v_n=v$ and $\{v_{i-1},v_i\}\in\varphi(E)$, 
$i=1,\dots,n$. We write $u\conn{} v$ if $u$ and $v$ 
are connected (in~$G$). 
Now suppose we have two graphs, $G=(V_G,E_G,\varphi_G)$ and 
$H=(V_H, E_H, \varphi_H)$. We say that $H$ is a 
\emph{subgraph}\index{subgraph} 
of $G$ if $V_H\subset V_G$, $E_H\subset E_G$ and 
$\varphi_H(e)=\varphi_G(e)$ for all $e\in E_H$. 
If additionally $V_H=V_G$, we say $H$ is a 
\emph{spanning subgraph}\index{subgraph!spanning} 
of $G$.
Let $A\subset E$ be a subset of the edges of the graph 
$G=(V,E,\varphi)$. 
Then $G_A:=(V,A,\varphi)$ is a spanning subgraph of $G$ 
and we write $u\xconn{A}v$ if $u,v\in V$ are connected 
in $G_{A}$, i.e. there exist vertices $v_0,\dots,v_n\in V$ 
such that $v_0=u$, $v_n=v$ and $\{v_{i-1},v_i\}\in\varphi(A)$, 
$i=1,\dots,n$.
Clearly, $\xconn{A}$ defines an 
equivalence relation on $V$ for every $A\subset E$; its 
equivalence classes are called 
\emph{connected components}\index{connected components} of $G_A$.
Two distinct edges $e_1, e_2\in E$ that have the same endvertices, 
i.e.~they satisfy $\varphi(e_1)=\varphi(e_2)$, are said to be  
\emph{parallel},\index{parallel edges} and edges with equal 
endvertices, i.e.~$e\in E$ with $e^{(1)}= e^{(2)}$ 
(or $\abs{\varphi(e)}=1$), 
are called \emph{loops}\index{loops}.

Most graphs of this paper are 
\emph{simple}\index{graph!simple}, i.e. 
graphs that contain no parallel edges.
(Some authors use the term ``graph'' for simple graphs, 
``multigraph'' for graphs as defined above.)
The advantage of simple graphs is that 
$\abs{\varphi(E)}=\abs{E}$ and thus $\varphi$ is a 
bijection between $E$ and $\varphi(E)$. 
In this case we identify $\varphi(E)$ with $E$ and 
write $G=(V,E)$ for the graph $(V,E,\varphi)$, as well as 
$e=\{u,v\}$ for $\varphi(e)=\{u,v\}$.
With a slight abuse of notation we omit $\varphi$ also if the 
graph is not simple, i.e. we write $G=(V,E)$ for $(V,E,\varphi)$, 
and make the convention that 
$\{u,v\}=e$ 
(resp.~$\{u,v\}\in E$) 
simply means 
$\{u,v\}=\varphi(e)$ 
(resp.~$\{u,v\}\in \varphi(E)$). 
(Note that we will use this notation only in one direction, 
i.e. given $e\in E$ we obtain a unique $\{u,v\}\subset V$ 
with $\{u,v\}=e$. 
The other way cannot be done uniquely if the graph contains 
parallel edges.)

%

As example graphs one can have in mind the two-dimensional 
\emph{square lattice}\index{square lattice} 
$\Z^2_L$ of side length $L$, i.e.~the graph $\Z^2_L=(V_{L,2},E_{L,2})$ 
with vertex set $V_{L,2}=\{1,\dots,L\}^2\subset\Z^2$ and 
edge set $E_{L,2}=\bigl\{\{u,v\}\subset V_{L,2}:\,\abs{u-v}=1\bigr\}$, 
where $\abs{\,\cdot\,}$ denotes the $\ell_2$ norm, 
see Figure~\ref{fig:graphs} (left), and a \emph{tree}\index{tree}, 
which is a graph $T=(V,E)$ with $\abs{E}=\abs{V}-1$ and exactly one 
connected component (see Figure~\ref{fig:graphs}, right).
Obviously, deletion of any edge of $T$ increases the number of 
connected components.



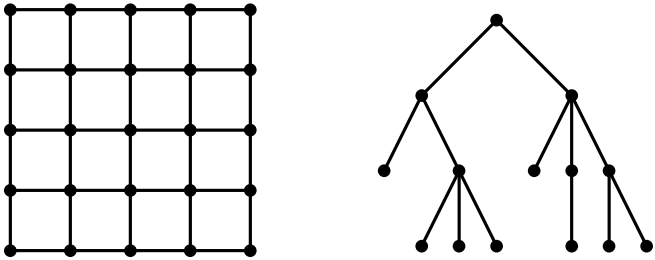
\begin{figure}[ht]
\begin{center}
\scalebox{.8}{
	\psset{xunit=1.0cm,yunit=1.0cm,algebraic=true,dotstyle=*,dotsize=6pt 0,linewidth=1.5pt}
	\begin{pspicture*}(1.67,0.7)(6.33,5.37)
	\psline(2,3)(3,3)
	\psline(3,3)(3,2)
	\psline(3,2)(3,1)
	\psline(3,1)(2,1)
	\psline(2,1)(2,2)
	\psline(2,2)(3,2)
	\psline(2,2)(2,3)
	\psline(2,5)(2,4)
	\psline(2,4)(2,3)
	\psline(2,5)(3,5)
	\psline(3,5)(3,4)
	\psline(3,4)(2,4)
	\psline(3,4)(3,3)
	\psline(4,5)(3,5)
	\psline(4,5)(4,4)
	\psline(4,4)(3,4)
	\psline(4,4)(4,3)
	\psline(4,3)(3,3)
	\psline(4,3)(4,2)
	\psline(4,2)(3,2)
	\psline(4,2)(4,1)
	\psline(4,1)(3,1)
	\psline(5,5)(4,5)
	\psline(5,5)(5,4)
	\psline(5,4)(4,4)
	\psline(5,4)(5,3)
	\psline(5,3)(4,3)
	\psline(5,3)(5,2)
	\psline(5,2)(4,2)
	\psline(5,2)(5,1)
	\psline(5,1)(4,1)
	\psline(6,5)(5,5)
	\psline(6,5)(6,4)
	\psline(6,4)(5,4)
	\psline(6,4)(6,3)
	\psline(6,3)(5,3)
	\psline(6,3)(6,2)
	\psline(6,2)(5,2)
	\psline(6,2)(6,1)
	\psline(6,1)(5,1)
	\psdots(2,3)
	\psdots(2,2)
	\psdots(2,1)
	\psdots(3,3)
	\psdots(3,2)
	\psdots(3,1)
	\psdots(2,4)
	\psdots(2,5)
	\psdots(3,4)
	\psdots(3,5)
	\psdots(4,1)
	\psdots(4,2)
	\psdots(4,3)
	\psdots(4,4)
	\psdots(4,5)
	\psdots(5,1)
	\psdots(5,2)
	\psdots(5,3)
	\psdots(5,4)
	\psdots(5,5)
	\psdots(6,1)
	\psdots(6,2)
	\psdots(6,3)
	\psdots(6,4)
	\psdots(6,5)
	\end{pspicture*}
}
\hspace{6mm}
\scalebox{.8}{
	\psset{xunit=2.5cm,yunit=2.5cm,algebraic=true,dotstyle=*,dotsize=6pt 0,linewidth=1.5pt}
	\begin{pspicture*}(0,-0.15)(2.25,1.75)
	\psline(1,1.5)(0.5,1)
	\psline(1,1.5)(1.5,1)
	\psline(1.5,1)(1.25,0.5)
	\psline(1.5,1)(1.5,0.5)
	\psline(1.5,1)(1.75,0.5)
	\psline(1.5,0.5)(1.5,0)
	\psline(1.75,0.5)(1.75,0)
	\psline(1.75,0.5)(2,0)
	\psline(0.75,0.5)(1,0)
	\psline(0.75,0.5)(0.75,0)
	\psline(0.75,0.5)(0.5,0)
	\psline(0.5,1)(0.75,0.5)
	\psline(0.5,1)(0.25,0.5)
	\psdots(1,1.5)
	\psdots(1.5,1)
	\psdots(0.5,1)
	\psdots(1.75,0.5)
	\psdots(1.5,0.5)
	\psdots(1.25,0.5)
	\psdots(0.75,0.5)
	\psdots(0.25,0.5)
	\psdots(1,0)
	\psdots(0.75,0)
	\psdots(0.5,0)
	\psdots(1.5,0)
	\psdots(1.75,0)
	\psdots(2,0)
	\end{pspicture*}
}
\end{center}
\vspace*{-2mm}
\caption[Example graphs]{The graph $\Z^2_5$ (left) and a tree (right)}
\label{fig:graphs}
\end{figure}

\newpage

For the further analysis we need the following graph 
quantities.
First of all, we will refer to the number of vertices of a graph, 
i.e. $\abs{V}$ for $G=(V,E)$, as the 
\emph{size} of the graph\index{graph!size}.
The \emph{degree}\index{degree} of a vertex $v$ in $G=(V,E)$, i.e. 
$\bigl\vert\{e\in E:\, v\in e\}\bigr\vert$, is denoted by 
$\deg_G(v)$, and we use $\D(G)$ for the 
\emph{maximum degree}\index{degree!maximum} of $G$, i.e.
\[
\D(G) \,=\, \max_{v\in V}\, \deg_G(v).
\]
For example, $\D(\Z_L^2)=4$ for every $L\ge3$.
Additionally, for a graph $G=(V,E)$ we denote by $c(G_A)$ the 
\emph{number of connected components}
\index{connected components!number of} 
of the graph $G_A=(V,A)$, $A\subset E$. Thus, $c(G_A)$ is the 
number of equivalence classes in $V$ with respect to $\xconn{A}$. 
If the graph is fixed we simply write $c(A)$ for $c(G_A)$.
Note that, if $T=(V,E)$ is a tree, $c(T_A)=\abs{V}-\abs{A}$ for all 
$A\subset E$.

Now we introduce the models under consideration.
For this, fix a graph $G=(V,E)$, a natural number $q\ge1$ and 
a real number $\beta\ge0$. Typically, $\beta$ is called the 
\emph{inverse temperature}\index{inverse temperature}.
The \emph{$q$-state Potts model}\index{Potts!model} on $G$ is 
defined as the set of possible \emph{configurations} 
$\Op:=\O_{\rm P}(G)=[q]^V$, where $[q]\,{:=}\,\{1,\dots,q\}$ 
is the set of \emph{colors} (or spins),
together with the probability measure
\vspace{1mm}
\begin{equation} \label{eq:Potts}
\qquad\qquad\pi(\sigma) \;:=\; \pi^{G}_{\beta,q}(\sigma) \;=\; 
	\frac1{Z(G,\beta,q)}\;	e^{\beta\,\abs{E(\sigma)}},
	\qquad \sigma\in\O_{\rm P},
\end{equation}
where 
\vspace{3mm}
\begin{equation} \label{eq:Esigma}
E(\sigma) \,:=\, \Bigl\{e\in E:\, \sigma(u)=\sigma(v) 
		\,\text{ for }\, \{u,v\}=e\Bigr\}
\end{equation}
is the set of edges with equally colored endvertices. 
The {\it normalization constant} 
(also called {\it partition function}) $Z$ is given by
\[
Z(G,\beta,q) \,:=\, \sum_{\sigma\in\Op(G)}\, 
	e^{\beta\,\abs{E(\sigma)}}. 
\] 
This measure is called the 
\emph{Potts measure}\index{Potts!measure} 
(or \emph{Boltzmann distribution}) and 
if $q=2$ we call the Potts model the 
\emph{Ising model}\index{Ising model}. 
The inverse temperature $\beta$ determines the 
interaction strength between neighboring vertices in the graph 
and since $\beta$ is non-negative, configurations with more 
equally colored neighbors have a larger weight in $\pi$. 
In particular, at infinite temperature, i.e. $\beta=0$, 
the measure $\pi$ is the uniform distribution on $\Op$ and the 
larger $\beta$ the more weight goes to the (almost) constant 
configurations.

A closely related model is the 
\emph{random-cluster model}\index{random-cluster!model} 
\index{RC|see{random-cluster}}
(also known as the \emph{FK-model}), that was introduced by
Fortuin and Kasteleyn~\cite{FK}. 
It is defined on the graph $G=(V,E)$ by its state space 
$\Orc=\{A: A\subseteq E\}$ and the 
\emph{random-cluster} ({\it RC}\/) \emph{measure}
\index{random-cluster!measure}
\vspace{1mm}
\begin{equation} \label{eq:RC}
\mu(A) \;:=\; \mu^G_{p,q}(A) \;=\; 
\frac1{Z(G,\log(\frac1{1-p}),q)}\,
\left(\frac{p}{1-p}\right)^{\abs{A}}\,q^{c(A)},\quad A\subset E,
\end{equation}
where $p\in[0,1]$, $c(A)$ is the number of connected components 
in the graph $(V,A)$
and $Z(\cdot,\cdot,\cdot)$ is the same normalization constant as 
for the Potts model~\cite[Thm.~1.10]{G1} (see 
\eqref{eq:FKSW-RC} and \eqref{eq:FKSW-P}). 
For a historical treatment and related topics see Grimmett~\cite{G1}.

To describe the connection of Potts and random-cluster models 
we state a coupling of the corresponding measures that is 
due to Edwards and Sokal~\cite{ES}. This coupling leads to 
the third model that will be considered, namely the 
{\it \jointnamelong}\ ({\it \jointname}) model (or simply the 
joint model). 
For this let $p=1-e^{-\beta}$, as we will assume henceforth.
The \emph{joint model}\index{FK@\jointname!model}
\index{joint model|see{\jointname\ model}}
is defined on $\O_{\rm J}:=\O_{\rm P}\times\O_{\rm RC}$ by the 
\emph{\jointname\ measure}\index{FK@\jointname!measure}
\vspace{1mm}
\begin{equation} \label{eq:joint}
\nu(\sigma,A)\;:=\;\nu_{p,q}^G(\sigma,A)
	\;=\;\frac{1}{Z(G,\log(\frac1{1-p}),q)}\,
	\left(\frac{p}{1-p}\right)^{\abs{A}}\,\ind(A\subset E(\sigma))
\end{equation}
for $(\sigma,A)\in\O_{\rm J}$. 
Again, the normalization constant $Z(\cdot,\cdot,\cdot)$ is the 
same as for the Potts measure~\cite[Thm.~1.10]{G1}.
To see this, note that 
\[
A\subset E(\sigma) \;\iff\; \forall 
\{u,v\}\in A:\, \sigma(u) = \sigma(v),
\]
which implies that $\sigma$ is constant on each 
connected component of $(V,A)$. Hence, for a given $A\subset E$ 
there are exactly $q^{c(A)}$ configurations $\sigma\in\O_{\rm P}$ 
with $\nu(\sigma,A)>0$. Using this, for $\sigma\in\O_{\rm P}$ 
and $A\subset E$ we obtain
\begin{equation}\label{eq:FKSW-RC}
\begin{split}
\sum_{\tau\in\O_{\rm P}} 
	\left(\frac{p}{1-p}\right)^{\abs{A}}\,&\ind(A\subset E(\tau)) 
\,=\, \left(\frac{p}{1-p}\right)^{\abs{A}}\,
	\abs{\bigl\{\tau\in\O_{\rm P}:\, A\subset E(\tau)\bigr\}} \\
&=\, \left(\frac{p}{1-p}\right)^{\abs{A}}\,	q^{c(A)}
\end{split}\end{equation}
and, since $p=1-e^{-\beta}$,
\begin{equation}\label{eq:FKSW-P}
\begin{split}
\sum_{B\in\O_{\rm RC}} 
	\left(\frac{p}{1-p}\right)^{\abs{B}}\,&\ind(B\subset E(\sigma)) 
\;=\; \sum_{k=0}^{\abs{E(\sigma)}} 
	\genfrac(){0pt}{}{\abs{E(\sigma)}}{k}
	\left(\frac{p}{1-p}\right)^{k} \hspace*{1.8cm} \\
&=\, \left(1+\frac{p}{1-p}\right)^{\abs{E(\sigma)}} 
\,=\, e^{\beta \abs{E(\sigma)}}.
\end{split}\end{equation}
Summing over $A\in \O_{\rm RC}$ in \eqref{eq:FKSW-RC} and over 
$\sigma\in\O_{\rm P}$ in \eqref{eq:FKSW-P} 
proves that the normalization constants of $\pi$, $\mu$ and 
$\nu$ are equal.
Another fact that can be deduced from \eqref{eq:FKSW-RC} and 
\eqref{eq:FKSW-P} is that the marginal distributions of 
$\nu$ are exactly $\pi$ and $\mu$, respectively 
(see \cite{ES} or \cite{G1}).
This means that for all $\sigma\in\O_{\rm P}$ and 
$A\subset E$, 
\[
\pi(\sigma) \,=\, \sum_{A\subset E} \nu(\sigma,A) 
\,=\, \nu(\sigma,\O_{\rm RC})
\]
and 
\[
\mu(A) \,=\, \sum_{\sigma\in\O_{\rm P}} \nu(\sigma,A) 
\,=\, \nu(\O_{\rm P},A).
\]

We define the conditional probability of $\sigma\in\O_{\rm P}$ 
with respect to $\nu$ given $A\subset E$ by
\begin{equation} \label{eq:joint-cond}
\nu(\sigma | A) \,:=\, \begin{cases}
\frac{\nu(\sigma,A)}{\nu(\O_{\rm P},A)}, 
	&\text{ if } (\sigma,A)\in\O_{\rm J}, \\
\quad 0,& \text{ otherwise.}
\end{cases}
\end{equation}
For $\nu(A \,|\, \sigma)$ replace $\nu(\Op,A)$ by 
$\nu(\sigma,\Orc)$. 
We obtain 
\[
\sum_{A\in\O_{\rm RC}} \mu(A)\, \nu(\sigma | A) 
\,=\, \pi(\sigma)
\]
and
\[
\sum_{\sigma\in\O_{\rm P}} \pi(\sigma)\, \nu(A \,|\, \sigma) 
\,=\, \mu(A).
\]
This can be interpreted in the following way.
Assume that we can simulate a random variable $X\in\O_{\rm RC}$ that is 
distributed with respect to~$\mu$. The random variable $Y\in\O_{\rm P}$ 
obtained by sampling from $\nu(\cdot |X)$ is then 
distributed with respect to $\pi$. 
In fact, the conditional probabilities take a rather simple form, 
namely $\nu(\sigma | A) = q^{-c(A)}\,\ind(A\subset E(\sigma))$ and 
$\nu(A \,|\, \sigma) = p^{\abs{A}} (1-p)^{\abs{E(\sigma)}-\abs{A}}\,
	\ind(A\subset E(\sigma))$.
Thus, given a RC configuration $X\sim\mu$, i.e.~one that is distributed with 
respect to $\mu$, we can generate a Potts configuration 
$Y\sim\pi$ by assigning a random color from $\{1,\dots,q\}$ 
independently to each connected component of the graph $(V,X)$. 
For the reverse way, given $Y\sim\pi$, include all edges 
$e=\{e^{(1)},e^{(2)}\}\in E$ with $Y(e^{(1)})=Y(e^{(2)})$ in $X$ with 
probability $p$. 
Hence, if the Potts and the RC model are 
defined on the same graph $G$ and $p=1-e^{-\beta}$, 
then the possibility of efficient sampling from 
either model enables also efficient sampling from the other.

We will see in the following section that the 
non-local Markov chain considered, namely the Swendsen--Wang dynamics, 
is based on this connection of the Potts and random-cluster models.


\section{The dynamics} \label{sec:2_dynamics}

This section is devoted to the definition of the 
Markov chains that will be used to sample approximately 
from the Potts and the random-cluster model. 
In Chapters~\ref{chap:spin} and~\ref{chap:bond} 
we will compare the spectral gaps of these Markov chains 
and, as a consequence, obtain new results on 
the mixing properties of the Swendsen--Wang and the single-bond 
dynamics. 
Recall that we have fixed a graph $G=(V,E)$, some $q\in\N$, 
a~real $\beta\ge0$ and $p=1-e^{-\beta}$.
We distinguish between dynamics for the Potts model 
and dynamics for the random-cluster model, 
and we will propose one \emph{local} and one 
\emph{non-local} Markov chain in both cases. 
``Local'' means that the Markov chain changes the current 
state in one step only locally. In other words:  
let $P$ (resp.~$\Pt$) be the 
transition matrix of a Markov chain on $\O_{\rm P}$ 
(resp.~$\O_{\rm RC}$) with stationary distribution $\pi$ 
(resp.~$\mu$); then we say that $P$ (resp.~$\Pt$), 
or its corresponding Markov chain, is \emph{local} 
if $\abs{\{v\in V:\, \sigma(v)\neq\tau(v)\}}\le1$ for all 
$\sigma,\tau\in\O_{\rm P}$ with $P(\sigma,\tau)>0$ 
(resp.~$\abs{(A\setminus B) \cup (B\setminus A)}\le1$ for all 
$A,B\in\O_{\rm RC}$ with $\Pt(A,B)>0$). 
Thus, local Markov chains update only one 
vertex (resp.~edge) of the current configuration per step. 

Another notion of locality is that the transition 
probabilities, e.g.~$P(\sigma,\tau)$ for $\sigma,\tau\in\O_{\rm P}$, 
depend only locally on the states $\sigma$ and $\tau$. 
This is called the \emph{finite range interaction property} of 
the transition probabilities (see e.g.~\cite{M}), and means 
(in the case of the Potts model) that for every 
$\sigma,\tau\in\O_{\rm P}$ with $P(\sigma,\tau)>0$ there exists a 
subset of the vertices $V_0\subset V$, 
with $\abs{V_0}$ 
``not too large'', 
such that $P(\sigma,\tau)=P(\sigma',\tau')$ for 
$\sigma',\tau'\in\O_{\rm P}$ with $\sigma(v)=\sigma'(v)$ and 
$\tau(v)=\tau'(v)$ for all $v\in V_0$. 
An analogous formulation can be found for Markov chains for the 
RC model.
This property is especially interesting for computational reasons, 
because it enables efficient computation of the transition 
probabilities and therefore efficient simulation of the 
Markov chain.
From the local dynamics of this section only the one for the 
Potts model satisfies the finite range interaction property.

\subsection
{Dynamics for the Potts model}

\paragraph{\bf The heat-bath dynamics} \index{dynamics!heat-bath} 
We begin with the definition of the local Markov chain for the 
Potts model. 
This Markov chain, namely the 
\emph{heat-bath $($HB$)$ dynamics}
(or \emph{Glauber dynamics}), 
was introduced by Glauber \cite{Glauber} in 1963 for the 
Ising model and, since then, has become the most studied Markov 
chain for the $q$-state Potts model (especially for $q=2$). 
The one-step transitions can be described as follows. 
Suppose that the current state of the Markov chain is 
$\sigma\in\O_{\rm P}$. Then:
\begin{itemize}\label{page:P-HB}
	\item[(HB1)] Choose a vertex $v\in V$ uniformly at random.
	\item[(HB2)] Let 	$\O_{\sigma,v}:=\{\tau\in\O_{\rm P}:\, 
											\tau(u)=\sigma(u), \forall u\neq v\}$
						and choose the next state of the Markov chain, 
						say $\tau$, with respect to 
						$\pi(\cdot\mid\O_{\sigma,v})$, that is, the conditional 
						probability (with respect to $\pi$) given that $\tau$ differs 
						from $\sigma$ only at $v$.
\end{itemize}
%
For $\sigma\in\O_{\rm P}$, $k\in[q]$ and $u,v\in V$, 
let
\begin{equation} \label{eq:sigma_v_k}
\sigma^{v,k}(u) \,:=\, \begin{cases}
\sigma(u), & \text{ if } u\neq v, \\
k, & \text{ if } u=v,
\end{cases} 
\end{equation}
be the configurations that differ from $\sigma$ at most at $v\in V$ 
and let 
\[
d_{v,k}(\sigma) 
\,:=\, \abs{\bigl\{e\in E:\, \sigma(u)=k \,\text{ for }\, \{u,v\}=e\bigr\}}
\]
be the number of edges in $G$ connecting $v$ to vertices with 
color $k$ in $\sigma$.
It is easy to see that the set $\O_{\sigma,v}$ from (HB2) 
can be written as $\O_{\sigma,v}=\bigcup_{l=1}^q\{\sigma^{v,l}\}$, 
and that the conditional probabilities 
satisfy
\begin{equation} \label{eq:HB-cond}
\pi(\sigma^{v,k}\mid\O_{\sigma,v}) \,=\, \frac{\pi(\sigma^{v,k})}
	{\sum_{l=1}^q\pi(\sigma^{v,l})}
\,=\, \frac{e^{\beta d_{v,k}(\sigma)}}{\sum_{l=1}^q e^{\beta d_{v,l}(\sigma)}}
\end{equation}
and $\pi(\tau\mid\O_{\sigma,v})=0$, whenever $\tau\neq\sigma^{v,k}$ 
for all $k$.
Therefore we can write the transition matrix of the 
heat-bath dynamics on $\O_{\rm P}$ as
\begin{equation}\label{eq:P-HB}
\hb(\sigma,\tau) 
\;:=\; P_{{\rm HB},\beta,q}^G(\sigma,\tau) \;=\;
\frac{1}{\abs{V}}\sum_{v\in V}\,\frac{\pi(\tau)}
	{\sum_{l=1}^q\pi(\sigma^{v,l})}\,
	\ind(\tau\in\O_{\sigma,v}).
\end{equation}
Clearly, $P_{\rm HB}$ is reversible with respect to $\pi$ and, 
as long as $\beta<\infty$, also ergodic. 

An interesting feature of this simple 
construction of a Markov chain is that the (temporal) mixing 
properties of the heat-bath dynamics are proven to be 
equivalent to some \emph{spatial mixing} 
properties of the associated Potts measure if the 
underlying graph is a rectangular subset of the (physically most relevant) 
$d$-dimensional integer lattice $\Z^d$; see 
e.g.~\cite{AizHol,DSVW,Holley,MO_finite-volume,MO1,MO2,MOS,SZ_equivalence}
or the survey \cite{M} by Martinelli. 
Similar results are known for trees \cite{MSW}.
It turns out that certain spatial properties of 
the model imply tight bounds on the spectral gap (or mixing time) 
of the heat-bath dynamics.
In Section~\ref{sec:2_known-results} we will state 
some of these results, in particular for the two-dimensional 
square lattice where we have an almost complete characterization 
of the mixing properties.

One drawback of the heat-bath dynamics is that it is 
typically slowly (i.e. not rapidly) mixing at low temperatures 
(large $\beta$); see e.g.~\cite{CGMS,CCS,LPW,Schonmann}. 

We refer again to Section~\ref{sec:2_known-results}.
In fact, at zero temperature ($\beta=\infty$) the heat-bath dynamics 
is not even irreducible, although this is the simplest case: 
uniform distribution on $q$ (constant) configurations.

In the next paragraph we introduce the Swendsen--Wang dynamics, 
which overcomes the slow mixing behavior at low temperatures.
This Markov chain is the primarily studied Markov chain of this paper.

\paragraph{\bf The Swendsen--Wang dynamics} \index{dynamics!Swendsen-Wang} 
Inspired by the representation of 
the Potts models of Fortuin and Kasteleyn~\cite{FK}, 
Swendsen and Wang~\cite{SW} invented around 1987 a Markov chain, 
that makes large 
steps through the state space
and is surprisingly simple to implement. 
We call this Markov chain, which is presently the best 
candidate for the most efficient dynamics to sample from the Potts 
model on general families of graphs, the 
Swendsen--Wang (SW) dynamics. 

As stated several times, this Markov chain 
is based on the connection of the random-cluster and Potts models 
that is given by the coupling of the corresponding measures 
of Edwards and Sokal~\cite{ES}. 

First, recall that this coupling is given by the \jointname\ measure 
\eqref{eq:joint}, i.e. 
\[
\nu(\sigma,A)\,=\,\frac{1}{Z}\,
	\left(\frac{p}{1-p}\right)^{\abs{A}}\,
	\ind\bigl(A\subset E(\sigma)\bigr), \qquad \sigma\in\Op,\, A\in\Orc, 
\]
with $E(\sigma)$ from \eqref{eq:Esigma} and some $p\in[0,1]$, 
and that the 
conditional probabilities (w.r.t.~$\nu$) of $\sigma$ given $A$ 
(or of $A$ given $\sigma$) satisfy
\begin{align} \label{eq:joint-cond-1}
\nu(\sigma \,|\, A) \,&=\, q^{-c(A)}\;
	\ind\bigl(A\subset E(\sigma)\bigr)\\[-2mm]
\shortintertext{and \vspace*{-2mm}} \label{eq:joint-cond-2} 
\nu(A \,|\, \sigma) \,&=\, p^{\abs{A}} 
	(1-p)^{\abs{E(\sigma)}-\abs{A}}\;
	\ind\bigl(A\subset E(\sigma)\bigr),
\end{align}
see \eqref{eq:joint-cond} and the discussion thereafter.

The \emph{Swendsen--Wang dynamics} 
(for the Potts model) 
makes use of these conditional probabilities in such a way that, 
given the current state $\sigma\in\Op$, the next state $\tau\in\Op$ 
is generated by first sampling an $A\subset E$ from 
$\nu(\,\cdot \,|\, \sigma)$, and then sampling $\tau$ from 
$\nu(\,\cdot\, | A)$. Thus, the transition matrix of the 
Swendsen--Wang dynamics for the $q$-state Potts model 
on $G$ at inverse temperature $\beta=-\ln(1-p)$ is given by
\begin{equation} \label{eq:SW-P}
\begin{split}
\sw(\sigma,\tau) \,&\coloneqq\, P_{{\rm SW},\beta,q}^G(\sigma,\tau) 
\,=\, \sum_{A\subset E}\, \nu(A \,|\, \sigma)\; \nu(\tau \,|\, A) \\
&=\,(1-p)^{\abs{E(\sigma)}}\,\sum_{A\subset E} 
	\,\left(\frac{p}{1-p}\right)^{\abs{A}}\, q^{-c(A)}\; 
	\ind\bigl(A\subset E(\sigma)\cap E(\tau)\bigr).
\end{split}\end{equation}
Using the same interpretation of this sampling with respect to the 
conditional probabilities that was given in 
Section~\ref{sec:2_models}, 
we can describe one step of the Swendsen--Wang dynamics as the 
following two-step procedure:

\begin{enumerate}
	\item[(SW1)] Given a Potts configuration $\sigma\in\Op$ on $G$, 
		delete each edge of $E(\sigma)$ independently with 
		probability $1-p = e^{-\beta}$. This gives $A\in\O_{\rm RC}$.
	\item[(SW2)] Assign a random color from $[q]$ independently to each 
		connected component of $(V,A)$. Vertices of the same component 
		get the same color. This gives $\tau\in\O_{\rm P}$.
\end{enumerate}
Ergodicity of this Markov chain is easy to check. 
For this, imagine 
that the generated edge set $A$ of the first step (SW1) is the empty 
set (this happens with probability $(1-p)^{\abs{E(\sigma)}}>0$ 
for $p<1$), and 
then, in step (SW2), every possible configuration can be generated 
with probability $q^{-\abs{V}}$. Hence, $\sw(\sigma,\tau)>0$ for all 
$\sigma,\tau\in\Op$, which implies ergodicity. 
If $p=1$ ($\beta=\infty$), then the only elements of $\Op$ 
with positive measure are the $q$ constant 
configurations, for which we have $\sw(\sigma,\tau)=q^{-1}>0$. 
For the reversibility of $\sw$ with respect to $\pi$ note that  
$\pi(\sigma)=\frac1Z (1-p)^{-\abs{E(\sigma)}}$, since $p=1-e^{-\beta}$. 
We deduce from \eqref{eq:SW-P} that
\[
\pi(\sigma)\,\sw(\sigma,\tau)\,=\, \frac1Z\,\sum_{A\subset E} \,p^{\abs{A}}\, 
	(1-p)^{-\abs{A}}\, q^{-c(A)}\; 
	\ind\bigl(A\subset E(\sigma)\cap E(\tau)\bigr),
\]
which is symmetric in $\sigma,\tau$.
This shows that the SW dynamics is reversible and ergodic for 
any $G$, $q$ and $\beta$, but, in contrast with the heat-bath 
dynamics, this Markov chain seems to be efficient also for 
large $\beta$. For results showing that this is true for large enough 
$\beta$ 
see e.g.~\cite{Hu} and \cite{M2}.
Additionally, there are plenty of numerical results suggesting 
that the Swendsen--Wang dynamics is rapidly mixing at all 
temperatures in various cases. 
These include the Ising model and the 3-state Potts model 
on the two-dimensional square 
lattice \cite{DuZhengWang,SalasSokal} and the Ising 
model on the three-dimensional cubic lattice \cite{OssolaSokal}.

However, rigorous proofs of such statements are rare. 
There are, as far as we know, presently only four cases where 
rapid mixing at all temperatures is proven. 
These are on trees, cycles \cite{CF, Hu, Long} and 
narrow grids \cite{CDFR} for all $q\in\N$ and 
on the complete graph for $q=2$ \cite{Long,LNP}. 
One of our goals is to add the Ising model ($q=2$) on 
the two-dimensional square lattice to this list.

There are also some rigorous results regarding slow mixing. 
We state three of them: Li and Sokal 
\cite{LiS} proved that the inverse spectral gap is larger 
than some constant times the \emph{specific heat} 
(that is proportional to the second derivative with respect to 
$\beta$ of the Potts normalization constant $Z$, 
see \eqref{eq:Potts}). 
This shows (at least) that the inverse spectral gap 
cannot be bounded independently of the size of the 
graph in many cases of interest. 
Furthermore, Gore and Jerrum \cite{GJ} showed slow mixing of the 
SW dynamics on the complete graph for $q\ge3$ for some value of 
the inverse temperature. 
Recently, Borgs, Chayes and Tetali \cite{BCT} 
gave tight upper and lower bounds on the spectral gap 
(or mixing time) for SW on rectangular subsets 
of $\Z^d$, $d\ge2$, with periodic boundary conditions 
at some (critical) temperature for $q$ large enough, 
which show that the spectral gap is exponentially small in 
the size of the graph.
Since we need some of the above results in the following 
we state them in more detail in 
Section~\ref{sec:2_known-results}.

Apart from the results given above, the Swendsen--Wang dynamics 
resisted a precise analysis in (physically) relevant cases, 
where the heat-bath dynamics is well-understood.
For instance, it was 
not proven that SW is rapidly mixing for the Ising model on the 
two-dimensional square lattice for all temperatures above 
the critical one, which has been known for the heat-bath dynamics 
for 20 years \cite{MO1,MO2}.

We will prove this statement in Chapter~\ref{chap:spin} by 
comparison with the heat-bath dynamics.
Furthermore we will obtain rapid mixing at all temperatures below 
the critical one.
 
The main challenge in proving rapid mixing of the Swendsen--Wang 
dynamics for the two-dimensional square lattice is the low 
temperature (large $\beta$) regime, where no Markov chain 
is proven so far to be rapid down to the critical temperature. 
Therefore, we need to consider dynamics for the corresponding 
random-cluster model that seems to be 
(and we will see that it indeed is) 
more convenient for relating spectral gaps at high and low 
temperatures.

\subsection
	{Dynamics for the random-cluster model}

We begin this subsection with the definition of a non-local 
Markov chain for the random-cluster model. 
This is again the Swendsen--Wang dynamics.  
Since it is defined by means of conditional 
probabilities with respect to the \jointname\ measure $\nu$ 
(see \eqref{eq:SW-P}), this Markov chain can be naturally defined also for 
the random-cluster model.

\paragraph{\bf The Swendsen--Wang dynamics}\index{dynamics!Swendsen-Wang}
The SW dynamics for the random-cluster model 
is based on the same connection that was given in the last 
subsection, but here we make the two-step procedure in 
reverse order:
\begin{enumerate}
	\item[($\tilde{\text{SW}}$1)] Given a random-cluster configuration 
		$A\in\O_{\rm RC}$ on $G$, assign a random color from $[q]$ 
		independently to each connected component of $(V,A)$. 
		Vertices of the same component get the same color. 
		This gives $\sigma\in\O_{\rm P}$.		
	\item[($\tilde{\text{SW}}$2)] Take $E(\sigma)$ and delete each edge 
		independently with probability $1-p$. 
		This gives the new RC configuration $B\in\O_{\rm RC}$.
\end{enumerate}
Thus, the transition matrix of the Swendsen--Wang dynamics 
for the random-cluster model on $G$ with parameters $p$ and $q$ 
is given by
\begin{equation} \label{eq:SW-RC}
\begin{split}
\tsw(A,B) \,&\coloneqq\, \Pt_{{\rm SW},p,q}^G(A,B) 
\,=\, \sum_{\sigma\in\Op}\, \nu(\sigma \,|\, A)\; \nu(B \,|\, \sigma) \\
&=\, q^{-c(A)}\,\left(\frac{p}{1-p}\right)^{\abs{B}}\,
	\sum_{\sigma\in\O_{\rm P}} \,(1-p)^{\abs{E(\sigma)}}\,
		\ind\bigl(A\cup B \subset E(\sigma)\bigr).
\end{split}
\end{equation}

Ergodicity and reversibility of this Markov chain can be shown by 
similar ideas to those above.
Although the Swendsen--Wang dynamics for the RC model appears 
quite often in the literature (especially its generalization to 
non-integer values of $q$; see 
e.g.~\cite{CMachta1,CMachta2,DGMOS,DengQianBloete})
we are not aware of any attempt to prove mixing properties 
of this Markov chain. 
The following lemma shows however that every result on  
the spectral gap of the SW dynamics for the Potts model 
is also valid for SW for the corresponding random-cluster model.

\begin{lemma} \label{lemma:SW_P-RC}
Let $\tsw$ $($resp.~$\sw$$)$ be the Swendsen--Wang dynamics 
for the random-cluster $($resp.~$q$-state Potts$)$ model with parameters 
$p$ and $q$ $($resp.~at inverse temperature $\beta$ with $p=1-e^{-\beta}$$)$. 
Then  
\[
\lambda(\sw) \;=\; \lambda(\tsw).
\]
\end{lemma}

\begin{proof}
Define the $\Op\times\Orc$-matrix $\C$ by
\[
\C(\sigma,A) \,:=\, \nu(A \,|\, \sigma)
\,=\, \frac{\nu(\sigma, A)}{\pi(\sigma)} , \qquad \sigma\in\Op, A\in\Orc.
\]
Regard $\C$ as an operator that maps from $L_2(\mu)$ to $L_2(\pi)$  
(cf. \eqref{eq:map}) by
\[
\C g(\sigma) \,=\, \sum_{A\in\Orc} \C(\sigma,A) g(A)
\]
for $g\in\R^{\Orc}$. 
By the usual definition of the adjoint of an operator between different 
Hilbert spaces (see e.g.~\cite[Def.~3.9-1]{Krey}), the adjoint operator 
$\C^*$ of $\C$ satisfies
\[
\l\C^* f, g\r_\mu \,=\, \l f, \C g\r_\pi
\]
for all $f\in\R^{\Op}$ and $g\in\R^{\Orc}$. Thus, the matrix 
corresponding to $\C^*$ is given by 
\[
\C^*(A,\sigma) \,=\, \frac{\pi(\sigma)}{\mu(A)}\, \C(\sigma,A)
\,=\, \frac{\nu(\sigma, A)}{\mu(A)} 
\,=\, \nu(\sigma \,|\, A).
\]
The definitions of the Swendsen--Wang dynamics in \eqref{eq:SW-P} and 
\eqref{eq:SW-RC} imply $\sw~=~\C \C^*$ and $\tsw=\C^* \C$. 
Additionally, we define $S_\pi(\sigma,\tau):=\pi(\tau)$, 
$S_\mu(A,B):=\mu(B)$ and 
$S_{(\pi,\mu)}(\sigma,A)=\mu(A)$ for all 
$\sigma,\tau\in\Op$ and $A,B\in\O_{\rm RC}$. 
It is straightforward to verify that 
$S^*_{(\pi,\mu)}(A,\sigma)=\pi(\sigma)$, 
$\sigma\in\Op$, $A\in\O_{\rm RC}$. Thus,  
$S_{(\pi,\mu)} S^*_{(\pi,\mu)} = S_{\pi}$ and 
$S^*_{(\pi,\mu)} S_{(\pi,\mu)} = S_{\mu}$.
We obtain $\C \C^*-S_{\pi}=(\C-S_{(\pi,\mu)})(\C-S_{(\pi,\mu)})^*$
since by definition 
$S_{(\pi,\mu)} \C^*=S_\pi=S_\pi^*=\C S^*_{(\pi,\mu)}$  
(see \eqref{eq:joint-cond} and below). 
Analogously we get 
$\C^* \C-S_{\mu}=(\C-S_{(\pi,\mu)})^*(\C-S_{(\pi,\mu)})$.
Recall the definition of the spectral gap from \eqref{eq:gap-norm}. 
It follows from simple properties of the norm of (adjoint) operators 
between Hilbert spaces (see e.g. \cite[Thm.~3.9-4]{Krey}) that
\[\begin{split}
1-\lambda(\sw) \,&=\, \norm{\sw-S_\pi}_\pi
\,=\, \norm{(\C-S_{(\pi,\mu)})(\C-S_{(\pi,\mu)})^*}_\pi \\
&=\, \norm{\C-S_{(\pi,\mu)}}_{L_2(\mu)\to L_2(\pi)}^2 
\,=\, \norm{(\C-S_{(\pi,\mu)})^*(\C-S_{(\pi,\mu)})}_\mu \\
&=\, \norm{\tsw-S_\mu}_\mu \,=\, 1-\lambda(\tsw),
\end{split}\]
which proves the statement.
\end{proof}

\vspace{1mm}

\paragraph{\bf The single-bond dynamics}

The last kind of Markov chains we want to consider are the 
local Markov chains for the random-cluster model.
In fact, we need more than one construction in this case.
This comes from the fact that the two dynamics, 
namely heat-bath and single-bond dynamics, both have 
properties that are needed for the further analysis.

First of all we introduce the 
\emph{heat-bath} (HB) \emph{dynamics}\index{dynamics!heat-bath}
for the random-cluster model. 
Similarly to the heat-bath dynamics for the Potts model, 
suppose that $A\in\Orc$ is the current state of the Markov chain. 
Then:
\begin{itemize}
	\item[($\tilde{\text{HB}}$1)] 
						Choose an edge $e\in E$ uniformly at random.
	\item[($\tilde{\text{HB}}$2)] 
						Choose the next state, 
						say $B\in\Orc$, with respect to 
						$\mu\bigl(\cdot\bigm|\{A\cup e, A\setminus e\}\bigr)$, 
						i.e. the conditional probability (w.r.t.~$\mu$) 
						given that $B$ differs from $A$ only by $e$.
\end{itemize}
We use $A\cup e$ instead of $A\cup\{e\}$ 
(similarly for $\cap,\setminus$).
Recall that, for $A\in\Orc$, 
we write $u\xconn{A} v$ if $u,v\in V$ are connected in 
the subgraph $(V,A)$ and that we denote the endpoints of 
$e\in E$ by $e^{(1)}$ and $e^{(2)}$.
Write $u\xnconn{A} v$ 
if $u$ and $v$ are 
not connected in $(V,A)$. 
If we additionally define $A^e:=A\ominus e$ with the symmetric 
difference~$\ominus$, i.e.
\begin{equation} \label{eq:RC-Ae}
A^e \;=\; \begin{cases}
A\cup e, & \text{ if } e\notin A, \\
A\setminus e, & \text{ if } e\in A,
\end{cases}
\end{equation}
we can write 
the transition probability matrix $\thb:=\Pt_{{\rm HB},p,q}^G$ of the 
heat-bath dynamics for the RC model as
\begin{equation}\label{eq:RC-HB-def}
\thb(A,B) \;=\; \frac1{\abs{E}}\,\sum_{e\in E}\,
\frac{\mu(B)}{\mu(A)+\mu(A^e)}\;\ind(B\setminus e = A\setminus e).
\end{equation}
Note that unless $A=B$, at most one summand in the above sum is not zero. 
Using the definition of $\mu$ from \eqref{eq:RC} we obtain 
\begin{equation} \label{eq:RC-quotient}
\frac{\mu(A)}{\mu(A^e)}\,=\,\left(\frac{p}{1-p}\right)^{\abs{A}-\abs{A^e}} 
	q^{c(A)-c(A^e)},
\end{equation}
where $c(A)=c(A^e)$ if and only if $e^{(1)}\lconn{A\setminus e}e^{(2)}$. 
Otherwise, $c(A)-c(A^e)=2 \,\ind(e\notin A)-1$.
Hence, $\thb$ satisfies
\begin{equation} \label{eq:RC-HB}
\thb(A,B) \;=\;  \frac1{\abs{E}}\sum_{e\in E}\,\begin{cases}
p, & \text{ if } B=A\cup e \text{ and } 
				e^{(1)}\stackrel{A\setminus e}{\longleftrightarrow}e^{(2)},\\
1-p, & \text{ if } B=A\setminus e \,\text{ and } 
				e^{(1)}\stackrel{A\setminus e}{\longleftrightarrow}e^{(2)},\\
\frac{p}{p+q(1-p)}, & \text{ if } B=A\cup e \text{ and } 
				e^{(1)}\nconn{A\setminus e}e^{(2)},\\
\frac{q(1-p)}{p+q(1-p)}, & \text{ if } B=A\setminus e \,\text{ and } 
				e^{(1)}\nconn{A\setminus e}e^{(2)},
\end{cases}
\end{equation}
for $A,B\in\Orc$. 
We immediately obtain from this equation that $\thb$ does not 
have the finite range interaction property, since we have to 
check connectivity of two vertices in each step, which is 
in general 
a non-local property. 
In addition to the need of this Markov chain in an 
intermediate step of our proof of rapid mixing for the Swendsen--Wang 
dynamics, the heat-bath dynamics is interesting in its own right. 
One reason is that it also provides a Markov chain for the 
random-cluster model with non-integer values of $q>0$. 
It would be interesting to prove a similar  
relation between temporal mixing properties of 
$\thb$ and some ``spatial mixing'' property in the RC model, 
as it is known for single-spin dynamics for the Potts model. 
Another reason is that the heat-bath dynamics was proposed by 
Propp and Wilson \cite{PW} to produce exact samples from $\mu$ 
(and so also from the Potts model) using their famous 
\emph{coupling from the past} procedure. 
Therefore, every mixing time bound on $\thb$ leads to a bound 
on the (expected) cost of their procedure.

To the best of our knowledge, there is presently only one result 
on mixing of local dynamics for the RC model. 
This is the upper bound on the mixing time of Ge and 
\v{S}tefankovi\v{c}~\cite{GeS} that 
shows rapid mixing of a Metropolis-type version of $\thb$ 
for every $p\in(0,1)$ and $q>0$ 
if the underlying graph has bounded \emph{tree-width}. 
We will state this result later (see 
Theorem~\ref{th:SB_linear-width}).

The advantage of the heat-bath dynamics for our purposes is that 
its spectral gap, if the underlying graph $G$ is \emph{planar}, 
can be shown to be the same as the spectral gap of 
the heat-bath dynamics for the RC model on the \emph{dual graph} 
of $G$ with suitable values of the parameters $p$ and $q$ 
(see Chapter~\ref{chap:2d}). 
This provides us with a relation of the mixing properties 
of the Swendsen--Wang dynamics at high 
and low temperatures  (see Theorem~\ref{th:SW-SB_dual}).

The second local Markov chain that we consider is inspired by 
the local behavior of the Swendsen--Wang dynamics and we will 
see (Theorem~\ref{th:main-SB}) that it enables us to give 
lower and upper bounds on the spectral gap of the SW dynamics 
in terms of the spectral gap of this Markov chain.
The aforementioned local behavior can be demonstrated by the 
following example. 
Fix $p\in(0,1)$, $q\in\N$ and let the graph be given 
by $G_1=\bigl(\{u,v\},\bigl\{\{u,v\}\bigr\}\bigr)$, 
i.e.~the graph with two vertices and one edge between them. 
Let $e=\{u,v\}$. Then it is easy to verify that the transition 
probabilities of the SW dynamics satisfy
$\tsw^{G_1}(\varnothing,e)=p/q$ and 
$\tsw^{G_1}(e,\varnothing)=1-p$. 
For this note that, 
following the steps ($\tilde{\text{SW}}$1) and 
($\tilde{\text{SW}}$2), if the current configuration is 
$\varnothing\subset \{e\}$ we assign independently a random color 
from $[q]$ to each of the vertices, and if their colors agree 
(which happens with probability $1/q$), we 
put an edge between them with probability $p$. 
Hence, $\tsw^{G_1}(\varnothing,e)=p/q$.
The second equality can be shown analogously. 
We use these transition probabilities to construct a 
local Markov chain for the RC model on arbitrary graphs, 
which we call the \emph{single-bond dynamics}, as follows. 
Let the current state be $A\subset E$, 
choose an edge $e\in E$ uniformly at random and decide if $e$ 
shall be included in the new configuration or not depending on 
the connectivity of the endvertices of $e$ in $A$. 
If the endvertices are not connected in $A$ include $e$ with 
probability $\tsw^{G_1}(\varnothing,e)$, otherwise include $e$ with 
probability $\tsw^{G_1}(e,e)$.
Note that in the heat-bath dynamics, we check 
connectivity in $A\setminus e$ instead of connectivity in~$A$.

The transition matrix of the \emph{single-bond} (SB) \emph{dynamics} 
\index{dynamics!single-bond} 
is given by
\begin{equation} \label{eq:SB}
\sb(A,B) \;=\;  \frac1{\abs{E}}\sum_{e\in E}\,\begin{cases}
p, & \text{ if } B=A\cup e \text{ and } 
				e^{(1)}\stackrel{A}{\longleftrightarrow}e^{(2)},\\
1-p, & \text{ if } B=A\setminus e \,\text{ and } 
				e^{(1)}\stackrel{A}{\longleftrightarrow}e^{(2)},\\
\frac{p}{q}, & \text{ if } B=A\cup e \text{ and } 
				e^{(1)}\nconn{A}e^{(2)},\\
1-\frac{p}{q}, & \text{ if } B=A\setminus e \,\text{ and } 
				e^{(1)}\nconn{A}e^{(2)}.
\end{cases}
\end{equation}
Ergodicity of $\sb$ is clear. 
For the reversibility with respect to $\mu$ it is enough 
to prove
\[
\frac{\sb(A\setminus e,A\cup e)}{\sb(A\cup e, A\setminus e)} 
\,=\, \frac{\mu(A\cup e)}{\mu(A\setminus e)} 
\,\stackrel{\eqref{eq:RC-quotient}}{=}\, 
\begin{cases}
\frac{p}{1-p}, & \text{ if } e^{(1)}\lconn{A\setminus e}e^{(2)}, \\
\frac{p}{q(1-p)}, & \text{ if } e^{(1)}\nconn{A\setminus e}e^{(2)}, \\
\end{cases}
\]
for every $A\in\Orc$ and $e\in E$.
First note that $e^{(1)}$ and $e^{(2)}$ are always 
connected in $A\cup e$, hence we divide by 
$\sb(A\cup e, A\setminus e)=1-p$ independently of  
$A$ and $e$ in the above equation. 
Additionally, the numerator of the above fraction is $p$ or 
$p/q$ depending on connectivity in $A\setminus e$, as desired.

As for the heat-bath dynamics for the RC model, the single-bond 
dynamics is also a valid local Markov chain for non-integer 
values of $q$, but here we have to assume $q\ge1$ in order to 
ensure that $\sb$ is well-defined.

In Section~\ref{sec:4_repr} we will present the usefulness of this 
dynamics by providing a common representation of this and 
the Swendsen--Wang dynamics on the joint (\jointname) model, 
using the same ``building blocks''. 
We finish this section with an inequality between 
the spectral gaps of $\sb$ and $\thb$.

\begin{lemma} \label{lemma:SB-HB}
For $\thb$ and $\sb$ for the random-cluster model with 
parameters $p$ and $q$ we have
\[
\left(1-p \left(1-\frac1q\right)\right)\, \lambda(\thb) 
\;\le\; \lambda(\sb) 
	\;\le\; \lambda(\thb). 
\]
\end{lemma}
\vspace{4mm}

\begin{proof}
First we show that $\thb$ and $\sb$ have only non-negative 
eigenvalues. For this write 
$\thb=\frac{1}{\abs{E}}\sum_{e\in E}\Pt_e$ with 
\[
\Pt_e(A,B) \,:=\, 
	\frac{\mu(B)}{\mu(A)+\mu(A^e)}\;\ind(B\setminus e = A\setminus e), 
	\qquad A,B\in\Orc 
\]
(see \eqref{eq:RC-HB-def}).
Obviously, $\Pt_e$, $e\in E$, is reversible with respect to $\mu$ 
and satisfies $\Pt_e^2=\Pt_e$, since the distributions 
$\Pt_e(A\setminus e,\cdot)$ and $\Pt_e(A\cup e,\cdot)$ are equal. 
This shows that all $\Pt_e$, $e\in E$, are projections (see 
\cite[Thm.~9.5-1]{Krey}), and thus positive, 
i.e.~$\l\Pt_e g, g\r_\mu\ge0$ for all $g\in L_2(\mu)$ 
(see \cite[Thm.~9.5-2]{Krey}). 
Using the fact that the sum of positive operators is positive we 
obtain positivity of $\thb$. It follows that $\thb$ has only 
non-negative eigenvalues (see \cite[Obs.~7.1.4]{HJ-matrix}).
Similar arguments lead to the same statement for $\sb$ 
(see Remark~\ref{remark:SB-SW_pos-def}).
By Lemma~\ref{lemma:prelim_comparison}, and since $\thb$ and $\sb$ 
have the same stationary distribution, it 
is sufficient to show
\[
\left(1-p \left(1-\frac1q\right)\right)\, \thb(A,B) 
\;\le\; \sb(A,B) 
\;\le\; \thb(A,B)
\]
for $A\neq B\in\Orc$, i.e.~for $B=A^e$ for 
some $e\in E$ (otherwise $\thb(A,B)=\sb(A,B)=0$). 
By reversibility we obtain 
\[
\frac{\thb(A\setminus e,A\cup e)}{\sb(A\setminus e,A\cup e)} 
\,=\, \frac{\thb(A\cup e,A\setminus e)}{\sb(A\cup e,A\setminus e)} 
\,\stackrel{\eqref{eq:SB}}{=}\, 
\frac{\abs{E}}{1-p}\; \thb(A\cup e,A\setminus e).
\]
Since 
$\frac{1-p}{\abs{E}} \le \thb(A\cup e,A\setminus e) 
\le \frac{1}{\abs{E}}\,\frac{q(1-p)}{p+q(1-p)}$ 
for $q\ge1$ it follows 
\begin{equation}\label{eq:proof_lemma:SB-HB}
1 \,\le\, \frac{\thb(A,A^e)}{\sb(A,A^e)} \,\le\, \frac{q}{p+q(1-p)} 
\,=\, \frac{1}{1-p(1-q^{-1})}
\end{equation}
for all $A\in\Orc$ and $e\in E$. 
\end{proof}

\vspace{2mm}


\section{Known results} \label{sec:2_known-results}

In this section we present a selection of known results 
on the mixing properties for the above introduced algorithms. 
In fact, for the heat-bath dynamics for the Potts model, 
we will state only results 
that are needed for the further analysis. 
For the Swendsen--Wang dynamics and the local Markov chains 
for the random-cluster model we try to give a 
complete overview of the known results.

Since, in the original papers, some results are given in terms of 
spectral gap and some with mixing times, we first state a 
corollary to Lemma~\ref{lemma:mixing-gap} that we need for 
translation.

\begin{corollary} \label{coro:mixing-gap}
Let $P$ $($resp.~$\Pt$$)$ be the transition matrix of a reversible, 
ergodic Markov chain for the $q$-state Potts 
$($resp.~random-cluster$)$ model on a graph $G=(V,E)$ at inverse 
temperature $\beta$ $($resp.~with parameters $p$ and $q$$)$. 
Then
\[
\lambda(P)^{-1}-1 \;\le\; t_{\rm mix}(P) 
\;\le\; \Bigl(2 + \beta\abs{E} + \abs{V} \log q\Bigr)\,
			\lambda(P)^{-1} 
\]
and
\[
\lambda(\Pt)^{-1}-1 \;\le\; t_{\rm mix}(\Pt) 
\;\le\; \left(2 + \abs{E} \log\frac{1}{p(1-p)} + \abs{V} \log q\right)\,
			\lambda(\Pt)^{-1}. 
\]
\end{corollary}

In particular, this shows that every result of this paper can 
also be written in terms of the mixing time, 
loosing the same factor as in Corollary~\ref{coro:mixing-gap}. 

We begin with the probably most studied instance: 
the heat-bath dynamics for the $q$-state Potts model 
on the two-dimensional square lattice.
For $d\ge1$, define the $d$-dimensional {\it hypercubic lattice} $\Z^d_L$ 
of side length $L$ as the graph $\Z^d_L=(V_{L,d},E_{L,d})$ with 
vertex set $V_{L,d}=\{1,\dots,L\}^d\subset\Z^d$ and 
edge set $E_{L,d}=\bigl\{\{u,v\}\subset V_{L,d}:\,\abs{u-v}=1\bigr\}$, 
where $\abs{\,\cdot\,}$ denotes the Euclidean norm. 
For $d=2$ we call $\Z_L^2$ the two-dimensional square lattice 
of side length $L$.
In this case ($d=2$) there is an almost complete characterization 
of the spectral gap, that was established over the last decades, 
in particular for the Ising model ($q=2$). 
Beginning with the work of Holley \cite{Holley}, 
Aizenman and Holley \cite{AizHol} and Stroock and Zegarli{\'n}ski 
\cite{SZ_lattice}, who showed rapid mixing of the heat-bath dynamics 
given some spatial mixing property, which is called 
\emph{complete analyticity} (or {\it Dobrushin--Shlosman mixing condition}), 
see \cite{DobrShlos_description}, it was finally proven by 
Martinelli and Olivieri \cite{MO1,MO2} that the heat-bath dynamics 
for the Ising model is rapidly mixing up to the \emph{critical 
temperature}, i.e.~if the inverse temperature $\beta$ 
satisfies $\beta<\beta_c(2)=\log(1+\sqrt{2})$.
Using results of Cesi, Chayes, Chayes, Guadagni, Martinelli and 
Schonmann \cite{CCS,CGMS,Schonmann} it is known that this is (almost) 
best possible in the sense that the spectral gap of the HB dynamics 
on $\Z_L^2$ at inverse temperature $\beta$ is smaller than 
$\exp(-c L)$, for some $c>0$, if $\beta>\beta_c(2)$.
Only recently has rapid mixing at the critical temperature $\beta_c$ 
been proven by Lubetzky and Sly \cite{LS}.
For the proof of rapid mixing of the heat-bath dynamics 
for the Potts model at all $\beta<\log(1+\sqrt{q})$ we need 
\emph{exponential decay of connectivities} in the RC model 
(see Beffara and Duminil-Copin~\cite{BDC}).  
This implies \emph{weak mixing} in the Potts model 
(see Alexander~\cite{A}), and thus 
rapid mixing of the heat-bath dynamics 
(see Martinelli, Olivieri and Schonmann~\cite{MOS}).

Before we summarize these results in Theorem~\ref{th:P-HB_2d}, we 
introduce a variant of the Potts measure \eqref{eq:Potts} and 
the heat-bath dynamics \eqref{eq:P-HB} with (constant) 
boundary condition. 
For this define the \emph{boundary}\index{boundary ($\Z_L^d$)} 
of $V_{L,d}$ by 
\[
\partial V_{L,d} \,:=\, \Bigl\{v=(v_1,\dots,v_d)\in V_{L,d}:\, 
v_i\in\{1,L\} \text{ for some } i\in[d]\Bigr\}
\]
and let $d_{L,d}^{+}(v)$ be the number of neighbors of $v$ in 
$\Z^d\setminus V_{L,d}$, i.e., 
\[
d_{L,d}^{+}(v) \,:=\, 
\abs{\bigl\{u\in\Z^d\setminus V_{L,d}:\ \abs{v-u}=1\bigr\}}.
\]
The
Potts measure on $\Z_L^d$ with \emph{$1$-boundary condition} 
is defined by 
\begin{equation}\label{eq:Potts_boundary}
\pi^{\Z_L^d,1}_{\beta,q}(\sigma) 
\,:=\, \bar{Z}^{-1}\,\pi^{\Z_L^d}_{\beta,q}(\sigma)\,
\prod_{v\in\partial V_{L,d}} \exp\bigl(\beta\, 
d_{L,d}^{+}(v)\,\ind(\sigma(v)=1)
\bigr),
\quad \sigma\in\Op(\Z_L^d), 
\end{equation}
where 
$\bar{Z}$ is the proper 
normalization constant and $\pi^{\Z_L^d}_{\beta,q}$ is defined as 
in \eqref{eq:Potts}.
This measure can be interpreted as the conditional distribution 
of the configurations on $V_{L,d}$ given that all vertices of 
$\Z^d\setminus V_{L,d}$ have color~1. 

\begin{remark}\label{remark:critical}
The critical inverse temperature $\beta_c(d,q)$ 
for the $q$-state Potts model on~$\Z^d$, 
that was cited above for $d=2$, is generally defined by
\begin{equation}\label{eq:critical}
\beta_c(d,q) \,:=\, \inf\{\beta:\, M_{d,q}(\beta)>0\},
\end{equation}
where 
\[
M_{d,q}(\beta) \,:=\, \lim_{L\to\infty}\, 
\frac{1}{\abs{V_{L,d}}}\,\sum_{v\in V_{L,d}}
\left(\pi^{\Z_L^d,1}_{\beta,q}
\Bigl(\{\sigma:\,\sigma(v)=1\}\Bigr)-\frac1q\right).
\]
(It is well-known that these limits exist; see e.g.~Grimmett~\cite{G1}.)
We write $\beta_c(q)$ for $\beta_c(2,q)$. 
A closed formula for $\beta_c(q)$ was first established by 
Onsager \cite{Onsager} in the case $q=2$ by giving 
an explicit formula for $M_{2,2}(\beta)$. A proof of the equality 
$\beta_c(q)=\log(1+\sqrt{q})$ for all $q\ge2$, 
which was expected to be true, has been given only 
recently by Beffara and Duminil-Copin \cite{BDC}. 
For $d\ge3$ it is still a challenging open problem to give 
an explicit formula for the critical inverse temperature.
However, it is known (see Laanait~et~al.~\cite{LaanaitMMRS}) that 
\[
\beta_c(d,q) \,=\, \frac{1}{d} \log q + \mathcal{O}(q^{-1/d}), 
\]
for $q$ large enough.
\end{remark}

Let $\pi^1:=\pi^{\Z_L^2,1}_{\beta,q}$. 
Similarly to \eqref{eq:P-HB} we define the transition matrix of the 
heat-bath dynamics for the Potts 
model on $\Z_L^2$ with 1-boundary condition by
\begin{equation}\label{eq:P-HB_boundary}
P_{{\rm HB},1}(\sigma,\tau) 
\;:=\; P_{{\rm HB},1,\beta,q}^{\Z_L^2}(\sigma,\tau) \;=\;
\frac{1}{\abs{V_{L,2}}}\sum_{v\in V_{L,2}}\,\frac{\pi^1(\tau)}
	{\sum_{l=1}^q\pi^1(\sigma^{v,l})}\,
	\ind(\tau\in\O_{\sigma,v})
\end{equation}
for all $\sigma,\tau\in\Op(\Z_L^2)$.
We summarize the rapid mixing results stated above in the following 
theorem.

\begin{theorem}\label{th:P-HB_2d}
Let $\hb$ be the transition matrix of the heat-bath dynamics for 
the \mbox{$q$-state} Potts model on $\Z_L^2$ at inverse temperature $\beta$. 
Let $n=L^2=\abs{V_{L,2}}$. Then there exist constants 
$c_{\beta}=c_\beta(q)>0$ and $C<\infty$ such that
\begin{align*}
\lambda(\hb) \;&\ge\; \frac{c_\beta}{n} \qquad\;\;\, 
	\text{ for } \beta < \beta_c(q) \hspace*{2cm}\\[-3mm]
\intertext{\vspace*{-1mm} and}
\lambda(\hb) \;&\ge\; n^{-C} \qquad 	
	\text{ for } q=2 \text{ and } \beta = \beta_c(2),
\end{align*}
where $\beta_c(q)\,=\,\log(1+\sqrt{q})$. 
These bounds hold also if we replace $\hb$ by $P_{{\rm HB},1}$.
\end{theorem}

\begin{proof}
The results, as originally given in \cite{LS,MO1}, refer to a 
continuous-time Markov process for the Potts model. 
See e.g.~\cite{M} for an introduction to the 
``graphical construction'' of the continuous-time heat-bath 
dynamics for the Ising model. 
In fact, these papers present lower bounds on 
${\rm gap}(\Z^2_L)$, which is defined by 
\[
{\rm gap}(\Z^2_L) \,:=\, 
\inf_{\substack{f\in L_2(\pi):\\ \Var_\pi(f)=1}}\,
\frac12 \sum_{\sigma\in\Op}\sum_{v\in V_{L,2}}\sum_{k=1}^q \,
\pi(\sigma)\; 
\frac{\pi(\sigma^{v,k})}{\sum_{l=1}^q \pi(\sigma^{v,l})}\,
\left(f(\sigma^{v,k})-f(\sigma)\right)^2 
\]
(cf.~\cite[Sec.~3]{MOS}). By the variational characterization 
of the eigenvalues of reversible transition matrices 
(see e.g.~\cite{DS} or the proof of 
Lemma~\ref{lemma:prelim_comparison}) we can write 
\[
1-\xi_2 \,=\, \inf_{\substack{f\in L_2(\pi):\\ \Var_\pi(f)=1}}\,
\frac12 \sum_{\sigma\in\Op}\sum_{\tau\in\Op}\,
\pi(\sigma)\, \hb(\sigma,\tau)\,
\left(f(\tau)-f(\sigma)\right)^2,
\]
where $\xi_2$ is the second largest eigenvalue of $\hb$. 
By definition, ${\rm gap}(\Z^2_L)=n (1-\xi_2)$ (see \eqref{eq:P-HB}). 
But $\hb$ has in general 
only non-negative eigenvalues, and thus $\lambda(\hb)=1-\xi_2$. 
To see this, write $\hb=\frac{1}{\abs{V}}\sum_{v\in V} P_v$ with 
$P_v(\sigma,\tau)=\pi(\tau\mid\O_{\sigma,v})$ (cf.~\eqref{eq:HB-cond}), 
and note that $P_v^2=P_v$ since $\O_{\sigma,v}=\O_{\tau,v}$ for 
all $\sigma,\tau$ with $P_v(\sigma,\tau)>0$. 
That is, $P_v$ is a projection (see \cite[Thm.~9.5-1]{Krey}).
It follows that all $P_v$, $v\in V$, and thus $\hb$, 
have only non-negative eigenvalues by \cite[9.5-2]{Krey}.
Similar ideas were used in the proof of Lemma~\ref{lemma:SB-HB}. 

Hence, it is enough to show that there exist constants 
$\tilde{c}_\beta>0$ and $\tilde{C}<\infty$ such that 
\begin{itemize}
\item\quad ${\rm gap}(\Z^2_L) \;\ge\; \tilde{c}_\beta$ 
			\qquad\, for $\beta < \beta_c(q)$,\vspace{1mm}
\item\quad ${\rm gap}(\Z^2_L) \;\ge\; n^{-\tilde{C}}$ 
			\quad\; for $q=2$ and $\beta = \beta_c(2)$.
\end{itemize}
The second inequality is proven in \cite[Thm.~4.2]{LS}. 
By \cite[Thm.~3.2]{MOS} the first inequality is equivalent to 
a weak mixing property of the Potts measure 
(see \cite[eq.~(1.11)]{MOS}). 
This weak mixing property is shown \cite[Thm.~3.6]{A} to hold 
whenever correlations 
decay exponentially or, equivalently, we have exponential decay 
of connectivities in the corresponding infinite-volume 
random-cluster model, i.e., for all $u,v\in \Z^2$ we have 
\[
\lim_{L\to\infty}\,\mu_{p,q}^{\Z_L^2}
	\bigl(\{A\subset E_{L,2}:\, u\conn{A}v\}\bigr) 
\,\le\, \a_1\, e^{-\a_2 \abs{u-v}}
\]
with some $0<\a_1(p,q),\a_2(p,q)<\infty$ 
and Euclidean norm $\abs{\,\cdot\,}$.
This was proven by Beffara and Duminil-Copin~\cite[Thm.~2]{BDC} 
for all $q\ge1$ and $p<p_c(q):=1-e^{-\beta_c(q)}$. 
Furthermore, the statements of this theorem hold true if we consider 
the case of \mbox{1-boundary} condition. For this see 
\cite[Thm.~1]{LS} and note that the result of \cite[Thm.~3.2]{MOS} 
holds for arbitrary boundary conditions. This proves the theorem.
\end{proof}


\begin{remark}
As stated above it is known that the heat-bath dynamics 
for the Ising model (without boundary conditions) is slow mixing 
if $\beta>\beta_c(2)$. Additionally, we are only aware of a result 
that shows an exponentially small upper bound on the spectral gap 
if $\beta>\beta_c(q)$ and $q$ is large enough for some specific 
(periodic) boundary condition; 
see \cite[Thm.~1.2]{BCT} or Theorem~\ref{th:P-SW-HB_torus} below. 
However, it is reasonable to believe that the HB dynamics is slowly 
mixing for all $q\ge2$ and $\beta>\beta_c(q)$ on $\Z_L^2$ 
without (or with periodic) boundary condition.
\end{remark}

We now turn to another class of underlying graphs, namely to 
rectangular subsets of the hypercubic lattice $\Z^d$. 
In fact, we consider only the case of \emph{periodic} 
boundary condition. For this consider the 
\emph{cycle}\index{cycle} 
$C_L$ of length $L$, that is the 
graph $C_L=(\{1,\dots,L\},\tilde{E}_L)$ with 
$\tilde{E}_L:=\bigl\{\{v,v+1\}:\,v\in\{1,\dots,L-1\}\bigr\}
\cup\{1,L\}$ and define, for two graphs $G_1=(V_1,E_1)$ and 
$G_2=(V_2,E_2)$, the 
\emph{graph product}\index{graph!product}
of $G_1$ and $G_2$, 
written $G_1\times G_2$, as the graph with vertex set 
$V_1\times V_2$ and $(u_1,u_2),(v_1,v_2)\in V_1\times V_2$ 
are neighbors in $G_1\times G_2$ iff either $u_1$ and $v_1$ 
are neighbors in $G_1$ and $u_2=v_2$ or $u_2$ and $v_2$ are 
neighbors in $G_2$ and $u_1=v_1$.
We then define the \emph{$d$-dimensional torus}\index{torus} 
$\Zt_L^d$ 
of side length $L$ by the $d$-fold graph product
\begin{equation}\label{eq:def-torus}
\Zt_L^d \,:=\, C_L^d = C_L\times \dots \times C_L.
\end{equation}
From Borgs, Chayes and Tetali \cite{BCT} we obtain the following 
theorem.

\begin{theorem}\label{th:P-SW-HB_torus}
Let $\hb$ $($resp.~$\sw$$)$ be the transition matrix of the heat-bath 
$($resp. Swendsen--Wang$)$ dynamics for the $q$-state Potts model 
on $\Zt_L^d$, $d\ge2$, at inverse temperature $\beta$. 
Then there exist constants $k_1,k_2<\infty$ 
and a constant $k_3>0$ 
$($all depending on $d$, $\beta$ and $q$$)$
such that, for $q$ and $L$ large enough, 
\begin{align*}
e^{-(k_1 + k_2 \beta) L^{d-1}} \,&\le\, \lambda(\hb) 
	\,\le\, e^{-k_3 \beta L^{d-1}} \qquad 
	\text{ for all } \beta\ge\beta_c(d,q) \\[-3mm]
\intertext{\vspace*{-1mm} and}
e^{-(k_1 + k_2 \beta) L^{d-1}} \,&\le\, \lambda(\sw) 
	\,\le\, e^{-k_3 \beta L^{d-1}} \qquad 	
	\text{ for } \beta=\beta_c(d,q)
\end{align*}
with $\beta_c(d,q)$ from \eqref{eq:critical}. 
In fact, the lower bounds hold for all $\beta$, $q$ and $L$. 
\end{theorem}

This theorem shows (at least for large $q$) that the 
heat-bath dynamics is slowly mixing at and below the critical 
temperature and, additionally, that also the Swendsen--Wang dynamics 
has an exponentially small spectral gap at the 
critical temperature if $q$ is large enough.
We will see in Chapter~\ref{chap:bond} that an analogous 
result holds for the single-bond dynamics for the 
random-cluster model.

Before we discuss other results for the Swendsen--Wang 
dynamics, we state a result for the HB dynamics on 
a more general class of graphs. 
For this, fix a graph $G=(V,E)$ and define the 
\emph{adjacency matrix}\index{adjacency matrix} $A_G$ 
of $G$ by $A_G(u,v):=\ind(\{u,v\}\in E)$, $u,v\in V$.
Additionally, write $\norm{A_G}$ for its (unweighted) 
operator norm, i.e.~$\norm{A_G}=\max_{\abs{x}=1}\abs{A_G x}$, 
where the maximum is taken over $x\in\R^{\abs{V}}$ and 
$\abs{\,\cdot\,}$ is the Euclidean norm.
In the literature $\norm{A_G}$ is called the 
\emph{principal eigenvalue} of the graph~$G$.
The following theorem is based on a result of Hayes~\cite{Ha}.

\begin{theorem}\label{th:HB_degree}
The heat-bath dynamics for the $q$-state Potts model  
at inverse temperature $\beta$
on a graph $G$ with $n$ vertices 
satisfies
\[
\lambda(\hb) \,\ge\, \frac{1-\eps}{n}
\]
whenever $\beta\,\le\,2\eps\norm{A_G}^{-1}$.
\end{theorem}

\begin{proof}
Note first that one key ingredient for the result is Observation~11 
of \cite{Ha}. The result is stated only for the Ising model ($q=2$), 
but it can be generalized quite easily by induction on $q$.
We state the proof here for completeness.
For this, we have to show, for all $u,v\in V$, that 
\[
\rho_{u,v} \,\le\, \tanh\left(\frac{\beta}{2}\right)\,A_G(u,v),
\]
where $\rho_{u,v}$, i.e. the \emph{influence of $v$ on $u$}, 
is defined by 
\[\begin{split}
\rho_{u,v} \,&:=\, 
\max_{\substack{\sigma\in\Op\\ \tau\in\O_{\sigma,v}}}\,
\frac12 \sum_{k=1}^q \abs{
	\frac{\pi(\sigma^{u,k})}{\sum_{l=1}^q\pi(\sigma^{u,l})} -
	\frac{\pi(\tau^{u,k})}{\sum_{l=1}^q\pi(\tau^{u,l})}} \\
&=\, \max_{\substack{\sigma\in\Op\\ \tau\in\O_{\sigma,v}}}\,
\frac12 \sum_{k=1}^q \abs{
	\frac{e^{\beta d_{u,k}(\sigma)}}
		{\sum_{l=1}^q e^{\beta d_{u,l}(\sigma)}} -
	\frac{e^{\beta d_{u,k}(\tau)}}
		{\sum_{l=1}^q e^{\beta d_{u,l}(\tau)}}},
\end{split}\]
see \cite[Def.~4]{Ha}.  
As before, $\O_{\sigma,v}$ is the set of configurations that differ from 
$\sigma$ only at $v$ (cf.~(HB2) on page~\pageref{page:P-HB}), 
and $d_{u,k}(\sigma)$ is the number 
of neighbors of $u$ in $G$ with color $k$ in~$\sigma$.
Note that in \cite{Ha} the above bound is stated with $\beta$ 
in place of $\beta/2$. This comes from the difference 
in the normalization of the measure. 

Obviously, $\rho_{u,v}=0$ if $\{u,v\}\notin E$ since the term 
inside the absolute value depends only on the colors of the 
neighbors of $u$, which are equal if $v$ is neither of them.
Now fix some neighbors $u,v\in V$ and configurations 
$\sigma,\tau\in\Op$ with $\tau\in\O_{\sigma,v}$.
Let $r_k:=d_{u,k}(\sigma)$. Since $\sigma$ and $\tau$ differ only 
at one neighbor of $u$, there exist $i,j\in[q]$ such that 
$d_{u,i}(\tau)=r_i+1$, $d_{u,j}(\tau)=r_j-1$ and 
$d_{u,k}(\tau)=r_k$ for all $k\neq i,j$. 
Assume without loss of generality that $i=1$ and $j=2$. Then some simple 
calculations show that
\[
\frac12 \sum_{k=1}^q \abs{
	\frac{e^{\beta d_{u,k}(\sigma)}}
		{\sum_{l=1}^q e^{\beta d_{u,l}(\sigma)}} -
	\frac{e^{\beta d_{u,k}(\tau)}}
		{\sum_{l=1}^q e^{\beta d_{u,l}(\tau)}}}
\,=\, \frac{N_q(r)}{D_q(r)}, \qquad r=(r_1,\dots,r_q),
\]
where
\[
N_q(r) \,:=\, (e^\beta-e^{-\beta}) e^{\beta(r_1+r_2)} + 
	\max\Bigl\{(e^\beta-1) e^{\beta r_1},(1-e^{-\beta}) 
	e^{\beta r_2}\Bigr\} \sum_{k=3}^q e^{\beta r_k}
\]
and
\[
D_q(r) \,:=\, \left(\sum_{k=1}^q e^{\beta r_k}\right)\,
	\left(e^{\beta(r_1+1)}+e^{\beta(r_2-1)}
	+\sum_{l=3}^q e^{\beta r_l}\right).
\]
We will prove that $N_q(r)/D_q(r)\le\tanh(\beta/2)$ 
for all $r\in\R^q$ and $q\in\N$ by induction. 
Let us first recall the $q=2$ case from Observation~11 of \cite{Ha}. 
In this case the last sum in the definition of $N_2$ and $D_2$ 
disappears. Thus, for $r\in\R^2$, 
\[\begin{split}
\frac{N_2(r)}{D_2(r)} \,&=\, 
\frac{(e^\beta-e^{-\beta}) e^{\beta(r_1+r_2)}}
	{\left(e^{\beta r_1}+e^{\beta r_2}\right)
	\left(e^{\beta(r_1+1)}+e^{\beta(r_2-1)}\right)}
\,=\, \frac{e^\beta-e^{-\beta}}
	{e^\beta+e^{-\beta} + e^{\beta(r_1-r_2+1)}+e^{-\beta(r_1-r_2+1)}}\\
&\le\, \frac{e^\beta-e^{-\beta}}
	{e^\beta+e^{-\beta} + 2} 
\,=\, \tanh\left(\frac{\beta}{2}\right),
\end{split}\]
where the last inequality comes from $e^x+e^{-x}\ge2$, $x\in\R$.
Now assume that the statement holds for $q-1$ and 
all $s=(r_1,\dots,r_{q-1})\in\R^{q-1}$ and let 
$r=(r_1,\dots,r_{q-1},r_q)$ for 
some $r_q\in\R$. We obtain
\[\begin{split}
\frac{N_q(r)}{D_q(r)} \,&=\, 
\frac{N_{q-1}(s) + e^{\beta r_q} 
		\max\bigl\{(e^\beta-1) e^{\beta r_1}, 
		(1-e^{-\beta})	e^{\beta r_2}\bigr\}}
	{D_{q-1}(s) + \left(e^{\beta}+1\right) e^{\beta (r_1+r_q)}
		+\left(1+e^{-\beta}\right) e^{\beta (r_2+r_q)} 
		+ e^{2\beta r_q} + 2e^{\beta r_q}
			\sum_{k=3}^{q-1} e^{\beta r_k}}\\
\,&\le\, \frac{N_{q-1}(s) + e^{\beta r_q} 
		\max\bigl\{(e^\beta-1) e^{\beta r_1}, 
		(1-e^{-\beta})	e^{\beta r_2}\bigr\}}
	{D_{q-1}(s) + e^{\beta r_q} [\left(e^{\beta}+1\right) e^{\beta r_1}
		+\left(1+e^{-\beta}\right) e^{\beta r_2}]} 
\,\le\, \tanh\left(\frac{\beta}{2}\right),
\end{split}\]
since $a/b\le t$ and $c/d\le t$ imply $(a+c)/(b+d)\le t$ for 
$a,b,c,d,t\ge0$. This proves Observation~11 of \cite{Ha} 
for all $q\in\N$ and thus, under the assumptions of this theorem, 
we find by \cite[Thm.~6]{Ha} that
\[
\max_{\sigma\in\Op(G)}\tvd{\hb^t(\sigma,\cdot)-\pi} 
\,\le\, n\left(1-\frac{1-\eps}{n}\right)^t.
\]
Using Lemma~\ref{lemma:TV-gap} we obtain the result. 
\end{proof}

\vspace{3mm}

In particular we have the following corollary (see~\cite{Ha}).

\begin{corollary}\label{coro:P-HB_degree}
The heat-bath dynamics for the $q$-state Potts model  
at inverse temperature $\beta$
on a graph $G$ with $n$ vertices and maximum degree $\D$
satisfies
\[
\lambda(\hb) \,\ge\, \frac{1-\eps}{n}
\]
if $\beta\,\le\,\frac{2\eps}{\D}$.
\end{corollary}

\begin{proof}
Using \cite[Thm.~8.1.22]{HJ-matrix} we obtain
$\norm{A_G} \le \max_{u\in V} \sum_{v\in V} A_G(u,v) = \D(G)$.
Thus, the result follows from Theorem~\ref{th:HB_degree}.
\end{proof}

In \cite{Ha} one can find also an improvement of this 
corollary if we restrict to the class of \emph{planar} graphs. 
We will use this in Chapter~\ref{chap:2d} to obtain a result 
for the Swendsen--Wang dynamics on planar graphs (see 
Corollary~\ref{coro:SW-planar}).

\begin{remark}
There are, of course, a lot of other results concerning 
mixing properties of the heat-bath dynamics for the Potts model. 
These include, e.g., rapid mixing for heat-bath dynamics on trees 
at all temperatures \cite{BKMP}. 
Additionally, only recently was the complete picture of rapid mixing 
or lack thereof established for the HB dynamics on the 
\emph{complete graph} \cite{CDLLPS,LLP}.
\end{remark}

Now we want to state some 
known results on the mixing properties of the Swendsen--Wang 
dynamics. We try to give an overview of all known results, but 
we omit results that involve boundary conditions 
(or external magnetic field) (see e.g.~\cite{M2,MOScopp1,MOScopp2}), 
and results where the underlying graph is random (see \cite{CF}).

We start with the known results on the complete graph, where 
at least for $q=2$ the mixing behavior is completely known.
Denote by $K_n$ the 
\emph{complete graph}\index{graph!complete} 
on $n$ vertices, 
i.e.~$K_n=\bigl([n],\genfrac(){0pt}{1}{[n]}{2}\bigr)$,
where $\genfrac(){0pt}{1}{[n]}{2}$ is the set of all two-element 
subsets of $[n]$.
Note that, for two non-negative functions $f,g: \N\to\R$, we write 
$f(n)=\Theta(g(n))$ iff $0<\lim_{n\to\infty}\frac{f(n)}{g(n)}<\infty$.
The result below is due to Long, Nachmias and Peres \cite{LNP,Long} 
in the case $q=2$ (see also \cite{CDFR}). For $q\ge3$ the results are 
adopted from Gore and Jerrum~\cite{GJ} and Huber \cite{Hu}.

\vspace{1mm}
\begin{theorem} \label{th:SW_complete}
Let $\sw$ be the transition matrix of the Swendsen--Wang dynamics for 
the $q$-state Potts model on the complete graph $K_n$ at inverse 
temperature $\beta=\log(\frac{n}{n-c})$, $c\ge0$ 
(or $p=1-e^{-\beta}=\frac{c}{n}$).
Then, with $t_{\rm mix}:=t_{\rm mix}(\sw)$, 
\begin{enumerate}
	\item[$(i)$] for $q=2$ and 
		\vspace*{-1mm}
		\begin{itemize}
			\item $c<2$:\; $t_{\rm mix} \,=\, \Theta(1)$. \vspace{1mm}
			\item $c=2$:\; $t_{\rm mix}\,=\, \Theta(n^{1/4})$. \vspace{1mm}
			\item $c>2$:\; $t_{\rm mix} \,=\, \Theta(\log(n))$.
		\end{itemize}
	\item[$(ii)$] for $q\ge3$ and 
		\vspace*{-1mm}
		\begin{itemize}
			\item $c=\frac{2(q-1)\log(q-1)}{q-2}$:\;\; 
					$t_{\rm mix} \,\ge\, e^{\eps \sqrt{n}}, \qquad
						\text{ for some } \eps>0$. \vspace{2mm}
			\item \quad $c<\frac13$:\qquad\qquad\quad
					$t_{\rm mix}\,\le\, C \log(n), \;
						\text{ for some } C<\infty$. 
			\item \;\;$c>2q\log(3q n)$:\quad\quad 
					$t_{\rm mix} \,\le\, C q  n, \qquad
						\text{ for some } C<\infty$.
		\end{itemize}
\end{enumerate}
\end{theorem}

As far as we know, Theorems~\ref{th:P-SW-HB_torus} and~\ref{th:SW_complete}(ii) 
contain the only presently known 
slow mixing results for the Swendsen--Wang dynamics 
(except for results on random graphs). 
Other classes of graphs where rapid mixing of SW is known at 
all temperatures, but now for every $q\in\N$, are 
trees and cycles (see Cooper and Frieze \cite{CF} and 
Long \cite{Long}).

\begin{theorem} \label{th:SW_tree-cycle}
Let $T$ be a tree on $n$ vertices and $C_n$ be a cycle of length $n$.
Then
\begin{itemize}
	\item $t_{\rm mix}(P_{{\rm SW},\beta,q}^T) \,=\, \Theta(\log(n))$ 
					\; and \vspace{2mm}
	\item $t_{\rm mix}(P_{{\rm SW},\beta,q}^{C_n})\,\le\,c\, n\log(n)$ 
					\quad for some $c<\infty$. \vspace{1mm}
\end{itemize}
In fact, we have 
$\lambda(P_{{\rm SW},\beta,q}^T)=1-p(1-\frac1q)$.
\end{theorem}

For the statement on the spectral gap consider the construction of 
\cite[Chap.~7]{Long} and Exercise~12.7 of \cite{LPW}.
The last result that we want to present here for the SW dynamics is, 
similarly to Corollary~\ref{coro:P-HB_degree}, a result on 
graphs of bounded maximum degree from \cite{Hu} (see also \cite{CF}).

\begin{theorem}\label{known-th:SW_degree}
The Swendsen--Wang dynamics for the $q$-state Potts model  
at inverse temperature $\beta$ on a graph $G$ with $n$ vertices 
and maximum degree $\D$ satisfies
\[
t_{\rm mix}(\sw) \,\le\, C \log(n), \quad \text{ for some } C<\infty.
\]
whenever $\beta\,\le\,\frac{1}{3(\D-1)}$.
\end{theorem}

Now we turn to a result for the single-bond (or heat-bath) 
dynamics for the random-cluster model. 
In fact, this is the only result on the mixing 
properties of this Markov chain that we are aware of.

For this define the 
\emph{linear width}\index{linear width} 
of a graph $G=(V,E)$ as the 
smallest number $\ell$ such that there exists an ordering 
$e_1,\dots,e_{\abs{E}}$ of the edges with the property that for 
every $i\in[\abs{E}]$ there are at most $\ell$ vertices that are 
an endvertex of both an edge in $\{e_1,\dots e_i\}$ and an edge in 
$\{e_{i+1},\dots e_{\abs{E}}\}$. See \cite{GeS} for bounds on the 
linear width of paths, cycles, trees and a bound in terms of a 
related quantity, the tree width.

The following result is due to Ge and {\v{S}}tefankovi{\v{c}} 
\cite{GeS} (we only state the $q\ge1$ case).

\begin{theorem} \label{th:SB_linear-width}
Let $\sb$ be the transition matrix
of the single-bond dynamics for the random-cluster model 
with parameters $p$ and $q\ge1$ 
on a graph $G=(V,E)$ with linear width bounded by $\ell$. 
Let $m:=\abs{E}$. Then
\begin{equation} \label{eq:th-width}
\lambda(\sb)
\;\ge\; \frac{1}{2\, q^{\ell+1}}\ \frac{1}{m^2}.
\vspace{2mm}
\end{equation}
\end{theorem}
\begin{proof}
In \cite{GeS} the authors consider the (lazy) Metropolis version 
of the single-bond dynamics. This Markov chain has transition 
probabilities 
\[
\Pt_{\rm M}(A,A^e) \;=\;	\frac1{2\abs{E}}
	\min\left\{1,\,q^{c(A^e)-c(A)}
	\left(\frac{p}{1-p}\right)^{\abs{A^e}-\abs{A}}\right\}, 
	\quad A\subset E,
\]
with $\Pt_{\rm M}(A,A)$ such that $\Pt_{\rm M}$ is a stochastic 
matrix and $A^e$ from \eqref{eq:RC-Ae}. 
For this Markov chain they prove a lower bound on the congestion, 
which is defined as follows. Let $\G:=\{\g_{AB}:\, A,B\subset E\}$, 
where $\g_{AB}$ are paths from $A$ to $B$ in the 
graph $\mathcal{H}=(\O_{\rm RC},\mathcal{E})$ with 
$\mathcal{E}=\left\{(A,B):\,\Pt_{\rm M}(A,B)>0\right\}$. Then we 
define the \emph{congestion} of $\Pt_{\rm M}$ (with respect to~$\G$) by
\[
\varrho(\Pt_{\rm M},\G) \;:=\; 
	\max_{(B_1,B_2)\in\mathcal{E}}\,
	\frac{1}{\mu(B_1)\, \Pt_{\rm M}(B_1,B_2)}\,
	\sum_{A,C: (B_1,B_2)\in\g_{AC}} \abs{\g_{AC}} \mu(A)\, \mu(C),
\]
where $\abs{\g_{AC}}$ denotes the length of the path. 
The bound of \cite[Lemma~16]{GeS} is 
$
\varrho(\Pt_{\rm M},\G) \le 2\abs{E}^2 q^\ell\; 
$
for a suitable choice of $\G$ and so, by 
\cite[Prop. 1']{DS} (note that $\Pt_{\rm M}$ is lazy) that
\[
\lambda(\Pt_{\rm M})^{-1} \;\le\; 2\abs{E}^2\,q^\ell.
\]
But since it is easy to show that 
$\Pt_{\rm M}(A,B)\le q \,\sb(A,B)$ for all 
$A\neq B\subset E$ and 
that $\sb$ has only non-negative eigenvalues 
(see Remark~\ref{remark:SB-SW_pos-def}), 
we can conclude from Lemma~\ref{lemma:prelim_comparison} 
that
\[
\lambda(P_{\rm SB})^{-1} 
\;\le\; q\, \lambda(P_{\rm M})^{-1} 
\;\le\; 2\abs{E}^2\,q^{\ell+1}.
\]
\end{proof}

Finally, we want to mention an algorithm that allows approximate 
sampling in polynomial time from the Ising model on arbitrary graphs 
and at all temperatures. 
This algorithm is due to Randall and Wilson~\cite{RaWi99} 
and is based on the seminal work of Jerrum and Sinclair~\cite{JS}, 
which shows the first (and presently the only) polynomial-time algorithm 
to approximate the partition function of an arbitrary (ferromagnetic) Ising 
system. We are not aware of an explicit bound on the expected 
running time of this sampling procedure, but
there is a bound for the algorithm for the partition function in \cite{JS}. 
Maybe due to its simplicity, the Swendsen--Wang dynamics is, however, 
still the preferred algorithm in 
practice and, needless to say, it would be amazing to obtain rapid mixing 
of this Markov chain in the same regime.


\chapter{Comparison with single-spin dynamics}\label{chap:spin}



This chapter is based on \cite{U2}. 
We prove by comparison that the spectral gap 
of the Swendsen--Wang dynamics (SW) is bounded from below by 
some constant times the spectral gap of the heat-bath chain (HB). 
This result leads to rapid mixing of SW on graphs of bounded 
degree whenever HB mixes rapidly.

We will prove the following theorem, 
which is a minor improvement of \cite[Thm.~1]{U2}.

\begin{theorem} \label{th:main-spin}
Suppose that $\sw$ $($resp.~$\hb$$)$ is the transition matrix of the 
Swendsen--Wang $($resp.~heat-bath$)$ dynamics for the $q$-state Potts 
model at inverse temperature $\beta$ on a graph $G$ with 
maximum degree $\D$. 
Then
\[
\lambda(\sw) \;\ge\; \csw\,\lambda(\hb),
\]
where
\begin{equation}\label{eq:csw}
\csw \;=\; \csw(\D,\beta,q) 
\;:=\; q^{-1}\left(q\,e^{2\beta}\right)^{-2\D}.
\end{equation}
\vspace{-3mm} 
\end{theorem}

\begin{remark}
The inequality of Theorem~\ref{th:main-spin} is probably 
off by a factor of $\abs{V(G)}$,  
because we compare the SW dynamics with a Markov chain that changes 
only the color of one vertex of the graph per step. 
We conjecture that a bound of the form 
$\lambda(\sw) \;\ge\; c \abs{V(G)}\lambda(\hb)$, 
for some constant $c>0$, holds and that, in particular, 
this constant $c$ has a ``better'' dependence on the parameters involved. 
Unfortunately, this does not seem to be possible to show with our 
techniques.
\end{remark}

Before we prove Theorem~\ref{th:main-spin} we state some corollaries 
that can be deduced directly from the known results for the 
heat-bath dynamics (see Section~\ref{sec:2_known-results}).

The first corollary deals with the class of graphs with bounded 
maximum degree. It relies on a slight generalization of a result 
of Hayes \cite{Ha}, who gives a simple condition on~$\beta$, 
depending on the maximum degree, for rapid mixing 
of the heat-bath dynamics for the Ising model (see 
Theorem~\ref{th:HB_degree} and Corollary~\ref{coro:P-HB_degree}). 

\begin{corollary} \label{coro:SW_degree}
The Swendsen--Wang dynamics for the $q$-state Potts model  
at inverse temperature $\beta$ on a graph $G$ with $n$ vertices 
and maximum degree $\D$ satisfies
\[
\lambda(\sw)\;\ge\;\frac{\csw(1-\eps)}{n}
\] 
with $\csw=\csw(\D,\beta,q)$ from \eqref{eq:csw} and $\eps>0$, if
\[
\beta \,\le\, \frac{2\eps}{\D}.
\vspace{2mm}
\]
\end{corollary}

This result improves that of Huber~\cite{Hu} 
(see Theorem~\ref{known-th:SW_degree}) in the range of 
applicability, which was $\beta\le1/(3(\D-1))$ before, but while 
the result of \cite{Hu} is a logarithmic (in $n$) upper bound on 
the mixing time, Corollary~\ref{coro:SW_degree} together with 
Corollary~\ref{coro:mixing-gap} leads only to a quadratic bound.
%
We will see a further 
improvement of the above result in Chapter~\ref{chap:2d} 
if we consider only \emph{planar} graphs  
(see Corollary~\ref{coro:SW-planar}).

The second corollary, which we call a theorem because of its 
importance for the rest of this work, 
gives a bound on the spectral gap for 
the Swendsen--Wang dynamics on the square lattice. 
For this recall that the two-dimensional square lattice 
of side length~$L$ is the graph $\Z^2_L=(V_{L,2},E_{L,2})$ 
with vertex set $V_{L,2}=\{1,\dots,L\}^2\subset\Z^2$ and 
edge set $E_{L,2}=\bigl\{\{u,v\}\subset V_{L,2}:\,\abs{u-v}=1\bigr\}$  
(see Figure~\ref{fig:graphs}).

\vspace{2mm}
\begin{theorem}\label{th:P-SW_2d_high}
Let $\sw$ be the transition matrix of the Swendsen--Wang dynamics for 
the $q$-state Potts model on $\Z_L^2$ at inverse temperature $\beta$. 
Let $n=L^2=\abs{V_{L,2}}$. Then there exist constants 
$c_{\beta}=c_\beta(q)>0$ and $C<\infty$ 
such that
\begin{align*}
\lambda(\sw) \;&\ge\; \frac{c_\beta}{n} \qquad\qquad\; 
	\text{ for } \beta < \beta_c(q) \\[-5mm]
\intertext{\vspace*{-3mm} and}
\lambda(\sw) \;&\ge\; \csw\,n^{-C} \qquad 	
	\text{ for } q=2 \text{ and } \beta = \beta_c(2),
\end{align*}
with $\csw=\csw(4,\beta_c(2),2)$ from \eqref{eq:csw} 
and $\beta_c(q)\,=\,\log(1+\sqrt{q})$. 
\end{theorem}

This result is a consequence of Theorem~\ref{th:main-spin} and the 
corresponding result for the heat-bath dynamics (see 
Theorem~\ref{th:P-HB_2d} as well as the references that are given 
in its proof).

\begin{remark}
Note that only the presence of $q$ under the exponent in the 
definition of $\csw$ from \eqref{eq:csw} 
prevents us from the application of Theorem~\ref{th:main-spin} 
to the complete graph $K_n$.
This comes from $\D(K_n)=n-1$ and the usual normalization of 
the inverse temperature to $\beta=\frac{c}{n}$ for some $c>0$ 
(cf.~Theorem~\ref{th:SW_complete}), which 
would lead (without the $q$) to a lower bound on 
$\csw(K_n,\frac{c}{n},q)$ independent of $n$.
\end{remark}

\section[Proof of Theorem \ref*{th:main-spin}]
{Proof of Theorem \ref{th:main-spin}} \label{sec:3_proof}

The proof is based on standard techniques 
for the comparison of Markov chains (see 
Lemma~\ref{lemma:prelim_comparison}), together with an appropriate 
choice of an auxiliary Markov chain that can be compared to 
both Swendsen--Wang and heat-bath dynamics.
For the remainder of this section fix a graph $G=(V,E)$, some $\beta\ge0$ and 
$q\in\N$, and recall that we denote by $\pi$ the measure for 
the $q$-state Potts model on $G$ at inverse temperature $\beta$ 
(see \eqref{eq:Potts}).
We will analyze the auxiliary Markov chain with transition 
probability matrix 
\begin{equation}\label{eq:Q}
Q \,=\, \hb \sw \hb
\end{equation}
where $\hb$ is from \eqref{eq:P-HB} and $\sw$ from \eqref{eq:SW-P}. 
Since $\hb$ and $\sw$ are reversible with respect to $\pi$, we see that 
$Q$ is also reversible. 

The first lemma shows that a Markov chain with 
transition matrix $Q$ has a larger spectral gap than 
the heat-bath dynamics.

\begin{lemma}	\label{lemma:Q}
With the definitions above we get 
\[
\lambda(Q) 
\;\ge\; \lambda(\hb).
\]
\end{lemma}
\begin{proof}
With $S_{\pi}(\sigma,\tau)=\pi(\tau)$ for 
$\sigma,\tau\in\Op$, we have 
$Q-S_{\pi}=(\hb-S_{\pi}) \sw (\hb-S_{\pi})$, 
which is self-adjoint. Hence we can write the spectral gap 
(see \eqref{eq:gap-norm}) as
\[\begin{split}
1-\lambda(Q) \;&=\; \norm{Q-S_{\pi}}_{\pi} 
\;=\; \norm{(\hb-S_{\pi}) \sw 
			(\hb-S_{\pi})}_{\pi} \\
&\le\; \norm{\hb-S_{\pi}}_{\pi}^2 \,\norm{\sw} _\pi
\;\le\; \norm{\hb-S_{\pi}}_{\pi} \\
&=\; 1-\lambda(\hb), 
\end{split}\]
where we use submultiplicativity of the spectral norm as well as 
$\norm{\sw}_{\pi}\le1$ and $\norm{\hb-S_{\pi}}_{\pi}\le1$.
\end{proof}

To prove a lower bound on the spectral gap of $\sw$ it remains to 
prove $\lambda(\sw)\ge c\lambda(Q)$ for some $c>0$.
For this we need an estimate of the transition probabilities of the 
Swendsen--Wang dynamics on $G$ with respect to some subgraph of $G$. 
Therefore we denote the transition matrix of the Swendsen--Wang 
dynamics for the $q$-state Potts model on a graph $G$ at inverse 
temperature $\beta$ throughout this section by $P_G$, i.e. 
\begin{equation}\label{eq:SW-subgraph}
P_G \,:=\, P_{{\rm SW},\beta,q}^G
\end{equation}
(cf.~\eqref{eq:SW-P}).
We prove the following lemma.

\goodbreak

\begin{lemma} \label{lemma:G0}
Let $G=(V,E)$ be a graph and $G_0=(V,E_0)$ be a spanning 
subgraph of $G$ with $E_0\subset E$. 
Then
\[
a_1^{\abs{E\setminus E_0}}\,P_{G_0}(\sigma,\tau) 
\;\le\; P_G(\sigma,\tau) 
\;\le\; a_2^{\abs{E\setminus E_0}}\,P_{G_0}(\sigma,\tau)
\]
for all $\sigma,\tau\in\O_{\rm P}$, where
\[
a_1 \;=\; a_1(\beta) \;:=\; e^{-\beta}
\vspace{-2mm}
\]
and
\vspace{-2mm}
\[
a_2 \;=\; a_2(\beta,q) \;:=\; 1+q\, (e^{\beta}-1).
\vspace{2mm}
\]
\end{lemma}

\begin{proof}
The first inequality is already known from the proof of Lemma 3.3 in 
\cite{BCT}, but we state it here for completeness. 
Let $p=1-e^{-\beta}$ and note that $E_0(\sigma)\subset E(\sigma)$ for 
all $\sigma\in\Op$. We deduce by \eqref{eq:SW-P} that
\[\begin{split}
P_G(\sigma,\tau) \;&=\; \sum_{A\subset E} \,p^{\abs{A}}\, 
	(1-p)^{\abs{E(\sigma)}-\abs{A}}\, q^{-c(A)}\, 
	\ind\bigl(A\subset E(\sigma)\cap E(\tau)\bigr)\\
&\ge\; \sum_{A\subset E_0} \,p^{\abs{A}}\, 
	(1-p)^{\abs{E(\sigma)}-\abs{A}}\, q^{-c(A)}\, 
	\ind(A\subset E_0(\sigma)\cap E_0(\tau))\\
&=\; (1-p)^{\abs{E(\sigma)}-\abs{E_0(\sigma)}}\,P_{G_0}(\sigma,\tau)
\;\ge\; (1-p)^{\abs{E\setminus E_0}}\,P_{G_0}(\sigma,\tau).
\end{split}\]
For the second inequality suppose for now $E_0=E\setminus e$ 
with some $e\in E$ and note that $c(A\cup e)\ge c(A)-1$. We get
\[\begin{split}
P_G(\sigma,\tau) \;&=\; \sum_{A\subset E(\sigma)\cap E(\tau)} 
	\,p^{\abs{A}}\, (1-p)^{\abs{E(\sigma)}-\abs{A}}\, q^{-c(A)} \\
&=\; \sum_{\substack{A\subset E(\sigma)\cap E(\tau):\\ e\in A}} 
	\,p^{\abs{A}}\, (1-p)^{\abs{E(\sigma)}-\abs{A}}\, q^{-c(A)} \\
&\qquad\quad + \sum_{\substack{A\subset E(\sigma)\cap E(\tau):\\ 
		e\notin A}} 
	\,p^{\abs{A}}\, (1-p)^{\abs{E(\sigma)}-\abs{A}}\, q^{-c(A)} \\
&=\; \sum_{A'\subset E_0(\sigma)\cap E_0(\tau)} 
	\,p^{\abs{A'\cup e}}\, 
	(1-p)^{\abs{E(\sigma)}-\abs{A'\cup e}}\, q^{-c(A'\cup e)} \\
&\qquad\quad + \sum_{A'\subset E_0(\sigma)\cap E_0(\tau)} 
	\,p^{\abs{A'}}\, (1-p)^{\abs{E(\sigma)}-\abs{A'}}\, q^{-c(A')} \\
&\le\; \frac{q\,p}{1-p}\,	\sum_{A'\subset E_0(\sigma)\cap E_0(\tau)} 
	\,p^{\abs{A'}}\, (1-p)^{\abs{E_0(\sigma)}-\abs{A'}}\, q^{-c(A')} 
	\,+\, P_{G_0}(\sigma,\tau) \\
&=\; \left(\frac{q\,p}{1-p}\, \,+\, 1\right)\,P_{G_0}(\sigma,\tau)
\;=\; \bigl(1 \,+\, q\, (e^{\beta}-1)\bigr)\,P_{G_0}(\sigma,\tau).
\end{split}\]
For $\abs{E\setminus E_0}>1$ one can iterate this procedure 
$\abs{E\setminus E_0}$ times. 
\end{proof}

We use this lemma to prove that the transition probability from 
$\sigma$ to $\tau$ is similar to the probability of going from 
a neighbor of $\sigma$ to a neighbor of $\tau$. 
Recall that $\sigma^{v,k}$ is defined by $\sigma^{v,k}(v)=k\in[q]$ and 
$\sigma^{v,k}(u)=\sigma(u)$, $u\neq v$.

\begin{lemma} \label{lemma:P}
Let $\sigma,\tau\in\O_{\rm P}$, $v\in V$ and $k,l\in[q]$.
Then
\[
P_G(\sigma^{v,k},\tau^{v,l}) \;\le\; 
a_3^{\deg_G(v)}\;
P_G(\sigma,\tau)
\]
with
\[
a_3 \;=\; a_3(\beta,q) \;:=\; q\,e^{2\beta} - (q-1)\,e^\beta,
\]
where $\deg_G(v)$ denotes the degree of the vertex $v$ in $G$.
\end{lemma}

\begin{proof}
Define $E_v:=\{e\in E: v\in e\}$ and $G_v:=(V,E\setminus E_v)$. 
Then $v\in V$ is an isolated vertex in $G_v$. 
By the definition of the Swendsen--Wang dynamics we get  
\[
P_{G_v}(\sigma^{v,k},\tau^{v,l}) 
\;=\; P_{G_v}(\sigma,\tau).
\]
If we set $E_0=E\setminus E_v$ we get 
$\abs{E\setminus E_0}=\deg_G(v)$. 
We deduce by Lemma \ref{lemma:G0} that
\[\begin{split}
P_{G}(\sigma^{v,k},\tau^{v,l}) 
\;&\le\; a_2^{\deg_G(v)}\,P_{G_v}(\sigma^{v,k},\tau^{v,l})
\;=\; a_2^{\deg_G(v)}\,P_{G_v}(\sigma,\tau) \\
&\le\;\left(\frac{a_2}{a_1}\right)^{\deg_G(v)}\,P_{G}(\sigma,\tau)\\
\end{split}\]
with $a_1$ and $a_2$ from Lemma \ref{lemma:G0}.
\end{proof}

Now we are able to prove the main result of this chapter.

\begin{proof}[Proof of Theorem \ref{th:main-spin}]
Because of Lemma \ref{lemma:Q} we only have to prove 
$\lambda(\sw)\ge \csw\,\lambda(Q)$.
Let
\begin{equation}\label{eq:proof_main-spin}
c \;:=\; \max_{\substack{\sigma_1,\sigma_2,\tau_1,\tau_2\in\O_{\rm P}\\ 
\sigma_1\sim \sigma_2,\,\tau_1\sim \tau_2}}\;
\frac{\sw(\sigma_1,\tau_1)}{\sw(\sigma_2,\tau_2)},
\end{equation}
where $\sigma\sim\tau:\Leftrightarrow 
\sum_{v\in V}\ind(\sigma(v)\neq\tau(v))\le1$, i.e.~$\sigma$ and $\tau$ 
differ in at most one vertex. 
Note that $\hb(\sigma,\tau)\neq0$ if and only if 
$\sigma\sim\tau$.
We find for $\sigma_1,\tau_1\in\Op$ that
\[\begin{split}
Q(\sigma_1,\tau_1) 
\;&=\; \sum_{\sigma_2,\tau_2\in\O_{\rm P}}\,
\hb(\sigma_1,\sigma_2)\,\sw(\sigma_2,\tau_2)\,
	\hb(\tau_2,\tau_1) \\
&\le\; c\,\sw(\sigma_1,\tau_1) \;\sum_{\sigma_2\sim\sigma_1} 
			\hb(\sigma_1,\sigma_2)\,
	\sum_{\tau_2\sim\tau_1} 
			\hb(\tau_2,\tau_1) \\
&\le\; q\,c\,\sw(\sigma_1,\tau_1).
\end{split}\]
Using Lemma~\ref{lemma:prelim_comparison} we obtain
\[
\lambda(Q) \;\le\; q\,c\,\lambda(\sw).
\]
It remains to bound $c$. 
Recall that $\deg_G(v)\le\D$ for all $v\in V(G)$. 
With $a_3$ from Lemma~\ref{lemma:P} we get, for 
$\sigma_1,\sigma_2,\tau_1, \tau_2\in\Op$ with 
$\sigma_1\sim \sigma_2$ and $\tau_1\sim \tau_2$,
\[
\frac{\sw(\sigma_1,\tau_1)}{\sw(\sigma_2,\tau_2)}
\;\le\; a_3^{\Delta}\,\frac{\sw(\sigma_2,\tau_1)}{\sw(\sigma_2,\tau_2)}
\;\le\; a_3^{2 \Delta}\,\frac{\sw(\sigma_2,\tau_2)}{\sw(\sigma_2,\tau_2)} 
\;=\; a_3^{2 \Delta}.
\]
Finally,
\[
\lambda(Q) \;\le\; q\,c\,\lambda(\sw)
\;\le\; q\,a_3^{2 \Delta}\,\lambda(\sw)
\;\le\; q\,(q\,e^{2\beta})^{2 \Delta}\,\lambda(\sw).
\]
This completes the proof.
\end{proof}


\section{A slight generalization} \label{sec:3_general}

In this section we present a generalization of 
Theorem~\ref{th:main-spin} 
that is necessary to handle also graphs with a single 
vertex of maximal degree.
The idea behind this modification is that the transition 
probabilities of the Swendsen--Wang dynamics are invariant 
under global flips of the color of all vertices, i.e. 
\begin{equation}\label{eq:SW-flip}
\sw(\sigma,\tau) \,=\, \sw(\sigma+k,\tau+l),
\qquad\sigma,\tau\in\Op,\; k,l\in[q], 
\end{equation}
where 
$(\sigma+k)(v):=(\sigma(v)+k-1\mod q)+1$ for all $v\in V$.
To see this, note that $\sw(\sigma,\tau)$ (cf.~\eqref{eq:SW-P}) 
does not depend on the precise colors in $\sigma$ and $\tau$, 
but on the ``edges of agreement'' $E(\sigma)$ and $E(\tau)$ 
(see \eqref{eq:Esigma}), 
which are invariant under global flips.
If we now consider the heat-bath dynamics with additional 
global flips, i.e.~in each step make one step using $\hb$ and 
then change the color of all vertices at once by a random 
increment, it is reasonable to conjecture that SW has also 
an (almost) larger spectral gap than this Markov chain.
In the following theorem we prove this statement, but 
in a different form, namely, we consider the original 
Markov chains under the condition that the color of the 
configurations is fixed at a single vertex. 
Because of symmetry we can assume that the fixed color equals 1.
This can be interpreted as ``boundary conditions'' for the 
neighbors of this single vertex.

For some fixed vertex $w\in V$ in the graph $G=(V,E)$ 
we denote by 
\[
\Lambda_w \,:=\, \bigl\{1,\dots,q\bigr\}^{V\setminus w}
\]
the set of all colorings of the vertices $V\setminus w$ 
and define a probability measure on $\Lambda_w$ by 
\begin{equation} \label{eq:Potts_w}
\pi^w(\bar\sigma) \,:=\, 
q\,\pi(\bar\sigma^1), 
\qquad \bar\sigma\in\Lambda_w, 
\end{equation}
with $\pi$ from \eqref{eq:Potts}, 
where $\bar\sigma^1\in\Op$ is a coloring of $V$ such that 
$\bar\sigma^1(u)=\bar\sigma(u)$, $u\neq w$, and 
$\bar\sigma^1(w)=1$. 
One may think of $\pi^w$ as the conditional probability measure 
on the colorings of the vertices $V\setminus w$ with respect to 
$\pi$ given that $w$ is colored 1.

It is not difficult to show (using ideas similar to those of the 
proof of Theorem~\ref{th:main-spin2} below) that the 
Markov chain on $\Lambda_w$ with transition matrix 
$\bar{P}(\bar\sigma,\bar\tau):= q \sw(\sigma,\tau)$, where 
$\sigma,\tau\in\Op$ and $\bar\sigma,\bar\tau\in\Lambda_w$ such that 
\[
\bar\sigma(v) \,:=\, (\sigma(v)-\sigma(w)\mod q)+1, \qquad v\in V\setminus w,
\] 
is reversible with respect to $\pi^w$ 
and has the same spectral gap as $\sw$.
Additionally, we define the heat-bath dynamics that is reversible 
with respect to $\pi^w$ similarly 
to \eqref{eq:P-HB} by 
\begin{equation}\label{eq:P-HB2}
\hb^w(\bar\sigma,\bar\tau) 
\;:=\; P_{{\rm HB},\beta,q}^{G,w,1}(\bar\sigma,\bar\tau) \;=\;
\frac{1}{\abs{V}-1}\sum_{v\in V\setminus w}\,
	\frac{\pi^w(\bar\tau)}
	{\sum_{l=1}^q\pi^w(\bar\sigma^{v,l})}\,
	\ind(\bar\tau\in\O^w_{\sigma,v}), 
\end{equation}
where $\O^w_{\bar\sigma,v}:=\{\bar\tau\in\Lambda_w:\, 
\bar\tau(u)=\bar\sigma(u), \forall u\neq v\}$.
The following theorem gives a comparison inequality between 
$\sw$ and $\hb^w$ that will be essential in utilizing results 
on the heat-bath dynamics for the Potts model with boundary 
conditions (cf.~\eqref{eq:Potts_boundary}) for the analysis 
of the SW dynamics.

\vspace{2mm}
\begin{theorem}\label{th:main-spin2}
Suppose that $\sw$ is the transition matrix of the 
Swendsen--Wang dynamics for the $q$-state Potts 
model at inverse temperature $\beta$ on a graph $G=(V,E)$. 
Furthermore, let $w\in V$ be any vertex and $\hb^w$ be as in 
\eqref{eq:P-HB2}. 
Then
\[
\lambda(\sw) \;\ge\; \tcsw\,\lambda(\hb^w) 
\]
with
\[
\tcsw \;=\; \tcsw(G,\beta,q) 
\;:=\; q^{-1}\left(q\,e^{2\beta}\right)^{-2\widetilde{\D}}, 
\]
where
\[
\widetilde{\D} \;:=\; \max_{u\in V\setminus w}\,\deg_G(u)
\] 
is the maximum degree over $V\setminus w$.
\end{theorem}

Before we prove Theorem~\ref{th:main-spin2}, we present an 
application of it.
For this recall the definition of the two-dimensional 
square lattice $\Z_L^2=(V_{L,2},E_{L,2})$ of side length $L$ 
from Section~\ref{sec:2_known-results}, 
and that the boundary of $V_{L,2}$ is defined by 
$\partial V_{L,2} := \bigl\{v=(v_1,v_2)\in V_{L,2}:\, 
v_i\in\{1,L\} \text{ for some } i\in[2]\bigr\}$. 
Now we introduce a new auxiliary vertex $v^*$ and 
let $\Z_L^{2\dag}$ be the graph $\Z_{L-1}^2$ with additional 
vertex $v^*$ and edges between $v^*$ and all boundary vertices 
$u\in\partial V_{L-1,2}$.
That is $\Z_L^{2\dag}:=(V_{L,2}^\dag,E_{L,2}^\dag,\varphi)$ 
is the (multi-)graph (cf. Section~\ref{sec:2_models}) with 
vertex set $V_{L,2}^\dag=V_{L-1,2}\cup v^*$ and edge set 
$E_{L,2}^\dag$ such that
the set of endpoints of the edges in $E_{L,2}^\dag$ satisfies 
$\varphi(E_{L,2}^\dag)=
E_{L-1,2}\cup\{\{v^*,u\}:\,u\in\partial V_{L-1,2}\}$ 
and, 
for each of the vertices 
$(1,1), (1,L-1), (L-1,1), (L-1,L-1)\in V_{L,2}^\dag$, 
there are two parallel edges to $v^*$ \setcounter{figure}1
(see Figure~\ref{fig:square_only-dual}).
%
%
\begin{figure}[ht]
\begin{center}
\hspace*{5mm}
\scalebox{1}{
	\psset{xunit=1.0cm,yunit=1.0cm,algebraic=true,dotstyle=o,dotsize=3pt 0,linewidth=0.8pt,arrowsize=3pt 2,arrowinset=0.25}
	\begin{pspicture*}(-0.45,0.12)(5.95,4.36)
	\psline(1.5,1.5)(1.5,2.5)
	\psline(1.5,2.5)(2.5,2.5)
	\psline(2.5,2.5)(2.5,1.5)
	\psline(2.5,1.5)(1.5,1.5)
	\parametricplot{0.1852913062710569}{2.989280090214289}{1*1.34*cos(t)+0*1.34*sin(t)+2.83|0*1.34*cos(t)+1*1.34*sin(t)+2.3}
	\parametricplot{1.1350083349985625}{2.0596018113318753}{1*1.84*cos(t)+0*1.84*sin(t)+3.37|0*1.84*cos(t)+1*1.84*sin(t)+0.87}
	\parametricplot{4.6717617082515}{5.883637174045906}{1*1.71*cos(t)+0*1.71*sin(t)+2.57|0*1.71*cos(t)+1*1.71*sin(t)+3.21}
	\parametricplot{1.5028552028831879}{2.000187227182034}
		{1*0.86*cos(t)+0*0.86*sin(t)+1.44|-0.075*0.86*cos(t)+1*0.86*sin(t)+0.64}
	\parametricplot{2.099045279311494}{4.0571577644321395}
		{1*0.54*cos(t)+0*0.54*sin(t)+1.37|0*0.54*cos(t)+1*0.54*sin(t)+0.99}
	\parametricplot{-2.1993684311435224}{0.19067988888988038}{1*1.98*cos(t)+0*1.98*sin(t)+2.2|0*1.98*cos(t)+1*1.98*sin(t)+2.17}
	\parametricplot{2.8453007051607155}{5.540388001593068}{1*0.52*cos(t)+0*0.52*sin(t)+3|0*0.52*cos(t)+1*0.52*sin(t)+1.35}
	\parametricplot{5.428195851138397}{6.21778798060977}{1*2.24*cos(t)+0*2.24*sin(t)+1.91|0*2.24*cos(t)+1*2.24*sin(t)+2.69}
	\parametricplot{1.651943081227404}{3.288962910742752}{1*0.59*cos(t)+0*0.59*sin(t)+3.08|0*0.59*cos(t)+1*0.59*sin(t)+2.58}
	\parametricplot{0.523108668073518}{1.593136560077821}{1*1.25*cos(t)+0*1.25*sin(t)+3.06|0*1.25*cos(t)+1*1.25*sin(t)+1.92}
	\parametricplot{2.956350314060717}{4.7749714687606275}{1*0.85*cos(t)+0*0.85*sin(t)+2.34|0*0.85*cos(t)+1*0.85*sin(t)+1.34}
	\parametricplot{-1.5125031955649737}{0.0989999773941644}{1*1.87*cos(t)+0*1.87*sin(t)+2.28|0*1.87*cos(t)+1*1.87*sin(t)+2.36}
	\parametricplot{2.677755480453707}{4.9277297505381314}{1*0.54*cos(t)+0*0.54*sin(t)+1.39|0*0.54*cos(t)+1*0.54*sin(t)+3.02}
	\parametricplot{1.8843071408090328}{2.583159724740975}{1*1.8*cos(t)+0*1.8*sin(t)+2.44|0*1.8*cos(t)+1*1.8*sin(t)+2.31}
	\parametricplot{0.17396397227159288}{1.8107485743137344}{1*1.85*cos(t)+0*1.85*sin(t)+2.32|0*1.85*cos(t)+1*1.85*sin(t)+2.22}
	\begin{scriptsize}
	\psdots[dotsize=6pt 0,dotstyle=*](1.5,1.5)
	\psdots[dotsize=6pt 0,dotstyle=*](2.5,1.5)
	\psdots[dotsize=6pt 0,dotstyle=*](1.5,2.5)
	\psdots[dotsize=6pt 0,dotstyle=*](2.5,2.5)
	\psdots[dotsize=6pt 0,dotstyle=*](4.15,2.54)
	\end{scriptsize}
	\end{pspicture*}
}
\end{center}\vspace*{-2mm}
\caption{The graph $\Z_3^{2\dag}$}
\label{fig:square_only-dual}
\end{figure}
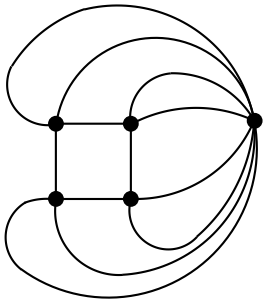
%
%
Now we can deduce the following directly from Theorem~\ref{th:P-HB_2d}.
\begin{corollary}\label{coro:P-SW_2d_boundary}
Let $\sw$ be the transition matrix of the Swendsen--Wang dynamics for 
the $q$-state Potts model on $\Z_L^{2\dag}$ at inverse temperature 
$\beta$. 
Let $n=|V^\dag_{L,2}|$. Then there exists a constant 
$c_\beta=c_\beta(q)>0$ 
such that
\[
\lambda(\sw) \;\ge\; \frac{c_\beta}{n} \qquad\;\;\, 
	\text{ for } \beta < \beta_c(q) 
\]
with $\beta_c(q)\,=\,\log(1+\sqrt{q})$. 
\end{corollary}

\begin{proof}
Let the auxiliary vertex $v^*$ be the vertex with fixed color 
from Theorem~\ref{th:main-spin2}. Then we see that 
$P_{{\rm HB}}^{v^*}$ from \eqref{eq:P-HB2} equals 
$P_{{\rm HB},1}$ from \eqref{eq:P-HB_boundary}.
Thus, the result follows from Theorem~\ref{th:P-HB_2d}.
\end{proof}

We finish this chapter with the proof of Theorem~\ref{th:main-spin2}.

\begin{proof}[Proof of Theorem~\ref{th:main-spin2}]
The proof is very similar to the proof of 
Theorem~\ref{th:main-spin}.
First, recall that for $\sigma\in\Op$, 
$\bar\sigma(v)=(\sigma(v)-\sigma(w)\mod q)+1$, $v\in V\setminus w$, 
with $\bar\sigma\in\Lambda_w$. 
Additionally, for $\bar\sigma\in\Lambda_w$, we denote by 
$\bar\sigma^1\in\Op$ the configuration with 
$\bar\sigma^1(v)=\bar\sigma(v)$, $v\neq w$, and $\bar\sigma^1(w)=1$.

We define, for $\sigma\in\O_{\rm P}$ and 
$\bar\tau\in\Lambda_w$, the ``flip'' 
transition matrices 
\[
F_1(\sigma,\bar\tau) \;:=\; \ind(\bar\tau\,=\,\bar\sigma)
\]
and
\[
F_2(\bar\tau,\sigma) \;:=\; 
\frac{1}{q}\,\sum_{l=0}^{q-1}\,\ind(\sigma=\bar\tau^1+l).
\]
(If we consider $F_1$ as an operator mapping from $L_2(\pi^w)$ 
to $L_2(\pi)$, then $F_2=F_1^*$.)\\
With $\pi$ from \eqref{eq:Potts} and $\pi^w$ from \eqref{eq:Potts_w}, 
it is easy to check that $\pi F_1=\pi^w$ and 
$\pi^w F_2=\pi$.
Following the same ideas as in Section~\ref{sec:3_proof}
with
\[
Q \;=\; F_1\,\hb^w\,F_2\,\sw\,
	F_1\,\hb^w\,F_2,
\]
which is reversible with respect to $\pi$, we get 
\[
\norm{ Q-S_\pi}_\pi \,\le\, 
\norm{ F_1\,\hb^w\,F_2-S_{\pi}}_{\pi}
\,=\, \norm{ F_1\,(\hb^w-S_{\pi^w})\,F_2}_{\pi}
\;\le\; \norm{\hb^{w}-S_{\pi^w}}_{\pi^w}.
\]
This proves $\lambda(Q)\ge\lambda(\hb^w)$. 
It remains to show that $\lambda(\sw)\ge\tcsw\,\lambda(Q)$. 
By the construction of the Swendsen--Wang dynamics we know that 
\[
\sw(\sigma,\tau) \;=\; \sw(\bar\sigma^1,\bar\tau^1)\qquad 
	\forall\sigma,\tau\in\O_{\rm P}.
\]
Hence we deduce with 
\[
\widetilde{c} \;:=\; 
\max_{\substack{\sigma_1,\sigma_2,\tau_1,\tau_2\in\Op\\ 
\bar\sigma_1^1\sim \bar\sigma_2^1,\,\bar\tau_1^1\sim \bar\tau_2^1}}\;
\frac{\sw(\sigma_1,\tau_1)}{\sw(\sigma_2,\tau_2)}
\]
(cf.~\eqref{eq:proof_main-spin}) 
that, for all $\sigma,\tau\in\Op$, 
\[\begin{split}
Q(\sigma,\tau) 
\;&=\; \sum_{\substack{\bar\sigma_1,\bar\tau_1\in\Lambda_w,\\
								\sigma_2,\tau_2\in\Op}}
	\hb^w(\bar\sigma,\bar\sigma_1)\,F_2(\bar\sigma_1,\sigma_2)\,
	\sw(\sigma_2,\tau_2)\,\hb^w(\bar\tau_2,\bar\tau_1)\,
	F_2(\bar\tau_1,\tau) \\
\;&=\; \sum_{\substack{\bar\sigma_1\in\Lambda_w,\\
								\tau_2\in\Op}}
	\hb^w(\bar\sigma,\bar\sigma_1)\,
	\sw(\bar\sigma_1^1,\bar\tau_2^1)\,\hb^w(\bar\tau_2,\bar\tau)\,
	F_2(\bar\tau,\tau) \\
&\le\; \widetilde{c}\,\sw(\bar\sigma^1,\bar\tau^1) \;
	\sum_{\bar\sigma_1\sim\bar\sigma} \hb^w(\bar\sigma,\bar\sigma_1)\;\cdot
	\frac{1}{q} \,
	\sum_{\tau_2:\bar\tau_2\sim\bar\tau} \hb^w(\bar\tau_2,\bar\tau) \\
&=\; \widetilde{c}\,\sw(\sigma,\tau) \;
	\sum_{\substack{\bar\tau_3\in\Lambda_w:\\
								\bar\tau_3\sim\bar\tau}} \hb^w(\bar\tau_3,\bar\tau)\\
\;&\le\; q\,\widetilde{c}\,\sw(\sigma,\tau).
\end{split}\]
With $a_3$ from Lemma~\ref{lemma:P} we see for 
$\sigma_1,\sigma_2,\tau_1,\tau_2\in\Op$ with 
$\bar\sigma_1^1\sim \bar\sigma_2^1$ and $\bar\tau_1^1\sim \bar\tau_2^1$ that
\[
\frac{\sw(\sigma_1,\tau_1)}{\sw(\sigma_2,\tau_2)}
\;\le\; a_3^{\widetilde{\Delta}}\,
	\frac{\sw(\sigma_2,\tau_1)}{\sw(\sigma_2,\tau_2)}
\;\le\; a_3^{2 \widetilde{\Delta}}\,
	\frac{\sw(\sigma_2,\tau_2)}{\sw(\sigma_2,\tau_2)} 
\;=\; a_3^{2 \widetilde{\D}}, 
\]
because $\deg_G(v)\le\widetilde\D$ for all $v\in V\setminus w$.
By the same ideas as in the proof of Theorem~\ref{th:main-spin} 
(and Lemma~\ref{lemma:prelim_comparison}) we conclude that 
\[
\lambda(Q) \;\le\; q\,\widetilde{c}\,\lambda(\sw)
\;\le\; q\,a_3^{2 \widetilde{\D}}\,\lambda(\sw)
\;\le\; q\,(q\,e^{2\beta})^{2 \widetilde{\D}}\,\lambda(\sw).
\]
This completes the proof.
\end{proof}


\chapter{Comparison with single-bond dynamics}\label{chap:bond}

The goal of this chapter is to present a comparison of 
Swendsen--Wang and single-bond dynamics for the random-cluster model. 
In fact, we prove that the spectral gaps of both Markov chains are 
equal up to a small polynomial in the number of edges of the 
underlying graph $G=(V,E)$, i.e.~up to $c \abs{E} \log\abs{E}$ for 
some $c<\infty$. 
In particular, this proves that rapid mixing of both Markov chains 
is equivalent.

In this chapter, we first give a representation 
of both Markov chains (or their transition matrices) on the joint 
(\jointname) model (see \eqref{eq:joint}) using the same 
``building blocks''. 
Then we provide some technical lemmas that yield estimates on 
the norm of products of operators between (not necessarily equal) 
Hilbert spaces.

Using these ingredients we obtain several new rapid mixing results 
for the single-bond dynamics, but we can also use 
Theorem~\ref{th:SB_linear-width} to get a result on Swendsen--Wang 
dynamics on graphs of bounded \emph{linear width}.
Additionally we adopt the result of Theorem~\ref{th:P-SW-HB_torus} 
to prove matching upper and lower bounds for the spectral gap 
of the single-bond dynamics on the $d$-dimensional torus 
(see \eqref{eq:def-torus}) at the critical temperature that 
show slow mixing if the number of colors is large enough.


\section{Common representation} \label{sec:4_repr}

We want to represent the Swendsen--Wang and the single-bond 
dynamics on the joint model, which consists of the 
product state space $\Oj:=\Op\times\Orc$ 
and the \jointname\ measure. 
This was done first in \cite{U3} and we follow the steps from 
that article. 
First recall that, if $G=(V,E)$, $p=1-e^{-\beta}$ and $q\in\N$ are 
fixed, we denote by $\mu$ (resp.~$\nu$) the random-cluster 
(resp.~\jointname) measure on $G$ with parameters $p$ and $q$ 
(see \eqref{eq:RC} and \eqref{eq:joint}) 
and that $L_2(\mu)$ (resp.~$L_2(\nu)$) is the Hilbert space that 
consists of $\R^{\Orc}$ (resp.~$\R^{\Oj}$) with the inner product 
$\l\cdot,\cdot\r_\mu$ (resp.~$\l\cdot,\cdot\r_\nu$).

We introduce the stochastic matrix $M$ that defines the mapping 
(by matrix multiplication) from the RC to the joint model by
\begin{equation} \label{eq:M}
M\bigl(B,(\sigma,A)\bigr) \;:=\; q^{-c(B)}\;\ind\bigl(A=B\bigr)\;
	\ind\bigl(B\subset E(\sigma)\bigr)
\end{equation}
for $B\in\Orc$ and $(\sigma,A)\in\Oj$.
Obviously, we find by definition of $\mu$ and $\nu$ that 
$M$ satisfies
$\mu M (\sigma,A)=\sum_{B\in\Orc}\mu(B)\, M(B,(\sigma,A))
	=\nu(\sigma,A)$.
(See the discussion after \eqref{eq:joint-cond} 
for a possible interpretation of that equality.)
As in \eqref{eq:map}, $M$ defines an operator that 
maps from $L_2(\nu)$ to $L_2(\mu)$ by 
\begin{equation} \label{eq:map2}
Mf(B) \;:=\; \sum_{(\sigma,A)\in\Oj}\,M\bigl(B,(\sigma,A)\bigr)\,
	f(\sigma,A), \qquad f\in L_2(\nu).
\end{equation}
Its adjoint operator $M^*$, i.e.~the operator that satisfies
\[
\l M^*g, f\r_\nu \;=\; \l g,Mf\r_\mu \qquad 
	\text{ for all }\; f\in L_2(\nu), g\in L_2(\mu), 
\]
can be given by the (stochastic) matrix
\[
M^*\bigl((\sigma,A),B\bigr) \;=\; \ind\bigl(A=B\bigr).
\]
One simple and useful property of the matrices $M$ and $M^*$ is 
that their product satisfies 
$M M^*(A,B)=\ind(A=B)$ for all $A,B\in\Orc$, or 
equivalently, the corresponding operators satisfy 
$M M^* = I$ with $If:=f$ for all $f\in L_2(\mu)$.

\begin{remark}
The matrices above will be used with the following 
interpretation in mind.
We want to construct a Markov chain for the RC model that 
is reversible with respect to~$\mu$. 
This could be done, starting in a RC configuration, 
by making a ``step'' with $M$ first, which leads to 
a configuration in the joint model 
(with an unchanged RC coordinate). 
Then make updates of the RC coordinate in the joint 
configuration that are reversible with respect to $\nu$.
Finally, forget about the coloring to obtain a new 
RC configuration.
At first sight this is no advantage, but we will see that 
it is possible in some cases to find a representation 
of Markov chains for the RC model by rather simple update 
rules on the joint model, while the original updates 
are difficult to handle.
\end{remark}
\vspace{2mm}

The following matrices provide the updates of the 
``RC coordinate'' in the joint model. 
For $(\sigma,A),(\tau,B)\in\O_{\rm J}$ and $e\in E$ with 
$e=\{e^{(1)},e^{(2)}\}$ let
\begin{equation} \label{eq:T-e}
T_e\bigl((\sigma,A),(\tau,B)\bigr) \;:=\; \ind\bigl(\sigma=\tau\bigr)\;
	\begin{cases}
	p, & B=A\cup e  \,\text{ and }\;  \sigma(e^{(1)})=\sigma(e^{(2)})\\
	1-p, & B=A\setminus e  \;\text{ and }\;  \sigma(e^{(1)})=\sigma(e^{(2)})\\
	1, & B=A\setminus e  \;\text{ and }\;  \sigma(e^{(1)})\ne\sigma(e^{(2)}).
	\end{cases}
\end{equation}

The lemma below
shows some interesting properties of the matrices from \eqref{eq:M} and 
\eqref{eq:T-e}. For example $\{T_e\}_{e\in E}$ is a family of commuting 
projections in $L_2(\nu)$. This will be important in the proof of 
the main result.

\begin{lemma} \label{lemma:repr-prop}
Let $M$, $M^*$ and $T_e$ be the matrices above.
Then 
\begin{enumerate}
	\renewcommand{\labelenumi}{\rm(\roman{enumi})}
	\item $M^*M$ and $T_e$ are self-adjoint in $L_2(\nu)$.
\vspace{1mm}
	\item $T_e T_e = T_e$ and $T_e T_{e'} = T_{e'} T_e$ 
				for all $e,e'\in E$.
\vspace{1mm}
	\item $\norm{T_{e}}_{\nu}=1$ and $\norm{M^*M}_{\nu}=1$.
\end{enumerate}
\end{lemma}

\begin{proof}
The first part of $(i)$ is obvious since 
$(M^*M)^*=M^* (M^*)^*=M^*M$ (see e.g.~\cite[Thm.~3.9-4]{Krey}). 
The second part can be checked easily using \eqref{eq:adjoint}. 
Part $(ii)$ comes from the fact that the transition 
probabilities depend only on the ``coordinate'' that will not 
be changed. 
For $(iii)$ note that $\norm{M^*M}_{\nu}=\norm{M M^*}_{\mu}$ 
(see \cite[Thm.~3.9-4]{Krey}), and $M M^*=I$ with $If:=f$, 
$f\in L_2(\mu)$. Hence, $\norm{M^*M}_{\nu}=\norm{I}_{\mu}=1$.
It remains to prove $\norm{T_{e}}_{\nu}=1$. 
By $T_e T_e = T_e$ and the self-adjointness of $T_e$ 
we obtain 
$\norm{T_{e}}_{\nu} \stackrel{(ii)}{=} \norm{T_{e}^2}_{\nu} 
\stackrel{(i)}{=} \norm{T_{e}}_{\nu}^2$. 
This implies $\norm{T_{e}}_{\nu}\in\{0,1\}$, but since 
$T_e g = g$, if $g$ is constant, $\norm{T_{e}}_{\nu}$ 
cannot be zero.
\end{proof}

We finish this section with a lemma that demonstrates the 
relation of $M$ and $T_e$ to the 
Swendsen--Wang and single-bond dynamics.
Recall that $\tsw$ and $\sb$ are their transition matrices 
(see \eqref{eq:SW-RC} and \eqref{eq:SB}).

\begin{lemma} \label{lemma:repr}
Let $M$, $M^*$ and $T_e$ be the matrices above.
Then
\begin{enumerate}
	\renewcommand{\labelenumi}{\rm(\roman{enumi})}
	\item $\tsw \,=\, M \left(\prod\limits_{e\in E} T_e\right) M^*$.
\vspace{1mm}
	\item $\sb 
			\,=\, \frac1{\abs{E}}\sum\limits_{e\in E}\,M\,T_e\,M^*
			\,=\, M\left(\frac1{\abs{E}}\sum\limits_{e\in E}\,T_e\right)\,M^*$.
\end{enumerate}
\end{lemma}
\vspace{2mm}

From Lemma \ref{lemma:repr-prop}$(ii)$ we know that the order of multiplication 
in $(i)$ is not important.

\begin{proof}
For $(i)$ note that
\[
\biggl(\prod\limits_{e\in E} T_e\biggr)\bigl((\sigma,A),(\tau,B)\bigr) 
	\;=\; \ind\bigl(\sigma=\tau\bigr)\,\ind\bigl(B\subset E(\sigma)\bigr)\,
	p^{\abs{B}}(1-p)^{\abs{E(\sigma)}-\abs{B}}.
\]
Hence,
\[\begin{split}
M \left(\prod\limits_{e\in E} T_e\right) M^*\bigl(A,B\bigr) 
\;&=\; \sum_{\sigma\in\Op}\, M\bigl(A,(\sigma,A)\bigr)\,
	\biggl(\prod\limits_{e\in E} T_e\biggr)\bigl((\sigma,A),(\sigma,B)\bigr)\\
&=\; \sum_{\sigma\in\Op} q^{-c(A)}\,
		\underbrace{\ind\bigl(A\subset E(\sigma)\bigr)\,
				\ind\bigl(B\subset E(\sigma)\bigr)}_{
				=\mathds{1}(A\cup B\subset E(\sigma))}\,
		p^{\abs{B}}(1-p)^{\abs{E(\sigma)}-\abs{B}}\\
&=\; P_{\rm SW}(A,B).
\end{split}\]
For part~$(ii)$ we define 
$\ind_e(\sigma):=\ind\bigl(\sigma(e^{(1)})=\sigma(e^{(2)})\bigr)$ and 
$\ind_e(A):=\ind\bigl(e^{(1)}\stackrel{A}{\leftrightarrow}e^{(2)}\bigr)$ 
for $\sigma\in\O_{\rm P}$, $A\in\O_{\rm RC}$ and $e\in E$ with 
$e=\{e^{(1)},e^{(2)}\}$.
Now write 
\[
T_e\bigl((\sigma,A),(\sigma,B)\bigr) \;=\;
\ind\bigl(B=A\setminus e\bigr) + p\,\ind_e(\sigma)
\Bigl[\ind\bigl(B=A\cup e\bigr)
-\ind\bigl(B=A\setminus e\bigr)\Bigr]
\]
and note that $\abs{\{\sigma\in\Op:\, A\subset E(\sigma)\}}=q^{c(A)}$ 
(because the colorings have to be constant on each of the $c(A)$ 
components) and
\[
q^{-c(A)}\sum_{\sigma: A\subset E(\sigma)}\ind_e(\sigma) 
\;=\; \frac1q \,+\, \ind_e(A)\left(1-\frac1q\right).
\]
Hence,
\[\begin{split}
M\,T_e\,M^*\bigl(A,B\bigr) 
\;&=\; \sum_{\sigma} q^{-c(A)}\,\ind\bigl(A\subset E(\sigma)\bigr)\,
				T_e\bigl((\sigma,A),(\sigma,B)\bigr)\\
&=\;\ind\bigl(B=A\setminus e\bigr) \\
& \qquad\quad +\, p\,
		\Bigl[\ind\bigl(B=A\cup e\bigr)
		-\ind\bigl(B=A\setminus e\bigr)\Bigr]
		\cdot q^{-c(A)}\sum_{\sigma: A\subset E(\sigma)}\ind_e(\sigma)\\
&=\; \begin{cases}
	p, & B=A\cup e  \,\text{ and }\; e^{(1)}\stackrel{A}{\leftrightarrow} e^{(2)}\\
	1-p, & B=A\setminus e \,\text{ and }\; e^{(1)}\stackrel{A}{\leftrightarrow} e^{(2)}\\
	\frac{p}{q}, 
		& B=A\cup e  \,\text{ and }\; e^{(1)}\stackrel{A}{\nleftrightarrow} e^{(2)}\\
	1-\frac{p}{q}, 
		& B=A\setminus e  \,\text{ and }\; e^{(1)}\stackrel{A}{\nleftrightarrow} e^{(2)}.
	\end{cases}\\
\end{split}\]
Summing over $E$ and dividing by $\abs{E}$ shows equality to 
\eqref{eq:SB}.
\end{proof}

\begin{remark}\label{remark:SB-SW_pos-def}
We know from Lemma~\ref{lemma:repr-prop} that $T_e$, $e\in E$, 
is self-adjoint and idempotent. 
Hence, $T_e$ is positive, i.e.~$\l T_ef,f\r_\nu\ge0$ 
for all $f\in L_2(\nu)$ (see~\cite[Thms.~\mbox{9.5-1} \& \mbox{9.5-2}]{Krey}). 
In fact, all eigenvalues of $T_e$ are either 0 and 1.
Additionally, $\prod_{e\in E} T_e$ is positive, 
because it is the product of commuting and 
self-adjoint linear operators (see \cite[Thm.~9.3-1]{Krey}).
Consequently, 
\[
\left\l \tsw g,g\right\r_\mu 
\,=\, \left\l M \left(\prod\limits_{e\in E} T_e\right) M^* g, 
				g\right\r_\mu 
\,=\, \left\l \left(\prod\limits_{e\in E} T_e\right) M^* g, 
				M^*g\right\r_\nu 
\,\ge\, 0
\]
and
\[
\left\l \sb g,g\right\r_\mu 
\,=\, \frac1{\abs{E}}\sum\limits_{e\in E}\,
				\left\l M T_e\,M^*g,g\right\r_\mu 
\,=\, \frac1{\abs{E}}\sum\limits_{e\in E}\,
				\left\l T_e\,M^*g, M^*g\right\r_\nu \,\ge\, 0
\]
for all $g\in L_2(\mu)$. 
This shows that the transition matrices of 
SW and SB 
dynamics are positive, which implies that they have 
only non-negative eigenvalues (see \cite[Obs.~7.1.4]{HJ-matrix}).
\end{remark}

\vspace{2mm}


\section{Technical lemmas} \label{sec:4_technical}

In this section we provide some technical lemmas that will be 
necessary for the analysis. We state them in a general form, 
because we guess that they could be useful also in other settings.
First let us introduce the notation. 
Consider two Hilbert spaces $H_1$ and $H_2$ with the corresponding 
inner products $\l\cdot,\cdot\r_{H_1}$ and $\l\cdot,\cdot\r_{H_2}$. 
The norms in ${H_1}$ and ${H_2}$ are defined as usual as the square 
root of the inner product of a function with itself. 
Throughout this section we consider two bounded, 
linear operators, $R: {H_2}\to {H_1}$ and $T: {H_2}\to {H_2}$, 
such that 
\begin{itemize}
	\item $T$\  is self-adjoint, i.e.~$T=T^*$, and \vspace{1mm}
	\item $T$ is positive, i.e.~$\l Tg,g\r_{H_2}\ge0$ for all $g\in {H_2}$, 
\end{itemize}
We denote by $\Vert \cdot\Vert_{H_1}$ 
(resp.~$\Vert \cdot\Vert_{{H_2}\to {H_1}}$) the operator norms of 
operators mapping from ${H_1}$ to ${H_1}$ (resp.~${H_2}$ to ${H_1}$), 
i.e.
\[
\Vert R\Vert_{{H_2}\to {H_1}} \,:=\, 
	\max_{\Vert g\Vert_{H_2}=1} \Vert Rg\Vert_{H_1}
\]
and $\Vert \cdot\Vert_{H_1}$ as in \eqref{eq:norm}.
As is well-known, the adjoint operator of $R$, 
i.e.~$R^*: {H_1}\to {H_2}$ with $\l R^*f,g\r_{H_2}=\l f,R g\r_{H_1}$ 
for all $f\in {H_1}$ and $g\in {H_2}$, satisfies 
$\Vert R^*\Vert_{{H_1}\to {H_2}}=\Vert R\Vert_{{H_2}\to {H_1}}$  
(see e.g.~Kreyszig~\cite[Thm.~3.9-2]{Krey}).
Additionally note that self-adjointness of $T$ implies that 
$R T^k R^*$, $k\in\N$, are self-adjoint operators 
on $H_1$.

\begin{lemma} \label{lemma:tech_mon}
Let $T$ and $R$ be as above. Then
\[
\norm{R T^{k+1} R^*}_{{H_1}} \;\le\; \norm{T}_{{H_2}}\,\norm{R T^k R^*}_{{H_1}}
\; \quad \text{ for all } k\in\N.
\]
In particular, if $\norm{T}_{{H_2}}\le1$ this proves monotonicity in $k$.
\end{lemma}

\begin{proof}
By the assumptions, $T$ has a unique positive square root $\widetilde{T}$, 
i.e.~$T=\widetilde{T} \widetilde{T}$, which is again self-adjoint (see 
e.g.~\cite[Th. 9.4-2]{Krey}). 
Using the fact that $\norm{A}_{H_2\to H_1}^2=\norm{A A^*}_{H_1}$ for every bounded 
linear operator $A:H_2\to H_1$, we obtain
\[\begin{split}
\norm{R T^{k+1} R^*}_{{H_1}}& \;=\; \norm{R \widetilde{T}^{2k+2} R^*}_{{H_1}}
\;=\; \norm{R \widetilde{T}^{k+1}}_{{H_2}\to {H_1}}^2 \\
\;&\le\; \norm{R \widetilde{T}^k}_{{H_2}\to {H_1}}^2 \norm{\widetilde{T}}_{{H_2}}^2 
\;=\; \norm{R \widetilde{T}^{2k} R^*}_{{H_1}}  \norm{T}_{{H_2}} \\
\;&=\; \norm{T}_{{H_2}}\,\norm{R T^k R^*}_{{H_1}}.
\end{split}\]
\end{proof}

\begin{lemma} \label{lemma:tech_in}
In the above setting let additionally 
$\norm{R}_{{H_2}\to {H_1}}^2 = \norm{R R^*}_{H_1} \le1$. Then 
\[
\norm{R T R^*}_{{H_1}}^{2^k} \;\le\; \norm{R T^{2^k} R^*}_{{H_1}}
\; \quad \text{ for all } k\in\N.
\]
\end{lemma}

\begin{proof}
The case $k=0$ is obvious. Now suppose the statement is correct for $k-1$; 
then
\[\begin{split}
\norm{R T R^*}_{{H_1}}^{2^k} \;&=\; \norm{R T R^*}_{{H_1}}^{2^{k-1}\, 2} 
\;\le\; \norm{R T^{2^{k-1}} R^*}_{{H_1}}^2 \\
&\le\; \norm{R T^{2^{k-1}}}_{{H_2}\to {H_1}}^2 \norm{R^*}_{{H_1}\to {H_2}}^2
\;=\; \norm{R T^{2^{k-1}} T^{2^{k-1}} R^*}_{{H_1}} \norm{R R^*}_{{H_2}} \\
&\le\; \norm{R T^{2^k} R^*}_{{H_1}},
\end{split}\]
which proves the statement for $k$. 
\end{proof}

The next corollary combines the statements of the last two lemmas to give 
a result similar to Lemma \ref{lemma:tech_in} for arbitrary exponents.

\begin{corollary} \label{coro:tech}
Additionally to the general assumptions of this section let 
$\norm{T}_{{H_2}}\le1$ and $\norm{R R^*}_{{H_1}}\le1$. Then
\[
\norm{R T R^*}_{{H_1}}^{2 k} \;\le\; \norm{R T^k R^*}_{{H_1}}
\; \quad \text{ for all } k\in\N.
\]
\end{corollary}

\begin{proof}
Let $\ell=\lceil\log_2 k \rceil$, so that $k\le2^\ell\le 2k$.
Since $\norm{R T R^*}_{{H_1}}\le1$ by assumption, we obtain 
\[
\norm{R T R^*}_{{H_1}}^{2 k} \;\le\; \norm{R T R^*}_{{H_1}}^{2^\ell} 
\;\stackrel{\text{\scriptsize \ref{lemma:tech_in}}}{\le}\; 
	\norm{R T^{2^\ell}  R^*}_{{H_1}}
\;\stackrel{\text{\scriptsize \ref{lemma:tech_mon}}}{\le}\; 
	\norm{R T^k  R^*}_{{H_1}}.
\]
\end{proof}

In the following section
we will apply these bounds for a specific 
choice of $R$ and $T$.

\vspace{2mm}


\section{The result} \label{sec:4_result}

We prove the following theorem (see \cite{U4,U3}).

\begin{theorem} \label{th:main-SB}
Let $\tsw$ $($resp.~$\sb$$)$ be the transition matrix of the 
Swendsen--Wang $($resp.~single-bond$)$ dynamics for the 
random-cluster model on a graph with $m \ge 3$ edges. Then 
\[
\lambda(\sb) \;\le\; \lambda(\tsw) \;\le\; 8 m \log m\; \lambda(\sb).
\]
\end{theorem}
\vspace*{3mm}

This theorem shows that the Swendsen--Wang dynamics is rapidly mixing 
if and only if the single-bond dynamics is rapidly mixing, 
since the spectral gaps can differ only by a polynomial in the 
number of edges of the graph.
In Section~\ref{sec:4_app} we will state some of the results that 
follow from the (already discussed) results for the Swendsen--Wang 
dynamics.

The proof of Theorem~\ref{th:main-SB} is based on the bounds of 
the last section and a suitable choice of the 
involved operators.
Recall that we consider both dynamics on a graph $G=(V,E)$ with 
$m$ edges, i.e.~$m=\abs{E}$. 
Fix an arbitrary ordering $e_1,\dots,e_m$ of the edges $e\in E$. 
We set the Hilbert spaces from Section \ref{sec:4_technical} to 
\[
{H_1}:=L_2(\mu) \quad \text{ and }\quad  {H_2}:=L_2(\nu)
\] 
and define the operators 
\begin{equation*} 
T \;:=\; \frac{1}{m}\,\sum_{i=1}^m\, T_{e_i}
\quad
\mbox{and}\quad
\T \;:=\; \prod_{i=1}^m\, T_{e_i},
\end{equation*}
where the $T_e$, $e\in E$, are 
from the common representation in 
\eqref{eq:T-e}. 
Recall that $\sb=M T M^*$ and $\tsw=M \T M^*$ 
by Lemma~\ref{lemma:repr} with $M$ from \eqref{eq:M}.
Additionally we define
\begin{equation*} 
\T_\a \;:=\; \prod_{i=1}^m\, T_{e_i}^{\a_i}
\end{equation*}
for $\a\in\N^m$. By Lemma~\ref{lemma:repr-prop}(ii) we deduce for 
$\a, \g\in\N^m$ that 
$\T_\a=\T_\g$ if and only if $\{i:\,\a_i=0\}=\{i:\,\g_i=0\}$. 
Furthermore, $\T_\a=\T$ iff $\a_i>0$ for all 
$i=1,\dots,m$.  

To conclude the proof of Theorem~\ref{th:main-SB} we 
need two lemmas that are stated in the sense of 
Section~\ref{sec:4_technical}. 
Afterwards we will see how this implies the result. 
The first lemma follows from Lemma~\ref{lemma:tech_mon}.

\begin{lemma}\label{lemma:norm1}
Let $R:{H_2}\to {H_1}$ be a bounded linear operator. 
Then
\[
\norm{R \T R^*}_{H_1} \;\le\; \norm{R T^k R^*}_{H_1}
\; \quad \text{ for all } k\in\N.
\]
\end{lemma}

\begin{proof}
See Remark~\ref{remark:SB-SW_pos-def} for arguments that 
show positivity of $T$. 
Additionally, Lemma~\ref{lemma:repr-prop} leads to 
$\norm{T}_{H_2}\le\frac{1}{m}\,\sum_{i=1}^m\,\norm{T_{e_i}}_{H_2} 
=1$. 
Hence $T$ satisfies the assumptions of Lemma~\ref{lemma:tech_mon}. 
We obtain  
\[
\norm{R T^\ell R^*}_{{H_1}} \;\le\; \norm{R T^k R^*}_{{H_1}} 
\ \text{ for every }\  k\le\ell.
\]
Note that $\lim_{\ell\to\infty} T^\ell = \T$ (in norm topology) 
implies 
$\lim\limits_{\ell\to\infty}\norm{RT^\ell R^*}_{H_1}=\norm{R\T R^*}_{H_1}$, 
which yields the result.
\end{proof}

Lemma~\ref{lemma:norm1} (or \ref{lemma:tech_mon}) shows that, 
while $k$ approaches infinity, $\norm{R T^k R^*}_{H_1}$ 
monotonically approaches $\norm{R \T R^*}_{H_1}$. 
This suggests that, for $k$ large enough, these norms 
are close to each other.
The next lemma yields such a reverse inequality.

\vspace{2mm}

\begin{lemma} \label{lemma:norm2}
Let $k = \lceil m \log\frac{m}{\eps}\rceil$ and $R:{H_2}\to {H_1}$ be a bounded 
linear operator with $\norm{R R^*}_{H_1}\le1$. Then
\[
\norm{R T^k R^*}_{H_1} \;\le\; (1-\eps) \,\bigl\|R \T R^*\bigr\|_{H_1} \,+\, \eps.
\]
\end{lemma}
\vspace{2mm}

\begin{proof}
Define 
the index sets $I_{m,k}:=\{\a\in\N^m: \sum_{i=1}^m \a_i =k\}$ and 
$I^1_{m,k}:=\{\a\in I_{m,k}:  \a_i>0, \;\forall i=1,\dots,m\}$. 
Let $I^0_{m,k}:=I_{m,k}\setminus I^1_{m,k}$ and denote by 
$\binom{k}{\a}$, for $\a\in I_{m,k}$, the multinomial coefficient, 
i.e. 
\[
\binom{k}{\a} \,=\, \frac{k!}{\a_1! \a_2! \cdots \a_m!}.
\]
Obviously (by the multinomial theorem \cite[eq.~(2.21)]{HHM}), 
\[
\sum_{\a\in I_{m,k}} \,\binom{k}{\a} \;=\; m^k
\] 
and 
\[\begin{split}
Z_{m,k} \;:=\; \sum_{\a\in I^0_{m,k}} \,\binom{k}{\a}
\;&\le\; \sum_{i=1}^m\,\sum_{\a\in I^0_{m,k}: \a_i=0} \,\binom{k}{\a}
\;=\; m\,\sum_{\g\in I_{m-1,k}} \,\binom{k}{\g}
\;=\; m(m-1)^k.
\end{split}\] 
We write
\[\begin{split}
T^k \;&=\; \left(\frac1m\sum_{i=1}^m T_{e_i}\right)^k 
\;=\; \frac{1}{m^k} \sum_{\a\in I_{m,k}}\,\binom{k}{\a}\, \T_\a \\
&=\; \frac{1}{m^k} \sum_{\a\in I^1_{m,k}}\,\binom{k}{\a}\, \T_\a 
		\;+\; \frac{1}{m^k} \sum_{\a\in I^0_{m,k}}\,\binom{k}{\a}\, \T_\a.
\end{split}\]
Note that we use for the second equality the fact that the $T_e$'s are commuting by 
Lemma~\ref{lemma:repr-prop}(ii). 
Since we know that $\T_\a=\T$ for every $\a\in I^1_{m,k}$ 
(note that $I^1_{m,k}=\varnothing$ for $k< m$) and 
$\norm{R \T_\a R^*}_{H_1}\le1$ for every $\a\in I_{m,k}$, we obtain
\[\begin{split}
\norm{R T^k R^*}_{H_1} 
\;&\le\; \frac{1}{m^k} \sum_{\a\in I^1_{m,k}}\,\binom{k}{\a}\, 
						\bigl\|R \T_\a R^*\bigr\|_{H_1}
	\;+\; \frac{1}{m^k} \sum_{\a\in I^0_{m,k}}\,\binom{k}{\a}\, 
						\bigl\|R \T_\a R^*\bigr\|_{H_1} \\
&\le\; \left(1-\frac{Z_{m,k}}{m^k}\right) \,\bigl\|R \T R^*\bigr\|_{H_1}  
	\;+\; \frac{Z_{m,k}}{m^k}.
\end{split}\]
Using $(1-a)c+a\le(1-b)c+b$ for $c\le1$ and $a\le b$, 
we find that
\[
\norm{R T^k R^*}_{H_1} \;\le\; 
	\left(1-m\left(1-\frac1m\right)^k\right)\, \bigl\|R \T R^*\bigr\|_{H_1}  
		\;+\; m\left(1-\frac1m\right)^k.
\]
Setting $k = \lceil m \log\frac{m}{\eps}\rceil$ yields the result.
\end{proof}

Now we are able to prove the comparison result for SW and SB dynamics. 
For this let $S_1(B,(\sigma,A)):=\nu(\sigma,A)$ for all 
$B\in\O_{\rm RC}$ and $(\sigma,A)\in\O_{\rm J}$, which defines 
an operator (by \eqref{eq:map2}) that maps from $H_2$ to $H_1$. 
The adjoint operator $S_1^*$ is then given by 
$S_1^*((\sigma,A),B):=\mu(B)$, and thus 
$S_1 S_1^*(A,B) = S_{\mu}(A,B) = \mu(B)$ for all $A,B\in\O_{\rm RC}$. 
For the proof we choose the operator
\[
R \;:=\; M - S_1
\]
with $M$ from \eqref{eq:M}.
This operator has some useful properties which can be deduced 
directly from the properties of $M$ and $S_1$.
First of all, note that the matrix corresponding to $S_1^*$ is 
constant in the first (\jointname) coordinate. 
This readily implies $M S_1^* = S_\mu$. 
But $S_\mu$ defines a self-adjoint operator on ${H_1}:=L_2(\mu)$, 
and hence $S_\mu=S_\mu^*=(M S_1^*)^*=S_1 M^*$.
It follows that 
\[
R R^* \,=\, (M-S_1)(M^*-S_1^*) \,=\, M M^* - S_\mu 
\,=\, I - S_\mu,
\] 
with $If:=f$ for all $f\in L_2(\mu)$, because 
$M M^*(A,B)=\ind(A=B)$. 
Since $(R R^*)^2 = (I - S_\mu)^2 = I-S_\mu = R R^*$, 
we deduce that $R$ satisfies $\norm{R R^*}_{H_1}\in\{0,1\}$ 
and thus the assumptions of Lemmas~\ref{lemma:norm1} and 
\ref{lemma:norm2}. 
Note that $\norm{R R^*}_{H_1}=1$ whenever $m>1$, since 
then there always exists a non-constant function $f\in H_1$ 
with $\l f,1\r_\mu=0$, i.e.~$(I-S_\mu)f=f$.

Additionally, we find from Lemma~\ref{lemma:repr} that 
\begin{align*}
\sb-S_\mu \;&=\; R T R^* \\[-5mm]
\intertext{\vspace*{-2mm} and}
\tsw-S_\mu \;&=\; R \T R^*.
\end{align*}

We finish the section with the conclusion of the proof.

\begin{proof}[Proof of Theorem~\ref{th:main-SB}]
Recall that by definition 
\[
\lambda(\sb) \;=\; 1-\norm{\sb-S_\mu}_{H_1}
\;=\; 1-\norm{R T R^*}_{H_1}
\]
and 
\[
\lambda(\tsw) \;=\; 1-\norm{\tsw-S_\mu}_{H_1}
\;=\; 1-\norm{R \T R^*}_{H_1}.
\]
Obviously, Lemma~\ref{lemma:norm1} (with $k=1$) implies 
the first inequality of Theorem~\ref{th:main-SB}.

Now let $k = \lceil m \log\frac{m}{\eps}\rceil$. 
Since we know from Lemma~\ref{lemma:norm2} that 
\[
\norm{R \T R^*}_{H_1} \,\ge\, 
\frac{1}{1-\eps}\left(\norm{R T^k R^*}_{H_1}-\eps\right)
\]
we obtain
\[\begin{split}
\lambda(\tsw) \;&=\; 1-\norm{R \T R^*}_{H_1} 
\;\stackrel{L.\text{\scriptsize \ref{lemma:norm2}}}{\le}\; 
		1-\frac{1}{1-\eps}\left(\norm{R T^k R^*}_{H_1}-\eps\right) \\
&=\; \frac{1}{1-\eps}\left(1-\norm{R T^k R^*}_{H_1}\right)
\;\stackrel{Coro.\text{\scriptsize \ref{coro:tech}}}{\le}\; 
		\frac{1}{1-\eps}\left(1-\norm{R T R^*}_{H_1}^{2k}\right) \\
&\le\; \frac{2k}{1-\eps}\left(1-\norm{R T R^*}_{H_1}\right) 
\;=\; \frac{2k}{1-\eps}\,\lambda(\sb), 
\end{split}\]
where the last inequality comes from $1-x^k\le k(1-x)$ for $x\in[0,1]$.
Setting $\eps=\frac12$, we obtain 
$\frac{2k}{1-\eps}=4k\le 8 m \log m$. This proves the statement.
\end{proof}

\vspace{2mm}


\section{Applications} \label{sec:4_app}

We present three applications of Theorem~\ref{th:main-SB}.
The first one concerns the Swendsen--Wang dynamics 
on graphs with bounded linear width 
and is based on a result of  Ge and {\v{S}}tefankovi{\v{c}} 
\cite{GeS} (see Theorem~\ref{th:SB_linear-width} above).
For this recall the definition of the \emph{linear width} 
of a graph $G=(V,E)$ as the smallest number $\ell$ such 
that there exists an ordering $e_1,\dots,e_{\abs{E}}$ 
of the edges with the property that for every $i\in[\abs{E}]$ 
there are at most $\ell$ vertices that have an adjacent edge 
in $\{e_1,\dots, e_i\}$ and in $\{e_{i+1},\dots, e_{\abs{E}}\}$.

\begin{corollary} \label{coro:SW-width}
Let $\tsw$ be the transition matrix 
of the Swendsen--Wang dynamics for the random-cluster model 
with parameters $p$ and $q\in\N$ 
on a graph $G=(V,E)$ with linear width bounded by $\ell$. 
Let $m:=\abs{E}$. Then
\begin{equation} \label{eq:coro_SW-width}
\lambda(\tsw)
\;\ge\; \frac{1}{2\, q^{\ell+1}}\ \frac{1}{m^2}.
\vspace{2mm}
\end{equation}
\end{corollary}

\begin{proof}
The result follows from Theorem \ref{th:main-SB} together with 
Theorem~\ref{th:SB_linear-width}.
\end{proof}

Note that Lemma~\ref{lemma:SW_P-RC} shows that 
the result holds true if we replace the random-cluster model by 
the corresponding Potts model at inverse temperature 
$\beta=-\log(1-p)$ and $\tsw$ by $\sw$.
\vspace{1mm}

Now we turn to the single-bond dynamics. 
The second result of this section, analogously to 
Theorems~\ref{th:P-HB_2d} and \ref{th:P-SW_2d_high}, 
describes rapid mixing on $\Z^2_L=(V_{L,2},E_{L,2})$, 
i.e.~on the two-dimensional square lattice of side length $L$, 
at high temperatures. 
\goodbreak

\begin{corollary} \label{coro:SB_2d_high}
Let $\sb$ be the transition matrix of the 
single-bond dynamics for the RC model 
on $\Z^2_L$ with parameters $p$ and $q\ge1$.
Let $m=2L(L-1)=\abs{E_{L,2}}$. 
Then there exist constants $c_p=c_p(q),c'>0$ and $C<\infty$ 
such that
\begin{align*}
\lambda(\sb) \;&\ge\; \frac{c_p}{m^2 \log m} \qquad\quad
	\text{ for } p < p_c(q) \hspace*{2cm}\\[-3mm]
\intertext{\vspace*{-2mm} and}
\lambda(\sb) \;&\ge\; c'\,m^{-C} \qquad\qquad 	
	\text{ for } q=2 \text{ and } p = p_c(2),
\end{align*}
where $p_c(q)\,=\,\frac{\sqrt{q}}{1+\sqrt{q}}$.
\end{corollary}

\begin{proof}
The bounds for the Swendsen--Wang (SW) dynamics for the 
Potts model from Theorem~\ref{th:P-SW_2d_high} and 
Lemma~\ref{lemma:SW_P-RC} show that SW is rapidly mixing 
in the desired range of $p$. 
Theorem~\ref{th:main-SB} leads to the specific bounds 
(since 
$\abs{V_{L,2}}\le\abs{E_{L,2}}=2L(L-1)\le2\abs{V_{L,2}}$). 
\end{proof}
\vspace{1mm}

In the next chapter we present techniques that allow us to 
relate the spectral gap of local Markov chains for the 
random-cluster model at high and low temperatures. 
In particular, we will extend the first bound of 
Corollary~\ref{coro:SB_2d_high} to all non-critical 
temperatures (see Theorem~\ref{th:SB_2d}).

The third and last result that we want to present here 
yields tight bounds on the spectral gap of the 
single-bond dynamics on the $d$-dimensional torus $\Zt_L^d$ 
of side length $L$ (see \eqref{eq:def-torus}), 
at the critical temperature. These bounds show slow mixing 
if $q$ is large enough.
This complements (and also uses) a result of 
Borgs, Chayes and Tetali~\cite{BCT} (see 
Theorem~\ref{th:P-SW-HB_torus}).

\vspace{2mm}

\begin{theorem}\label{th:RC-SW-SB_torus}
Let $\sb$ $($resp.~$\tsw$$)$ be the transition matrix of the single-bond 
$($resp. Swendsen--Wang$)$ dynamics for the random-cluster model  
on $\Zt_L^d$, $d\ge2$, with parameters $p$ and $q$. 
Then there exist constants $k_1,k_2<\infty$ 
and a constant $k_3>0$ 
$($depending all on $d$, $\beta$ and $q$$)$
such that, for $q$ and $L$ large enough, 
\begin{align*}
e^{-(k_1 + k_2 \beta) L^{d-1}} \,&\le\, \lambda(\sb) 
	\,\le\, e^{-k_3 \beta L^{d-1}} \\[-3mm]
\intertext{\vspace*{-2mm} and}
e^{-(k_1 + k_2 \beta) L^{d-1}} \,&\le\, \lambda(\tsw) 
	\,\le\, e^{-k_3 \beta L^{d-1}}
\end{align*}
if $p=1-e^{-\beta_c(d,q)}$ with $\beta_c(d,q)$ from 
\eqref{eq:critical}. 
The lower bounds hold for all $p$, $q$ and $L$. 
\end{theorem}
\vspace{1mm}

\begin{proof}
We obtain the second inequality immediately from 
Theorem~\ref{th:P-SW-HB_torus} and Lemma~\ref{lemma:SW_P-RC}. 
The result on the single-bond dynamics thus follows from 
Theorem~\ref{th:main-SB}.
\end{proof}
\vspace{1mm}

The original proof of the lower bound on $\lambda(\sb)$, as given in 
\cite{U3}, uses the bound of Theorem~\ref{th:SB_linear-width} 
together with a bound on the linear width of $\Zt_L^d$.

\begin{remark}
There are some more results on rapid mixing of the single-bond 
dynamics that follow directly from 
Theorem~\ref{th:main-SB}. These include rapid mixing 
on trees and cycles (cf.~Theorem~\ref{th:SW_tree-cycle}), 
on the complete graph (cf.~Theorem~\ref{th:SW_complete}) and 
on graphs with bounded degree if $p$ is small enough  
(cf.~Theorem~\ref{coro:SW_degree}). 
We do not state them here explicitly, 
because our main purpose is to prove rapid mixing of the
Swendsen--Wang dynamics.
\end{remark}


\chapter{Rapid mixing in two dimensions} \label{chap:2d}

This chapter is devoted to the study of the spectral gaps of 
Swendsen--Wang and single-bond dynamics if the underlying graph 
for the random-cluster model has a special structure.
Namely, we consider \emph{planar} graphs, i.e.~graphs that 
can be drawn in the plane without intersecting edges.
For each such graph it is possible to define a corresponding 
\emph{dual} graph and we will see that it is possible to 
prove bounds on the spectral gap (of both dynamics) on the 
original graph in terms of the spectral gap on the dual one, 
if the temperature is suitably changed.
In particular, we prove results on rapid mixing of the dynamics 
at low temperatures, where the previous techniques, 
which rely on rapid mixing of single-spin dynamics for the 
Potts model at the same temperature, do not apply.

The plan of this chapter is as follows.
In Section~\ref{sec:5_embeddings} we introduce the notion of 
representations of graphs in the plane and, consequently, 
the notion of planar graphs. This leads to the construction of 
dual graphs and we will see that there is a tight connection between 
the random-cluster model on a graph and its dual. 
In Section~\ref{sec:5_result} we use this connection to 
prove the main results of this chapter, i.e.~the 
comparison of the spectral gaps on a graph and on its dual graph.
This ends in a proof of rapid mixing of the considered 
dynamics on the two-dimensional square lattice at all 
non-critical temperatures, see Section~\ref{sec:5_square}.
Finally, in Section~\ref{sec:5_genus}, we present a slight 
generalization of the given results to graphs that can be drawn 
on a surface with bounded genus.


\section{Planar and dual graphs} \label{sec:5_embeddings}

The following introduction to embeddings of graphs in the plane 
and to the construction of dual graphs is presented 
according to Mohar and Thomassen~\cite{Mohar} and we refer 
to this monograph for a more comprehensive study of the topic.

Throughout this chapter we consider only connected graphs. 
Let a connected graph $G=(V,E,\varphi)$ be given and recall that 
$\varphi$ is the function that 
assigns to each edge $e\in E$ the set of its endvertices. 
For convenience, we will mostly omit $\varphi$ from the 
notation. 
We say that $G$ is a 
\emph{planar graph}\index{graph!planar}
if there is a 
\emph{representation of $G$ in the plane $\R^2$}.
That is, there exists a finite set $\Vt\subset\R^2$, 
a bijection $\kappa$ from $V$ to $\Vt$ 
and a set of 
\emph{simple curves} 
$\,\Et=\{\g_e:[0,1]\to\R^2:\, e\in E\}$ in $\R^2$ such that 
\vspace{1mm}
\begin{itemize}
	\item $\g_e(0)\in\Vt$ and $\g_e(1)\in\Vt$ for all $e\in E$,
		\vspace{1mm}
	\item for all $u,v\in V$ and $e\in E$ we have 
				\[
				\bigl\{\g_e(0), \g_e(1)\bigr\}
								=\bigl\{\kappa(u),\kappa(v)\bigr\}
				\;\;\Longleftrightarrow\;\; 
				\varphi(e)=\{u,v\} 
				\]
		and 
		\vspace{1mm}
	\item $\g_e\bigl((0,1)\bigr)\cap\g_f\bigl([0,1]\bigr)=\varnothing$ 
					for all $e\neq f\in E$.
\end{itemize}
Note that a curve $\g:[0,1]\to\R^2$ is always assumed to be 
continuous, and is called 
\emph{simple}\index{simple curve} 
if it does not cross itself except for its ends, 
i.e.~$\g(x)=\g(y)$ for $x\neq y$ implies $x,y\in\{0,1\}$.
The above assumptions mean respectively that the ends of all curves are 
elements of $\Vt$, that $\g_e$ connects the elements of $\Vt$ that 
correspond (by $\kappa$) to the endvertices of $e\in E$ and that 
the interior of each curve is disjoint from all other curves. 
The tuple $\Gt=(\!\Vt,\Et)$ is called a 
\emph{representation of $G$ in the plane}\index{graph!representation} 
and we assume henceforth that we fix for every planar graph $G$ 
a representation $\Gt$ in the plane. 
In the literature $\Gt$ is also called a \emph{plane graph}, 
\emph{planar embedding} or \emph{drawing} of $G$.
See e.g.~Figure~\ref{fig:graphs} for two examples of planar graphs 
(or, more precisely, their representations in the plane).

Clearly, being connected by a simple curve in a set $C\subset\R^2$ 
defines an equivalence relation on $C$, whose equivalence classes 
are called \emph{regions of $C$}.
A \emph{face of} $C$ is a region of $\R^2\setminus C$ and we call 
the faces of $\bigcup_{e\in E}\g_e\bigl([0,1]\bigr)$, i.e.~the faces 
of the union of the images of all curves of $\Gt$, the 
\emph{faces of $\Gt$}. 
Note 
that since we consider finite graphs, there is exactly one 
unbounded face. We call it the \emph{outer face of} $\Gt$.
In the following we need the number of faces of a representation 
$\Gt$ of a graph $G$, which we denote by $F(\Gt)$. 
One may think that this number depends on the choice of the 
representation, but the following lemma, which is known as 
\emph{Euler's $($polyhedral\/$)$ formula}, shows that this 
is not the case (see e.g.~\cite[Prop.~2.2.3]{Mohar}).

\begin{lemma}[Euler's formula]\label{lemma:euler}
Let $G=(V,E)$ be a connected planar graph and $\Gt$ be a 
representation of $G$ in the plane. Then
\[
2 \,=\, \abs{V} - \abs{E} + F(\Gt).
\]
\end{lemma}
 
This proves that the number of faces does not depend 
on the particular representation and thus we write 
$F(G)$ for the 
\emph{number of faces}\index{graph!number of faces} 
of any representation of $G$. 

In addition to the statement of Lemma~\ref{lemma:euler} for the 
whole graph $G$, we need a result on its spanning subgraphs. 
For this recall that we denote by $G_A=(V,A)$, $A\subset E$, 
the graph $G$ without the edges of $E\setminus A$ and that 
$c(G_A)$ (or simply $c(A)$) denotes the number of connected 
components of $G_A$. The representation of $G_A$ in the plane 
is fixed to be the representation of $G$ without the curves 
corresponding to the edges of $E\setminus A$ 
and we write $F(A)$ for the number of faces of the 
(fixed) representation of $G_A$.
The next equation is a simple corollary of Lemma~\ref{lemma:euler}.

\begin{corollary}\label{coro:euler}
Let $G$ be a planar graph and $G_A$, $A\subset E$, be a spanning 
subgraph of $G$. Then
\[
c(A) + 1 \,=\, \abs{V} - \abs{A} + F(A).
\]
\end{corollary}

\begin{proof}
We know from Lemma~\ref{lemma:euler} that for each connected 
component $G_i=(V_i,E_i)$ of $G_A$, $i=1,\dots,c(A)$, we have 
$2 \,=\, \abs{V_i} - \abs{E_i} + F(G_i)$. 
Summing over $i$ leads to 
$2 c(A) \,=\, \abs{V} - \abs{A} + \sum_{i=1}^{c(A)} F(G_i)$, 
but in the last sum we count the outer face $c(A)$ times. 
Subtracting $c(A)-1$ on both sides yields the result.
\end{proof}
\vspace{1mm}

Now we turn to the definition of the 
\emph{dual graph}\index{graph!dual} 
of a graph $G$. 
For this consider the (fixed) representation in the plane 
$\Gt=(\Vt,\Et)$ 
of the connected graph $G=(V,E)$. We define the dual graph 
as the (unique) graph $G^\dag=(V^\dag, E^\dag)$ such that its 
representation in the plane $\Gt^\dag=(\Vt^\dag, \Et^\dag)$ 
satisfies: 
\begin{itemize}
	\item there is exactly one 
					$\tilde{v}^\dag\in\Vt^\dag$ in every face of $\Gt$, and
		\vspace{1mm}
	\item for every $e\in E$ there is exactly one simple curve 
					$\g_e^\dag\in E^\dag$ with endpoints in $\Vt^\dag$ 
					that intersects $\g_e\in\Et$, but not $\g_f$, $f\neq e\in E$.
\end{itemize}
Note that, for every $\g_e\in\Et$, the endpoints of the 
corresponding ``dual'' curve $\g_e^\dag$ are unique, 
since there is only one vertex in every face of 
$\Gt$ and a curve can separate at most two faces from each other 
(by Jordan's Curve Theorem~\cite{Mohar}).
We call the edge $e_\dag\in E^\dag$ in the dual graph that 
belongs to the dual curve $\g_e^\dag$ the 
\emph{dual edge of}\index{dual!edge} 
$e\in E$.

In general, the definition of the dual graph of $G$ depends on the 
representation of $G$ in the plane. 
Therefore recall that we fix such a representation whenever we 
fix the graph. 
This implies 
that the dual graph is well-defined, unique and that 
the dual of the dual graph is the original graph. 
(Here we presume that the representation in the plane of the 
dual graph is fixed to be the one constructed above.)

\begin{remark}\label{remark:3-connected}
Note that the dual graph of any planar graph is connected.
Therefore the connectivity of $G$ is necessary to ensure 
that the dual of the dual graph of $G$ is~$G$.
Additionally, it is known (see \cite[Thm.~2.6.7]{Mohar}) that if we 
restrict to 3-connected planar graphs, i.e.~planar graphs 
with at least four vertices such that every subgraph obtained 
by deleting two vertices is still connected, then the dual 
graph is unique and thus independent of the representation. 
\end{remark}

In view of the application to Markov chains for the random-cluster 
model we also define a 
\emph{dual configuration}\index{dual!configuration} 
$A^\dag\subset E^\dag$ on the dual graph $G^\dag$ for every 
random-cluster configuration $A\subset E$ on $G$ by the 
property that for all $e\in E$ 
(and corresponding $e_\dag\in E^\dag$)
\begin{equation}\label{eq:dual-conf}
e_\dag\in A^\dag \quad\iff\quad e\notin A.
\end{equation}
Thus, a dual edge is present in $A^\dag$ whenever the original one 
is absent in $A$. 
In particular, $\abs{E}=\abs{E^\dag}=\abs{A}+\abs{A^\dag}$ for 
every $A\subset E$.
See Figure~\ref{fig:dual} for an example of a graph and its dual 
together with a proper pair of RC configurations.

\setcounter{figure}2
\vspace{-4mm}
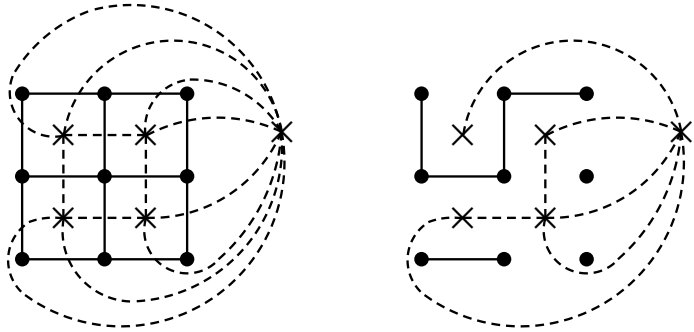
\begin{figure}[ht]
\begin{center}
\scalebox{1.1}{
	\psset{xunit=1.0cm,yunit=1.0cm,algebraic=true,dotstyle=*,dotsize=3pt 0,linewidth=0.8pt,arrowsize=3pt 2,arrowinset=0.25}
	\begin{pspicture*}(0,0)(4.5,4.5)
	\psline[linewidth=0.8pt](1,3)(2,3)
	\psline[linewidth=0.8pt](2,3)(3,3)
	\psline[linewidth=0.8pt](3,3)(3,2)
	\psline[linewidth=0.8pt](3,2)(3,1)
	\psline[linewidth=0.8pt](3,1)(2,1)
	\psline[linewidth=0.8pt](2,1)(2,2)
	\psline[linewidth=0.8pt](2,2)(3,2)
	\psline[linewidth=0.8pt](2,2)(2,3)
	\psline[linewidth=0.8pt](1,3)(1,2)
	\psline[linewidth=0.8pt](1,2)(2,2)
	\psline[linewidth=0.8pt](2,1)(1,1)
	\psline[linewidth=0.8pt](1,1)(1,2)
	\psline[linewidth=0.8pt,linestyle=dashed,dash=3pt 2pt](1.5,1.5)(1.5,2.5)
	\psline[linewidth=0.8pt,linestyle=dashed,dash=3pt 2pt](1.5,2.5)(2.5,2.5)
	\psline[linewidth=0.8pt,linestyle=dashed,dash=3pt 2pt](2.5,2.5)(2.5,1.5)
	\psline[linewidth=0.8pt,linestyle=dashed,dash=3pt 2pt](2.5,1.5)(1.5,1.5)
	\parametricplot[linewidth=0.8pt,linestyle=dashed,dash=3pt 2pt]{0.1852913062710569}{2.989280090214289}{1*1.34*cos(t)+0*1.34*sin(t)+2.83|0*1.34*cos(t)+1*1.34*sin(t)+2.3}
	\parametricplot[linewidth=0.8pt,linestyle=dashed,dash=3pt 2pt]{1.1350083349985625}{2.0596018113318753}{1*1.84*cos(t)+0*1.84*sin(t)+3.37|0*1.84*cos(t)+1*1.84*sin(t)+0.87}
	\parametricplot[linewidth=0.8pt,linestyle=dashed,dash=3pt 2pt]{4.6717617082515}{5.883637174045906}{1*1.71*cos(t)+0*1.71*sin(t)+2.57|0*1.71*cos(t)+1*1.71*sin(t)+3.21}
	\parametricplot[linewidth=0.8pt,linestyle=dashed,dash=3pt 2pt]{1.5028552028831879}{2.000187227182034}
		{1*0.86*cos(t)+0*0.86*sin(t)+1.44|-0.075*0.86*cos(t)+1*0.86*sin(t)+0.64}
	\parametricplot[linewidth=0.8pt,linestyle=dashed,dash=3pt 2pt]{2.249045279311494}{3.9271577644321395}
		{1*0.54*cos(t)+0*0.54*sin(t)+1.37|0*0.54*cos(t)+1*0.54*sin(t)+0.99}
	\parametricplot[linewidth=0.8pt,linestyle=dashed,dash=3pt 2pt]{-2.1993684311435224}{0.19067988888988038}{1*1.98*cos(t)+0*1.98*sin(t)+2.2|0*1.98*cos(t)+1*1.98*sin(t)+2.17}
	\parametricplot[linewidth=0.8pt,linestyle=dashed,dash=3pt 2pt]{2.8453007051607155}{5.420388001593068}{1*0.52*cos(t)+0*0.52*sin(t)+3|0*0.52*cos(t)+1*0.52*sin(t)+1.35}
	\parametricplot[linewidth=0.8pt,linestyle=dashed,dash=3pt 2pt]{5.428195851138397}{6.21778798060977}{1*2.24*cos(t)+0*2.24*sin(t)+1.91|0*2.24*cos(t)+1*2.24*sin(t)+2.69}
	\parametricplot[linewidth=0.8pt,linestyle=dashed,dash=3pt 2pt]{1.751943081227404}{3.288962910742752}{1*0.59*cos(t)+0*0.59*sin(t)+3.08|0*0.59*cos(t)+1*0.59*sin(t)+2.58}
	\parametricplot[linewidth=0.8pt,linestyle=dashed,dash=3pt 2pt]{0.523108668073518}{1.593136560077821}{1*1.25*cos(t)+0*1.25*sin(t)+3.06|0*1.25*cos(t)+1*1.25*sin(t)+1.92}
	\parametricplot[linewidth=0.8pt,linestyle=dashed,dash=3pt 2pt]{2.956350314060717}{4.6749714687606275}{1*0.85*cos(t)+0*0.85*sin(t)+2.34|0*0.85*cos(t)+1*0.85*sin(t)+1.34}
	\parametricplot[linewidth=0.8pt,linestyle=dashed,dash=3pt 2pt]{-1.5125031955649737}{0.0989999773941644}{1*1.87*cos(t)+0*1.87*sin(t)+2.28|0*1.87*cos(t)+1*1.87*sin(t)+2.36}
	\parametricplot[linewidth=0.8pt,linestyle=dashed,dash=3pt 2pt]{2.777755480453707}{4.9277297505381314}{1*0.54*cos(t)+0*0.54*sin(t)+1.39|0*0.54*cos(t)+1*0.54*sin(t)+3.02}
	\parametricplot[linewidth=0.8pt,linestyle=dashed,dash=3pt 2pt]{1.9343071408090328}{2.583159724740975}{1*1.8*cos(t)+0*1.8*sin(t)+2.44|0*1.8*cos(t)+1*1.8*sin(t)+2.31}
	\parametricplot[linewidth=0.8pt,linestyle=dashed,dash=3pt 2pt]{0.17396397227159288}{1.8107485743137344}{1*1.85*cos(t)+0*1.85*sin(t)+2.32|0*1.85*cos(t)+1*1.85*sin(t)+2.22}
	\psdots[dotsize=5pt 0](1,3)
	\psdots[dotsize=5pt 0](1,2)
	\psdots[dotsize=5pt 0](1,1)
	\psdots[dotsize=5pt 0](2,3)
	\psdots[dotsize=5pt 0](2,2)
	\psdots[dotsize=5pt 0](2,1)
	\psdots[dotsize=5pt 0](3,3)
	\psdots[dotsize=5pt 0](3,2)
	\psdots[dotsize=5pt 0](3,1)
	\psdots[dotsize=8pt 0,dotstyle=x](1.5,1.5)
	\psdots[dotsize=8pt 0,dotstyle=x](2.5,1.5)
	\psdots[dotsize=8pt 0,dotstyle=x](1.5,2.5)
	\psdots[dotsize=8pt 0,dotstyle=x](2.5,2.5)
	\psdots[dotsize=8pt 0,dotstyle=x](4.15,2.54)
	\end{pspicture*}
}
\scalebox{1.1}{
	\psset{xunit=1.0cm,yunit=1.0cm,algebraic=true,dotstyle=*,dotsize=3pt 0,linewidth=0.8pt,arrowsize=3pt 2,arrowinset=0.25}
	\begin{pspicture*}(0,0)(4.5,4.5)
	\psline[linewidth=0.8pt](2,3)(3,3)
	\psline[linewidth=0.8pt](2,2)(2,3)
	\psline[linewidth=0.8pt](1,3)(1,2)
	\psline[linewidth=0.8pt](1,2)(2,2)
	\psline[linewidth=0.8pt](2,1)(1,1)
	\psline[linewidth=0.8pt,linestyle=dashed,dash=3pt 2pt](2.5,2.5)(2.5,1.5)
	\psline[linewidth=0.8pt,linestyle=dashed,dash=3pt 2pt](2.5,1.5)(1.5,1.5)
	\parametricplot[linewidth=0.8pt,linestyle=dashed,dash=3pt 2pt]{0.1852913062710569}{2.989280090214289}{1*1.34*cos(t)+0*1.34*sin(t)+2.83|0*1.34*cos(t)+1*1.34*sin(t)+2.3}
	\parametricplot[linewidth=0.8pt,linestyle=dashed,dash=3pt 2pt]{1.1350083349985625}{2.0596018113318753}{1*1.84*cos(t)+0*1.84*sin(t)+3.37|0*1.84*cos(t)+1*1.84*sin(t)+0.87}\parametricplot[linewidth=0.8pt,linestyle=dashed,dash=3pt 2pt]{4.6717617082515}{5.883637174045906}{1*1.71*cos(t)+0*1.71*sin(t)+2.57|0*1.71*cos(t)+1*1.71*sin(t)+3.21}
	\parametricplot[linewidth=0.8pt,linestyle=dashed,dash=3pt 2pt]{1.5028552028831879}{2.000187227182034}
		{1*0.86*cos(t)+0*0.86*sin(t)+1.44|-0.075*0.86*cos(t)+1*0.86*sin(t)+0.64}
	\parametricplot[linewidth=0.8pt,linestyle=dashed,dash=3pt 2pt]{2.249045279311494}{3.9271577644321395}
		{1*0.54*cos(t)+0*0.54*sin(t)+1.37|0*0.54*cos(t)+1*0.54*sin(t)+0.99}
	\parametricplot[linewidth=0.8pt,linestyle=dashed,dash=3pt 2pt]{-2.1993684311435224}{0.19067988888988038}{1*1.98*cos(t)+0*1.98*sin(t)+2.2|0*1.98*cos(t)+1*1.98*sin(t)+2.17}
	\parametricplot[linewidth=0.8pt,linestyle=dashed,dash=3pt 2pt]{2.8453007051607155}{5.420388001593068}{1*0.52*cos(t)+0*0.52*sin(t)+3|0*0.52*cos(t)+1*0.52*sin(t)+1.35}
	\parametricplot[linewidth=0.8pt,linestyle=dashed,dash=3pt 2pt]{5.428195851138397}{6.21778798060977}{1*2.24*cos(t)+0*2.24*sin(t)+1.91|0*2.24*cos(t)+1*2.24*sin(t)+2.69}
	\psdots[dotsize=5pt 0](1,3)
	\psdots[dotsize=5pt 0](1,2)
	\psdots[dotsize=5pt 0](1,1)
	\psdots[dotsize=5pt 0](2,3)
	\psdots[dotsize=5pt 0](2,2)
	\psdots[dotsize=5pt 0](2,1)
	\psdots[dotsize=5pt 0](3,3)
	\psdots[dotsize=5pt 0](3,2)
	\psdots[dotsize=5pt 0](3,1)
	\psdots[dotsize=8pt 0,dotstyle=x](1.5,1.5)
	\psdots[dotsize=8pt 0,dotstyle=x](2.5,1.5)
	\psdots[dotsize=8pt 0,dotstyle=x](1.5,2.5)
	\psdots[dotsize=8pt 0,dotstyle=x](2.5,2.5)
	\psdots[dotsize=8pt 0,dotstyle=x](4.15,2.54)
	\end{pspicture*}
}
\end{center}
\vspace*{-2mm}
\caption[Dual graphs]{Left: The graph $\Z^2_3$ (dots and solid lines) 
		and its dual graph $\Z^{2\dag}_3$ (crosses and dashed lines). 
		Right: A RC configuration on $\Z^2_3$ and its dual configuration.}
\label{fig:dual}
\end{figure}


Recall that $\Gt_A$, $A\subset E$, is the representation of 
the spanning subgraph $G_A=(V,A)$ in the plane 
that is adopted from $\Gt$ by 
forgetting the curves corresponding to $E\setminus A$, and 
$c(A)$ is the number of connected components of $G_A$. 
With a slight abuse of notation we write 
$c(A^\dag)$, $A^\dag\subset E^\dag$, for the number of connected 
components of $G^\dag_{A^\dag}$.

It is easy to verify (cf.~Figure~\ref{fig:dual}) that 
every face of $\Gt_A$ contains a unique connected component of 
$G^\dag_{A^\dag}$ and vice versa.
Thus, we have $c(A^\dag)=F(A)$.
Using Euler's formula (see Corollary~\ref{coro:euler}) we obtain 
\begin{equation}\label{eq:dual-components}
\begin{split}
c(A) \,+\, 1 \;&=\; c(A^\dag) \,+\, \abs{V} \,-\, \abs{A}  \\
\;&=\; c(A^\dag) \,+\, \abs{V} \,-\, \abs{E} \,+\, \abs{A^\dag}.
\end{split}
\end{equation}
This equality leads immediately to the following lemma that 
shows the tight connection between the random-cluster measure 
on $G$ and $G^\dag$ (see e.g.~Grimmett~\cite[eq.~(6.4)]{G1}). 

\begin{lemma}\label{lemma:RC-dual}
Let $\mu_{p,q}^G$ be the random-cluster measure on a planar 
graph $G=(V,E)$ with parameters $p$ and $q$. 
Furthermore, let $G^\dag$ be the dual graph of $G$. 
Then
\[
\mu_{p,q}^G(A) \;=\; \mu_{p^*,q}^{G^\dag}(A^\dag) 
\qquad \text{ for all }\, A\subset E,
\]
where $p^*$ satisfies 
\begin{equation}\label{eq:dual-p}
\frac{p^*}{1-p^*} \;=\; \frac{q(1-p)}{p}.
\end{equation}
\end{lemma}
\vspace{2mm}
\begin{proof}
Recall from \eqref{eq:RC} that
\[
\mu^G_{p,q}(A) \;=\; 
\frac1{\Ztt(G,p,q)}\,
\left(\frac{p}{1-p}\right)^{\abs{A}}\,q^{c(A)}, 
\]
where 
$\Ztt(G,p,q)=\sum_{A\subset E}(\frac{p}{1-p})^{\abs{A}}\,q^{c(A)}$ 
is the normalization constant.
With \eqref{eq:dual-components} and \eqref{eq:dual-p} we obtain
\[\begin{split}
\left(\frac{p}{1-p}\right)^{\abs{A}}\,q^{c(A)} 
\,&\stackrel{\eqref{eq:dual-components}}{=}\, 
	\left(\frac{p}{1-p}\right)^{\abs{A}}\,
		q^{c(A^\dag)\,+\, \abs{V}\,-\,\abs{E} \,+\,\abs{A^\dag}\,-\, 1} \\
\,&\stackrel{\eqref{eq:dual-p}}{=}\, 
	\left(\frac{q(1-p^*)}{p^*}\right)^{\abs{E}-\abs{A^\dag}}\,
		q^{c(A^\dag)\,+\, \abs{V}\,-\,\abs{E} \,+\,\abs{A^\dag}\,-\, 1} \\
\,&\;=\; q^{\abs{V} \,-\, 1}
	\left(\frac{p^*}{1-p^*}\right)^{\abs{A^\dag} \,-\, \abs{E}}\,
		q^{c(A^\dag)}.
\end{split}\]
Thus, with $\Ztt(G^\dag,p^*,q)=
	q^{1-\abs{V}}\bigl(p^*/(1-p^*)\bigr)^{\abs{E}}\,\Ztt(G,p,q)$ 
we get $\mu^G_{p,q}(A)=\mu^{G^\dag}_{p^*,q}(A^\dag)$ as desired. 
\end{proof}
\vspace{1mm}

The relation from \eqref{eq:dual-p} has a unique
\emph{self-dual point}, i.e.~a value of $p$ such that $p=p^*$.
This value is given by 
\begin{equation}\label{eq:self-dual-p}
p_{sd}(q) \;:=\; \frac{\sqrt{q}}{1+\sqrt{q}}.
\vspace{2mm}
\end{equation}
Note that $p_{sd}(q)$ equals $1-e^{\beta_c(q)}$, where $\beta_c(q)$ 
is the critical inverse temperature for the $q$-state Potts model 
on the two-dimensional square lattice 
(see Remark~\ref{remark:critical}).

\vspace{2mm}


\section{Dynamics on planar graphs} \label{sec:5_result}

We prove that the spectral gap of the 
Swendsen--Wang and the single-bond dynamics for the random-cluster 
model on a planar graph $G$ is bounded from above and from below 
by the spectral gap of the corresponding dynamics on the 
dual graph $G^\dag$ 
if we change the temperature parameter $p$ to $p^*$ from the last 
section.
%
Furthermore, we state a second corollary of Theorem~\ref{th:HB_degree} 
at the end of this section, that shows rapid mixing of the Swendsen--Wang 
dynamics for the Potts model on planar graphs if the 
inverse temperature is small or large enough 
(depending on the maximum degree). 
This is a modification of Corollary~\ref{coro:SW_degree} to 
planar graphs (cf.~Theorem~\ref{th:HB_degree}). 

For this fix a connected planar graph $G=(V,E)$ 
(together with its representation in the plane), some $p\in(0,1)$ 
and a natural number $q$. 
Let $G^\dag$ be the dual graph of $G$, and 
$p^*$ be the unique value satisfying \eqref{eq:dual-p}, 
i.e. $p^*=\frac{q(1-p)}{p+q(1-p)}$. 
We call the random-cluster model on $G^\dag$ with parameters 
$p^*$ and $q$ the 
\emph{dual model}\index{dual!model} 
and abbreviate $\mu_{p^*,q}^{G^\dag}$ to $\mu^\dag$. 
Additionally, given a Markov chain for the random-cluster model 
with transition matrix $\Pt$, we write $\Pt^\dag$ for the transition 
matrix of the Markov chain for the corresponding 
dual model.

The first result demonstrates the usefulness of the heat-bath 
dynamics for the random-cluster model, 
which is one of the two previously defined local Markov chains 
for this model (see~\eqref{eq:RC-HB-def} for the definition 
of its transition matrix $\thb$). 
It turns out that the spectral gap of $\thb$ equals the 
spectral gap of $\thb^\dag$. 
This is a simple consequence of Lemma~\ref{lemma:RC-dual} 
and is probably known, but we could not find a reference 
for it.

\vspace{3mm}

\begin{lemma} \label{lemma:RC-HB-dual}
Let $\thb$ $($resp.~$\thb^\dag$$)$ be the transition matrix of the 
heat-bath dynamics for the random-cluster model 
on a planar graph $G$ $($resp.~for the dual model\/$)$.
Then
\[
\lambda(\thb) \;=\; \lambda(\thb^\dag).
\]
\end{lemma}
\vspace{2mm}

\begin{proof}
By definition we have for all $A,B\subset E$ that 
\[
\thb(A,B) \;=\; \frac1{\abs{E}}\,\sum_{e\in E}\,
\frac{\mu(B)}{\mu(A\cup e)+\mu(A\setminus e)}\;
\ind(B\setminus e = A\setminus e),
\]
see~\eqref{eq:RC-HB-def}. 
Using Lemma~\ref{lemma:RC-dual} and since 
$(A\cup e)^\dag=A^\dag\setminus e_\dag$ 
(resp. $(A\setminus e)^\dag=A^\dag\cup e_\dag$), 
by~\eqref{eq:dual-conf}, we obtain 
\[\begin{split}
\thb(A,B) \;&=\; \frac1{\abs{E}}\,\sum_{e\in E}\,
	\frac{\mu(B)}{\mu(A\cup e)+\mu(A\setminus e)}\;
	\ind(B\setminus e = A\setminus e) \\
&=\; \frac1{\abs{E}}\,\sum_{e\in E}\,
	\frac{\mu^\dag(B^\dag)}{\mu^\dag\bigl((A\cup e)^\dag\bigr)+
	\mu^\dag\bigl((A\setminus e)^\dag\bigr)}\;
	\ind\bigl( (B\setminus e)^\dag = (A\setminus e)^\dag \bigr) \\
&=\; \frac1{\abs{E^\dag}}\,\sum_{e_\dag\in E^\dag}\,
	\frac{\mu^\dag(B^\dag)}{\mu^\dag\bigl(A^\dag\setminus e_\dag\bigr)+
	\mu^\dag\bigl(A^\dag\cup e_\dag\bigr)}\;
	\ind(B^\dag\cup e_\dag = A^\dag\cup e_\dag) \\
&=\; \thb^\dag(A^\dag,B^\dag).
\end{split}\]
Note for the last equality that $B^\dag\cup e_\dag = A^\dag\cup e_\dag$ 
if and only if $B^\dag\setminus e_\dag = A^\dag\setminus e_\dag$. 
This obviously implies that the matrices $\thb$ and $\thb^\dag$ have 
the same eigenvalues and thus the same spectral gap.
\end{proof}
\vspace{2mm}

We immediately obtain the last ingredient for the proof 
of rapid mixing of Swendsen--Wang and single-bond dynamics 
on the two-dimensional square lattice at all non-critical 
temperatures (see Section~\ref{sec:5_square}).

\vspace{2mm}

\goodbreak

\begin{theorem} \label{th:SW-SB_dual}
Let $\tsw$ $($resp.~$\sb$$)$ be the transition matrix of the 
Swendsen--Wang $($resp.~single-bond\/$)$ dynamics for the 
random-cluster model on a planar graph $G$ with $m$ edges, 
and let $\tsw^\dag$ $($resp.~$\sb^\dag$$)$ be the SW $($resp.~SB$)$ 
dynamics for the dual model.
Then \vspace{1mm}
\begin{align*}
\lambda(\sb) \;&\le\; q \; \lambda(\sb^\dag) \\[-5mm]
\intertext{\vspace*{-4mm} and}
\lambda(\tsw) \;&\le\; 8q\, m \log m\; \lambda(\tsw^\dag).
\end{align*}
\end{theorem}
\vspace{2mm}

\begin{proof}
With Lemmas~\ref{lemma:SB-HB} and~\ref{lemma:RC-HB-dual} 
we get 
\[
\lambda(\sb) \;\le\; \lambda(\thb) 
\;=\; \lambda(\thb^\dag) 
\;\le\; \left(1-p^* \left(1-\frac1q\right)\right)^{-1}\,
				\lambda(\sb^\dag).
\]
The constant on the right hand side is maximal for $p^*=1$. 
This proves the first statement of the theorem. 
For the second we additionally use Theorem~\ref{th:main-SB}. 
Thus, 
\[
\lambda(\tsw) \;\le\; 8 m \log m\; \lambda(\sb) 
\;\le\; 8 q m \log m\; \lambda(\sb^\dag) 
\;\le\; 8 q m \log m\; \lambda(\tsw^\dag),
\vspace{-2mm}
\]
as claimed.
\end{proof}

Prior to the application of this result to the 
two-dimensional square lattice in the next section, 
we present the above-mentioned result on 
rapid mixing of the Swendsen--Wang dynamics 
at sufficiently high and low temperatures if the underlying 
graph is planar and has bounded maximum degree.
This corollary relies on a result of Hayes \cite{Ha}. 

\vspace{1mm}

\begin{corollary} \label{coro:SW-planar}
The Swendsen--Wang dynamics for the random-cluster model
with parameters $p$ and $q$
on a planar, simple and connected graph $G$ with $m$ edges 
and maximum degree $\D\ge6$ satisfies
\[
\lambda(\tsw)\;\ge\;\frac{c(1-\eps)}{m}, \qquad \text{ if }\; 
p \,\le\, \frac{\eps}{3\sqrt{\D-3}},
\vspace{-1mm}
\] 
and \vspace{-2mm}
\[
\hspace{2mm}
\lambda(\tsw)\;\ge\;\frac{c(1-\eps)}{m^2 \log m}, \quad\;\;
\text{ if }\; 
p \,\ge\, 1- \frac{\eps}{q \D^\dag}, \;\;\;\,
\] 
for some $c=c(\D,p,q)>0$ and $\eps>0$, 
where $\D^\dag$ is the maximum degree of a dual graph of~$G$.
\end{corollary}

\begin{proof}
From Theorem~\ref{th:HB_degree} together with the bound on the 
operator norm of the adjacency matrix of planar simple graphs 
from \cite[Cor.~17]{Ha} 
we obtain $\lambda(\hb)\ge (1-\eps)/n$, where $\hb$ is the 
transition matrix of the heat-bath dynamics for the Potts model 
on $G$, and $n$ is the number of vertices, 
if $\beta\le\eps/\sqrt{3(\D-3)}$.
Recall that $p=1-e^{-\beta}$, which implies $\beta\le p/(1-p)$; 
we deduce that the assumption on $p$ yields the desired bound 
on~$\beta$. By connectedness of $G$ we have $m\ge n-1$. Thus,  
Theorem~\ref{th:main-spin} completes the proof 
of the first inequality.

For the second we use 
$\lambda(\tsw)\ge 1/(8q m \log m)\, \lambda(\tsw^\dag)$ from 
Theorem~\ref{th:SW-SB_dual}, where 
$\tsw^\dag$ is the Swendsen--Wang dynamics for the 
dual model, i.e.~the RC model on the dual graph $G^\dag$ with 
parameters $p^*=\frac{q(1-p)}{p+q(1-p)}$ and $q\ge1$. 
Since $p^*\le q(1-p)$, 
it is enough to prove $\lambda(\tsw^\dag)\ge\tilde{c}(1-\eps)/m$ 
if $p^*\le\eps/\D^\dag$. 
But, because of $\beta\le p/(1-p)$ and $m\ge n-1$, this follows from 
Corollary~\ref{coro:SW_degree} and Lemma~\ref{lemma:SW_P-RC}.
\end{proof}
\vspace{1mm}

Note that the maximum degree of the dual graph corresponds to 
the maximum number of edges that are needed to surround a 
single face in a planar representation of the original graph. 
This quantity is sometimes called the \emph{maximum face-degree} of 
$G$ in the literature.

Unfortunately, in the second inequality of 
Corollary~\ref{coro:SW-planar} 
the lower bound on $p$ contains $\D^\dag$ and not its square root. 
The reason for this is that the proof of the bound on the 
principal eigenvalue of a planar graph requires the graph to be  
simple, see \cite[Cor.~17]{Ha}. 
(Note that although Euler's formula is valid also for non-simple 
graphs, one cannot deduce from it that there is a vertex 
of small degree.)
But since the proof of Corollary~\ref{coro:SW-planar} uses 
dual graphs, and we cannot guarantee in general that the 
dual graph is simple, we have to use the stronger assumption.
If we considered only 3-connected graphs (see 
Remark~\ref{remark:3-connected}), then it is known that 
the dual graph is simple. In this case one could improve 
the lower bound on $p$ to $1-\eps/(3q\sqrt{\D^\dag-3})$ 
under the additional assumption $\D^\dag\ge6$.

\vspace{2mm}


\section{Rapid mixing on the square lattice} \label{sec:5_square}

The goal of this section is to present an application 
of Theorem~\ref{th:SW-SB_dual} to the random-cluster model 
on the two-dimensional square lattice. 
We prove that the Swendsen--Wang and single-bond dynamics 
are rapidly mixing for each $q\in\N$ if the parameter 
$p$ satisfies $p\neq\frac{\sqrt{q}}{1+\sqrt{q}}$. 
Translated to the $q$-state Potts model this shows that the 
Swendsen--Wang dynamics is rapidly mixing at all 
non-critical temperatures, i.e.~at all $\beta\neq\beta_c(q)$ 
(see Remark~\ref{remark:critical}).

This is done by a successive application of the results of the 
previous chapters.
Especially, we need Theorem~\ref{th:P-SW_2d_high}, 
Corollary~\ref{coro:P-SW_2d_boundary}, Theorem~\ref{th:main-SB} 
and Theorem~\ref{th:SW-SB_dual}. 
Note that, since we proved only comparison results between 
different Markov chains, the results on rapid mixing 
rely ultimately on the known results on the mixing properties 
of the heat-bath dynamics for the Potts model (see 
Section~\ref{sec:2_known-results}).

First of all, recall the definition of $\Z_L^2=(V_{L,2},E_{L,2})$, 
i.e.~the two-dimensional square lattice of side length $L$, 
from Section~\ref{sec:2_known-results} and of 
$\Z_L^{2\dag}$ from Section~\ref{sec:3_general}.
It is easy to see that $\Z_L^{2\dag}$ is indeed 
the dual graph of $\Z_L^{2}$ (see Figure~\ref{fig:dual}).
Hence, we obtain the following result.

\vspace{1mm}

\begin{theorem} \label{th:SW_square}
Let $\tsw$ be the transition matrix of the 
Swendsen--Wang dynamics for the random-cluster model 
on $\Z^2_L$ with parameters $p$ and $q$. 
Let $m=2L(L-1)=\abs{E_{L,2}}$. 
Then there exist constants $c_p=c_p(q)$, $c'>0$ and $C<\infty$ 
such that 
\begin{itemize}
\item\quad $\displaystyle \lambda(\tsw) \;\ge\; \frac{c_p }{m}$ 
			\qquad\qquad\quad for $p < p_c(q)$, \vspace{1mm}
\item\quad $\displaystyle \lambda(\tsw) \;\ge\; \frac{c_p}{m^2 \log m}$ 
			\qquad for $p > p_c(q)$, \vspace{1mm}
\item\quad $\displaystyle \lambda(\tsw) \;\ge\; c' m^{-C}$ 
			\qquad\quad for $q=2$ and $p = p_c(2)$,
\end{itemize}
where $p_c(q)\,=\,\frac{\sqrt{q}}{1+\sqrt{q}}$.
\end{theorem}
\vspace{1mm}

\begin{proof}
The first and the last inequality follow from 
Theorem~\ref{th:P-SW_2d_high}, Lemma~\ref{lemma:SW_P-RC} and the 
fact that $\abs{V_{L,2}}\le m=2L(L-1)\le2\abs{V_{L,2}}$. 
For the second let $\tsw^\dag$ be the Swendsen--Wang dynamics on 
$\Z_L^{2\dag}$ with parameters $p^*=\frac{q(1-p)}{p+q(1-p)}$ and $q$.
Corollary~\ref{coro:P-SW_2d_boundary} 
(together with Lemma~\ref{lemma:SW_P-RC}) 
implies that $\lambda(\tsw^\dag) \ge c_{p^*}/{m}$ for some 
$c_{p^*}<\infty$, if $p^*<p_c(q)$. 
(Recall that $p_c(q)=1-e^{-\beta_c(q)}$.)
But since $p^*<p_c(q)$ if and only if $p>p_c(q)$, we can 
deduce the desired bound from Theorem~\ref{th:SW-SB_dual}.
\end{proof}
\vspace{1mm}

An immediate consequence is the following.

\begin{theorem} \label{th:SW_square2}
Let $\sw$ be the transition matrix of the 
Swendsen--Wang dynamics for the $q$-state Potts model 
on $\Z^2_L$ at inverse temperature $\beta$. Let $n=L^2$. 
Then there exist constants $c_\beta=c_\beta(q)$, $c'>0$ and $C<\infty$ 
such that 
\begin{itemize}
\item\quad $\displaystyle \lambda(\sw) \;\ge\; \frac{c_\beta}{n}$ 
			\qquad\qquad\quad for $\beta < \beta_c(q)$, \vspace{1mm}
\item\quad $\displaystyle \lambda(\sw) \;\ge\; \frac{c_\beta}{n^2 \log n}$ 
			\qquad\; for $\beta > \beta_c(q)$, \vspace{1mm}
\item\quad $\displaystyle \lambda(\sw) \;\ge\; c' n^{-C}$ 
			\qquad\quad\, for $q=2$ and $\beta = \beta_c(2)$,
\end{itemize}
where $\beta_c(q)\,=\,\log(1+\sqrt{q})$.
\end{theorem}
\vspace{2mm}

\begin{proof}
Apply Lemma~\ref{lemma:SW_P-RC} to Theorem~\ref{th:SW_square} 
and note, again, that 
$n\le\abs{E_{L,2}}\le2n$.
\end{proof}
\vspace{1mm}


The bounds as given in Theorem~\ref{th:SW_square2} are 
certainly not optimal. 
We conjecture $\lambda(\sw)$ to be bounded below by a 
constant for~$\beta\neq\beta_c(q)$.
However, this seems to be the first polynomial bound for the 
Swendsen--Wang dynamics in this regime.

We also obtain the following for the single-bond dynamics. 

\vspace{1mm}

\begin{theorem} \label{th:SB_2d}
Let $\sb$ be the transition matrix of the 
single-bond dynamics for the RC model 
on $\Z^2_L$ with parameters $p$ and $q$. Let $m=2L(L-1)=\abs{E_{L,2}}$. 
Then there exist constants $c_p=c_p(q)$, $c'>0$ and $C<\infty$ such that 
\begin{itemize}
\item\quad $\displaystyle \lambda(\sb) \;\ge\; \frac{c_p}{m^2 \log m}$ 
			\qquad for $p \neq p_c(q)$, \vspace{2mm}
\item\quad $\displaystyle \lambda(\sb) \;\ge\; c' m^{-C}$ 
			\qquad\quad for $q=2$ and $p = p_c(2)$,
\end{itemize}
where $p_c(q)\,=\,\frac{\sqrt{q}}{1+\sqrt{q}}$.
\end{theorem}

\begin{proof}
The result for $p\le p_c(q)$ was already given in 
Corollary~\ref{coro:SB_2d_high}. For the result for 
$p>p_c(q)$ take the bound of Corollary~\ref{coro:P-SW_2d_boundary} 
and apply Lemma~\ref{lemma:SW_P-RC} and Theorem~\ref{th:main-SB}.
\end{proof}

\vspace{2mm}


\section{Graphs of higher genus} \label{sec:5_genus}

In this final section we give a brief description of embeddings 
of graphs into surfaces, i.e.~into two-dimensional 
topological manifolds, and show that it is possible to extend 
the results of Section~\ref{sec:5_result} to this case. 
For convenience, we only consider connected, orientable and 
closed surfaces, which we simply call \emph{surfaces} in what follows.
The plan of this section is to define the \emph{genus} 
of a surface (or a graph) and then, analogously to 
Section~\ref{sec:5_embeddings}, to define representations of 
graphs into surfaces. Then we give a formulation of Euler's formula, 
define dual graphs and show the consequence 
for the random-cluster measure. 

For this, let $\G_1$ and $\G_2$ be two surfaces and define 
their \emph{connected sum} $\G_1\sqcup\G_2$ as the surface 
that is obtained by cutting out a small disc from each of 
the surfaces $\G_1$ and $\G_2$ and gluing them together 
along the boundaries of the resulting holes. 
(See e.g.~Giblin~\cite{Giblin} for a more formal definition.) 
It is known that the surface $\G_1\sqcup\G_2$ does not 
depend (up to homeomorphisms) on the choice of the discs that 
are cut out from the primal surfaces (see \cite[Prop.~2.17]{Giblin}). 
We consider the surfaces $\S_h$, $h\ge0$, which are 
given by $\S_{k+1}:=\S_k\sqcup\S_1$, where 
$\S_0$ is the two-dimensional sphere and $\S_1$ is the torus.
That is, $\S_h$ is the surface that can be obtained from a sphere 
by adding $h$ handles to it. We call $\S_h$ the (orientable) 
\emph{surface of genus $h$}. 
For example, the sphere is a surface of genus $0$ and the torus a 
surface of genus $1$. 
Note that each connected, closed and orientable surface is 
homeomorphic to precisely 
one of the surfaces $\S_h$, $h\ge0$ (see \cite[Thm.~3.1.3]{Mohar}).

Let $G=(V,E)$ be a connected graph. 
In the light of Section~\ref{sec:5_embeddings} we say 
that $G$ has a \emph{representation in $\S_h$}, 
namely $\Gt=(\Vt,\Et)$, 
if there exists a finite set $\Vt\subset\S_h$, 
a bijection $\kappa$ from $V$ to $\Vt$ 
and a set of simple, continuous curves
$\,\Et=\{\g_e:[0,1]\to\S_h:\, e\in E\}$ in $\S_h$ 
with the same three properties as given at the beginning of 
Section~\ref{sec:5_embeddings}. 
Here, we consider planar graphs as graphs 
that admit a representation in the sphere~$\S_0$.

The \emph{genus}\index{genus} 
$\genus(G)$ of the graph $G$ 
is defined to be the smallest integer $h$ such that 
$G$ has a representation in $\S_h$, and we call a 
representation of $G$ in $\S_{\genus(G)}$ a 
\emph{minimum genus representation}. 
For example, the graph $\Zt_L^2$, $L\ge3$, 
i.e.~the two-dimensional torus 
(see~\eqref{eq:def-torus}) has genus~1.

\begin{remark}
Every finite graph can be represented in $\S_h$ 
if $h$ is sufficiently large. For this consider the 
complete graph $K_n$, whose genus satisfies
$\genus(K_n)=\lceil{(n-3)(n-4)}/{12}\rceil$ 
(see \cite{Ringel-Youngs}), 
and observe that 
every graph on $n$ vertices can be represented in 
the same surface than $K_n$ (since they are subgraphs). 
Below we will see that there are constants involved 
in the results that depend exponentially on the 
genus. Therefore, the forthcoming results are only useful in 
cases where 
the genus is at most logarithmic in the size of the graph.
\end{remark}
 
Recall that $F(\Gt)$ denotes the number of faces 
of the representation $\Gt$.
It turns out that there is a version of 
Euler's formula also in this case 
(see \cite[eq.~(3.7)]{Mohar}). 

\begin{lemma}[Euler's formula II]\label{lemma:euler2}
Let $G=(V,E)$ be a connected graph and $\Gt$ be a 
minimum genus 
representation of $G$. Then
\[
2 - 2\genus(G) \,=\, \abs{V} - \abs{E} + F(\Gt).
\]
\end{lemma}
\vspace{1mm}

Following the ideas of Section~\ref{sec:5_embeddings} 
we fix a minimum genus representation $\Gt$ of $G$ and 
fix the representation of the spanning subgraphs 
$G_A=(V,A)$, $A\subset E$, to be $\Gt$ without the 
curves corresponding to $E\setminus A$. 
We write $F(A)$ for the number of faces in this 
representation of $G_A$.
%
Note that a subgraph could be of smaller genus, 
e.g.~$G_\varnothing=(V,\varnothing)$ is obviously planar. 
However, we equip subgraphs of $G$ with a representation 
in $\S_{\genus(G)}$. 
%
We obtain the following corollary, which is the basis 
of the remaining results.

\begin{corollary}\label{coro:euler2}
Let $G$ be a connected graph and 
$G_A$, $A\subset E$, be a spanning 
subgraph of $G$. Then
\[
2-2\genus(G) 
\;\le\; \abs{V} - \abs{A} + F(A) - c(A) + 1 
\;\le\; 2.
\]
\end{corollary}
\vspace{1mm}

\begin{proof} 
We prove that the term in the middle cannot increase if we replace 
$A$ by $A\cup e$, i.e.
$\abs{V} - \abs{A\cup e} + F(A\cup e) - c(A\cup e) + 1
\,\le\, \abs{V} - \abs{A} + F(A) - c(A) + 1$.
This proves the claim since 
the first and second inequalities of the statement become  
equalities for $A=E$ and $A=\varnothing$, respectively. 
Thus, it is enough to prove 
$F(A\cup e)-F(A)\le c(A\cup e)-c(A)+1$ for all $A\subset E$. 
The right hand side of this inequality equals 1 if 
the endpoints of $e$ are connected in $(V,A)$, 
i.e.~$e^{(1)}\xconn{A}e^{(2)}$, and 0 otherwise. 
Consequently, we want to verify 
$F(A\cup e)-F(A)\le \ind(e^{(1)}\xconn{A}e^{(2)})$, that is, 
if adding $e$ to $A$ induces a new face, then 
the endpoints of $e$ are connected in $G_A$.
But this is clearly true, because if a new face is induced by 
$e$, then $e$ is contained in a ``cycle'' in $G_{A\cup e}$, 
which implies that the endpoints are already connected in $G_A$.
\end{proof}

Given a graph $G=(V,E)$ and a corresponding minimum genus 
representation $\Gt=(\Vt,\Et)$ we define the \emph{dual graph} 
$G^\dag=(V^\dag,E^\dag)$ 
just as in Section~\ref{sec:5_embeddings}. 
That is, $G^\dag$ is the unique graph that has a representation 
in $\S_{\genus(G)}$ with a single vertex in every face of $\Gt$ 
and, for every curve $\g_e\in\Et$, a curve that intersects only 
$\g_e$ and connects the vertices (or the vertex) on both sides of 
$\g_e$. We write $e_\dag$ for the edge that corresponds to this 
``dual'' curve.

A \emph{dual configuration} $A^\dag\subset E^\dag$ of $A\subset E$ 
is given by
\[
e_\dag\in A^\dag \quad\iff\quad e\notin A
\]
(cf.~\eqref{eq:dual-conf}).
It is again easy to convince oneself that every face of the 
representation of $\Gt_A$ (from $\Gt$ adopted) contains 
a connected component of $G^\dag_{A^\dag}$, 
i.e.~$F(A)=c(A^\dag)$. Thus, 
\begin{equation} \label{eq:dual-components2}
c(A^\dag) + \abs{A^\dag} - 1
\;\le\; c(A) - \abs{V} + \abs{E}
\;\le\; c(A^\dag) + \abs{A^\dag} - 1 +2\genus(G).
\end{equation}
This implies an analogous result to Lemma~\ref{lemma:RC-dual} 
and, eventually, the following theorem. 
Recall that the \emph{dual model} to the random-cluster model 
on $G$ with parameters $p$ and $q$ is the RC model on $G^\dag$ 
with parameters $p^*$ and $q$, where $p^*=\frac{q(1-p)}{p+q(1-p)}$. 

\vspace{1mm}

\begin{theorem} \label{th:SW-SB_dual2}
Let $\thb$, $\sb$ and $\tsw$ 
$($resp.~$\thb^\dag$, $\sb^\dag$ and $\tsw^\dag$$)$ be the 
transition matrices of the 
heat-bath, single-bond and Swendsen--Wang dynamics for the 
random-cluster model on a graph $G$ $($resp.~for the dual model\/$)$.
Then
\[
\lambda(\thb) \;\le\; q^{8\genus(G)}\,\lambda(\thb^\dag) \qquad\qquad\quad
\]
and hence,
\begin{align*}
\lambda(\sb) \;&\le\; q^{8\genus(G)+1} \; \lambda(\sb^\dag) \\[-3mm]
\intertext{\vspace*{-2mm}and}
\lambda(\tsw) \;&\le\; 8q^{8\genus(G)+1}\, m \log m\; \lambda(\tsw^\dag).
\end{align*}
\end{theorem}
\vspace{3mm}

\begin{proof}
We only prove the first inequality. 
The other two follow in the same way as Theorem~\ref{th:SW-SB_dual} 
follows from Lemma~\ref{lemma:RC-HB-dual}. 
An easy computation (similar to the proof of 
Lemma~\ref{lemma:RC-HB-dual}) shows
\[
q^{-2\genus(G)}\,\mu_{p,q}^G(A) \;\le\; 
\mu_{p^*,q}^{G^\dag}(A^\dag) 
\;\le\; q^{2\genus(G)}\,\mu_{p,q}^G(A)
\]
for all $A\subset E$. 
By the definition of $\thb$ (see \eqref{eq:RC-HB-def}), this implies 
\vspace{1mm}
\[
\thb(A,B) \;\le\; q^{4\genus(G)}\,\thb^\dag(A^\dag,B^\dag).
\vspace{1mm}
\]
But these two inequalities yield 
$\lambda(\thb) \le q^{8\genus(G)}\lambda(\thb^\dag)$ (see 
Lemma~\ref{lemma:prelim_comparison}).
\end{proof}
\vspace{3mm}

Theorem~\ref{th:SW-SB_dual2} can be used in the same way 
as Theorem~\ref{th:SW-SB_dual} to prove results 
on the mixing properties of the Markov chains involved 
in specific cases. 
For example, a result, similar to Corollary~\ref{coro:SW-planar}, 
on graphs of bounded degree 
can be proven by using Theorem~\ref{th:HB_degree} and 
a bound on the principal eigenvalue of graphs with bounded 
genus (see e.g.~\cite{Dvorak-Mohar}).
Another application is that lower bounds on the spectral gap 
of the Markov chains for the Potts model 
on the two-dimensional square lattice with periodic boundary 
condition (i.e.~on the two-dimensional torus $\Zt_L^2$) 
at high temperatures 
can be translated to lower bounds at low temperatures 
(cf.~Section~\ref{sec:5_square}).
We omit the details.

\newpage
\thispagestyle{plain}

\def\arx#1#2{\href{http://arxiv.org/abs/#1}{#2}}
\def\xbibitem#1#2, \doi#3.{\bibitem{#1}\leavevmode\llap{\href{http://dx.doi.org/#3}{\phantom{[\theenumiv]}}\kern\labelsep}\ignorespaces#2.}
\optrue


\newpage
\thispagestyle{plain}
\printindex


\end{document}